\documentclass{scrartcl}
\usepackage{amsmath}
\usepackage{amsfonts}
\usepackage{amssymb}
\usepackage{dsfont}
\usepackage[T1]{fontenc}
\usepackage[latin1]{inputenc}
\usepackage{epsfig}
\usepackage{graphicx}

\usepackage[boxed, lined]{algorithm2e}
\usepackage{float}
\usepackage{comment}
\usepackage{mathtools}

\usepackage{enumitem}

\usepackage{tikz}
\usetikzlibrary{positioning}
\usetikzlibrary{shapes,snakes}
\usetikzlibrary{shapes,arrows}

\usepackage{bm}

\newcommand{\subalign}[1]{%
   \vcenter{%
     \Let@ \restore@math@cr \default@tag
     \baselineskip\fontdimen10 \scriptfont\tw@
     \advance\baselineskip\fontdimen12 \scriptfont\tw@
     \lineskip\thr@@\fontdimen8 \scriptfont\thr@@
     \lineskiplimit\lineskip
     \ialign{\hfil$\m@th\scriptstyle##$&$\m@th\scriptstyle{}##$\hfil\crcr
       #1\crcr
     }%
   }%
}

\newcommand{\bx}{x}

\newcommand{\bw}{\mathbf{w}}

\newcommand{\D}{{\mathcal{D}}}

\newcommand{\W}{{\mathcal{W}}}

\newcommand{\I}{{\mathcal{I}}}

\newcommand{\Nu}{{\mathcal{N}}}

\newcommand{\M}{{\mathcal{M}}}

\newcommand{\N}{\mathbb{N}}

\newcommand{\R}{\mathbb{R}}

\newcommand{\Rd}{\mathbb{R}^d}

\newcommand{\beq}{\begin{eqnarray*}}

\newcommand{\eeq}{\end{eqnarray*}}

\newcommand{\beqm}{\begin{eqnarray}}

\newcommand{\eeqm}{\end{eqnarray}}

\newtheorem{theorem}{Theorem}

\newtheorem{lemma}{Lemma}
\newtheorem{definition}{Definition}

\DeclareOldFontCommand{\bf}{\normalfont\bfseries}{\mathbf}
\DeclareOldFontCommand{\it}{\normalfont\itshape}{\mathit}

\newcommand{\EXP}{{\mathbf E}}
\newcommand{\PROB}{{\mathbf P}}

\renewcommand{\P}{{\cal P}}
\renewcommand{\bf}{\normalfont \bfseries}
\renewcommand{\it}{\normalfont \itshape}

\allowdisplaybreaks

\begin{document}
\renewcommand{\thefootnote}{\fnsymbol{footnote}}
\newcommand{\F}{{\cal F}}
\newcommand{\Sp}{{\cal S}}
\newcommand{\G}{{\cal G}}
\newcommand{\HH}{{\cal H}}

\newcounter{constcounter}

\makeatletter
\newcommand{\const}[1][\relax]{%
  \ifmeasuring@
    c_{?}%
  \else
    \ifx#1\relax
      \stepcounter{constcounter}%
      c_{\theconstcounter}%
    \else
      \@ifundefined{const@#1}{%
        \stepcounter{constcounter}%
        \expandafter\xdef\csname const@#1\endcsname{\theconstcounter}%
        c_{\csname const@#1\endcsname}%
      }{%
        c_{\csname const@#1\endcsname}%
      }%
    \fi
  \fi
}
\makeatother


\begin{center}

  {\LARGE \bf
    Learning of deep neural network regression estimates using
    gradient descent with pruning
  }
\footnote{
  Running title: {\it Learning of deep neural networks}}
\vspace{0.5cm}

Michael Kohler$^{1}$, 
Adam Krzy\.zak$^{2,}$\footnote{Corresponding author. Tel:
  +1-514-848-2424 ext. 3007, Fax:+1-514-848-2830}
and Vincent Molinero R\"omer$^{1}$
\\

{\it $^1$
Fachbereich Mathematik, Technische Universit\"at Darmstadt,
Schlossgartenstr. 7, 64289 Darmstadt, Germany,
email: kohler@mathematik.tu-darmstadt.de, roemer@mathematik.tu-darmstadt.de}

{\it $^2$ Department of Computer Science and Software Engineering, Concordia University, 1455 De Maisonneuve Blvd. West, Montreal, Quebec, Canada H3G 1M8, email: krzyzak@cs.concordia.ca}

\end{center}
\vspace{0.5cm}

\begin{center}
August 26, 2026
\end{center}
\vspace{0.5cm}

\noindent
    {\bf Abstract}\\
    Estimation of a regression function from independent and identically
    distributed data is considered. The $L_2$ error with integration with
respect to the design variable is used as the error criterion.
An initially randomly pruned fully connected deep neural
network with logistic squasher as activation function
is fitted to the data via gradient descent, using a data-dependent
choice of non-constant stepsizes during gradient descent.
It is shown that this network achieves (up to a logarithmic
factor) the optimal minimax rate of convergence in case
that the regression function is $(p,C)$--smooth. Here the estimate
is able to circumvent the curse of dimensionality provided the
predictors are concentrated in the neighborhood of a low dimensional manifold.
The finite sample size performance of the estimate
is illustrated by applying it to the simulated data.
    
    \vspace*{0.2cm}


\vspace*{0.2cm}


\section{Introduction}
\label{se1}
\subsection{Scope of this paper}
\label{se1sub1}
Deep learning is currently changing the world. It has recently achieved tremendous success in
various applications like image classification 
(cf., e.g., Krizhevsky, Sutskever and Hinton  (2012)),
in language recognition (cf., e.g.,  Kim (2014)),
in machine translation (cf., e.g., Wu et al. (2016)),
in game playing (cf., e.g., Silver et al. (2017))
or in simulation of human conversation (cf., e.g., Minaee et al. (2025)).
Nowadays it is also applied in mathematics in various fields outside of statistics,
e.g., Kohler et al. (2026) applies deep learning in efficiently solving
partial differential equations occurring in modelling of 
the diffusion MRIs.

Motivated by this huge success of deep learning in applications
there has been also an increasing interest in theoretical
foundations of deep learning.
Often deep learning is studied theoretically in some equivalent model
like the neural tangent kernel (cf.,  e.g., Jacot, Gabriel and Hongler (2018))
or the mean-field theory (cf., e.g. Mei, Montanari and Nguyen (2018)).
In this article we use a different approach: we study deep learning in the context of nonparametric
regression directly without using any asymptotic description.
Here we analyze simultaneously the approximation error,
the generalization error and the optimization error of the estimate
and use our results to propose a special topology for the deep
neural network, a special initialization scheme for its initial weights,
a special choice for the non-constant step sizes during gradient
descent, and a special choice for the number of gradient descent steps,
and derive results concerning the rate of convergence of this
estimate.

\subsection{Neural networks}
\label{se1sub2}
In order to introduce our theoretically motivated definition of 
deep neural network regression estimates let us first describe 
fully connected deep neural networks. These networks depend on an
activation function $\sigma: \R \rightarrow \R$, e.g., the so-called logistic squasher
\begin{equation}
  \label{inteq1}
  \sigma(x) = \frac{1}{1+e^{-x}}
\end{equation}
or the famous ReLU activation function
\[
\sigma(x) = \max\{ x,0\}.
\]
A fully connected deep neural network with $L$ layers and $r_n$ neurons
per layer is  defined as a function

\begin{figure}[h!]
\centering
\pagestyle{empty}
\def\layersep{3.5cm}
\begin{tikzpicture}[shorten >=1pt,->,draw=black, node distance=\layersep, scale=0.8, transform shape]
\centering
    \tikzstyle{every pin edge}=[<-,shorten <=1pt]
    \tikzstyle{neuron}=[circle,fill=black!25,minimum size=10pt,inner sep=0pt]
    \tikzstyle{input neuron}=[neuron, fill=
    black];
    \tikzstyle{output neuron}=[neuron, fill=
    black];
    \tikzstyle{hidden neuron}=[neuron, fill=black!50
    ];
    \tikzstyle{annot} = [
    text centered
    ]

    \foreach \name / \y in {1,...,4}
        \node[input neuron, pin=left:\footnotesize{$x^{(\y)}$}, 
        xshift=1cm 
        ] (I-\name) at (0,-1.4*\y) {};

    \foreach \name / \y in {1,...,7}
        \path[
        yshift=0.5cm
        ]
            node[hidden neuron] (H-\name) at (\layersep,-\y cm) {};
            
    \foreach \name / \y in {1,...,7}
        \path[yshift=1cm]
            node[hidden neuron, right of = H-\name] (H2-\name) {};

    \foreach \name / \y in {1,...,7}
        \path[yshift=1cm]
            node[hidden neuron, right of = H2-\name] (H3-\name) {};

    \node[output neuron,pin={[pin edge={->}]right:\footnotesize{$f(\bold{x})$}}, right of=H3-4, xshift=-0.8cm] (O) {};

    \foreach \source in {1,...,4}
        \foreach \dest in {1,...,7}
            \path (I-\source) edge (H-\dest);
            
    \foreach \source in {1,...,7}
        \foreach \dest in {1,...,7}
            \path (H-\source) edge (H2-\dest);

    \foreach \source in {1,...,7}
        \foreach \dest in {1,...,7}
            \path (H2-\source) edge (H3-\dest);

    \foreach \source in {1,...,7}
        \path (H3-\source) edge (O);

    \node[annot,above of =H-1, xshift=3.9cm, node distance=1cm] (hl) {\footnotesize{Hidden layers}};
    \node[annot,above of=I-1, node distance = 1cm] {\footnotesize{Input}};
    \node[annot,above of=O, yshift=1.5cm, node distance = 1cm] {\footnotesize{Output}};
      \node[draw, below of= H-7, yshift=3.0cm, xshift=-0.4cm, rounded corners, minimum size=1cm] (r) {\footnotesize{$\sigma(\bold{w}^T\bold{x}+w_0)$}};
\end{tikzpicture}
\caption{A fully connected network of the class $\mathcal{F}(3,7)$}
\label{fig1}
\end{figure}
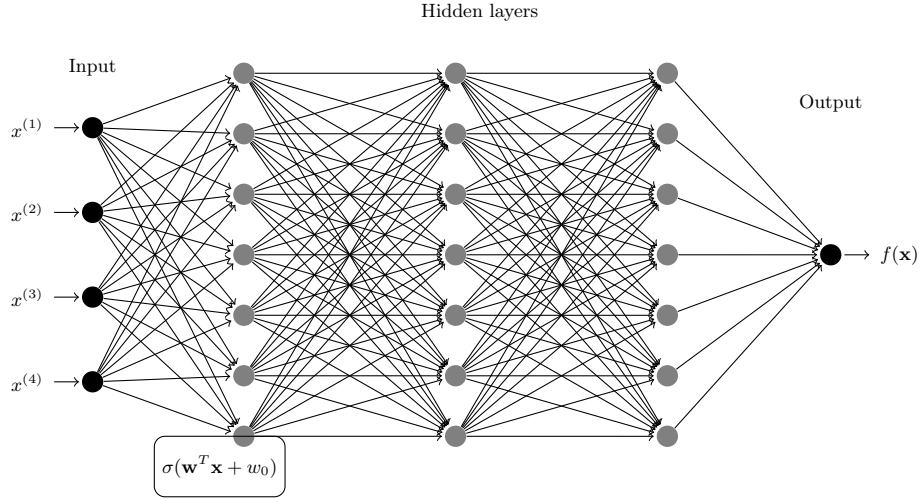

\begin{equation}
  \label{se2eq1}
f_{\bw}(x)
=
\sum_{k=1}^{r_n} w_{1,k}^{(L)} \cdot f_{\bw,k}^{(L)}(x),  
  \end{equation}
where
\begin{equation}
  \label{se2eq2}
f_{\bw,i}^{(l)}(x)
=
\sigma \left(
\sum_{k=1}^{r_n} w_{i,k}^{(l-1)} \cdot f_{\bw,k}^{(l-1)}(x) + w_{i,0}^{(l-1)}
\right)
\end{equation}
for $ i \in \{1, \dots, r_n\}$ and $l \in \{2, \dots, L\}$, and  
\begin{equation}
  \label{se2eq3}
f_{\bw,i}^{(1)}(x)
=
\sigma \left(
\sum_{k=1}^d w_{i,k}^{(0)} \cdot x^{(k)} + w_{i,0}^{(0)}
\right)  
\end{equation}
for $i \in \{1, \dots, r_n\}$.
Here
$w_{i,k}^{(l)}$
is the weight between neuron $k$ in layer $l$ and neuron $i$ in layer $l+1$
$(l \in \{1, \dots, L-1 \})$, and
\[
\bw=
\left(
w_{k,i}^{(l)}
\right)_{k,i,l} \in \R^{r_n \cdot (d+1) + (L-1) \cdot r_n \cdot (r_n+1) + r_n}
\]
is the vector of the weights of the network. We denote the class all such
networks by $\F(L,r_n)$, where $L$ is the depth of the network and $r_n$
is the width of the network. A schematic representation of such a network
in case $L=3$ and $r_n=7$ can be found in Figure \ref{fig1}.

In the sequel we want to use such networks to define  nonparametric
regression estimates.

\subsection{Nonparametric regression}
\label{se1sub3}
In nonparametric regression $(X,Y)$, $(X_1,Y_1)$, \dots are independent
and identically distributed (i.i.d.) $\R^d \times \R$ valued random variables with
$\EXP Y^2 < \infty$, and given the i.i.d. sample
\[
\D_n = \{ (X_1,Y_1), \dots, (X_n,Y_n) \}
\]
of $(X,Y)$ the task is to construct an estimate
\[
m_n(\cdot) = m_n(\cdot, \D_n): \R^d \rightarrow \R
\]
of the regression function
\[
m: \R^d \rightarrow \R, \quad m(x)= \EXP\{Y |X=x\}
\]
such that the so--called $L_2$ error
\[
\int | m_n(x)-m(x)|^2 \PROB_X (dx)
\]
is ''small''. For an introduction to nonparametric regression and
for a motivation of the $L_2$ error criterion
we refer to Gy\"orfi et al. (2002).

\subsection{Neural network regression estimates}
\label{se1sub4}
To construct a neural network regression estimate,
one chooses a class of neural networks
$f_\bw$ (e.g., the class $\F(L,r_n)$ of fully connected neural networks with depth $L$
and width $r_n$), and minimizes the
empirical $L_2$ risk
\begin{equation}
  \label{inteq5}
  F_n(\bw) = \frac{1}{n} \sum_{i=1}^n |f_\bw (X_i) - Y_i|^2
\end{equation}
with respect to the weight vector $\bw$ of the neural network.
It is well--known that the least squares neural network regression estimates,
where one minimizes (\ref{inteq5}) over suitably defined spaces
of deep neural network, achieve nice rate of convergence
results (cf., e.g., Bauer and Kohler (2019), Schmidt-Hieber (2020),
Kohler and Langer (2021) or Jiao et al. (2023), and the literature
cited therein).
Unfortunately, in applications these least squares estimates are not applicable:
Since (\ref{inteq5}) is a nonlinear and nonconvex function
of $\bw$, it is not possible to minimize (\ref{inteq5}) exactly.
In practice, gradient descent 
(or one of its variants like stochastic gradient descent)
is used to minimize it approximately.
Here one starts with a random initialization $\bw^{(0)}$ of the weight
vector, and then computes
\begin{equation}
  \label{inteq6}
  \bw^{(t+1)} = \bw^{(t)} - \lambda_t \cdot \nabla_\bw F_n(\bw)
  \quad \mbox{for } t=0, \dots, t_n-1,
\end{equation}
i.e., one performs $t_n$ gradient descent steps with step sizes
$\lambda_0, \dots, \lambda_{t_n-1} \geq 0$.

In order to define this estimate, one has to decide which
activation function one wants to use, what kind of topology
of the network one wants to use, how the initialization of the weights
is done, and how many gradient descent steps with what kind of step sizes
one wants to perform. All of these decisions influence the performance
of the estimate in applications, and it is not clear what the right
choices are.

In this article we use statistical theory to provide answers to
these questions. Next we quickly describe the results we get.

\subsection{Main results}
\label{se1sub5}
We will use the logistic squasher (\ref{inteq1}) as the activation
function. As far as the topology of the network is concerned we start with a fully
connected network which we initially randomly prune by randomly choosing for
each hidden neurons of the layers $2$, \dots, $L$  $r$
neurons from the previous layer, use only the connections
to the randomly selected neurons and ignore all the rest. See Figure \ref{fig2}
for an illustration of a pruning of the network from Figure \ref{fig1}.

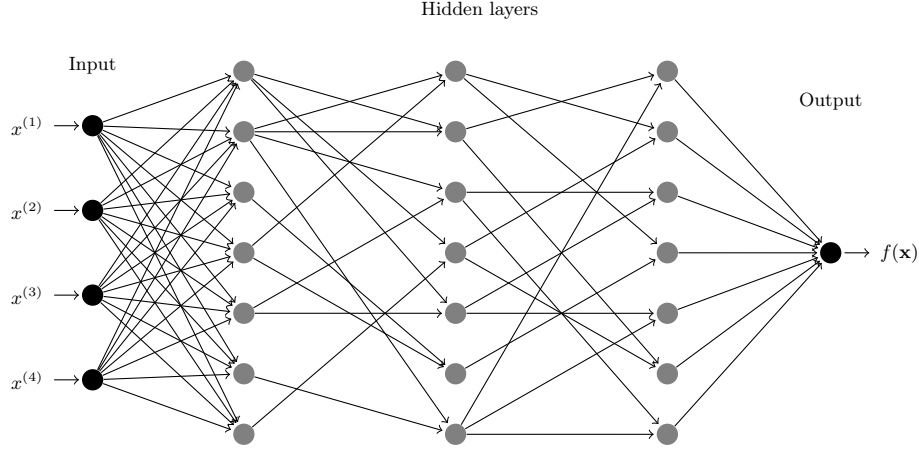
\begin{figure}[h!]
\centering
\pagestyle{empty}
\def\layersep{3.5cm}
\begin{tikzpicture}[shorten >=1pt,->,draw=black, node distance=\layersep, scale=0.8, transform shape]
\centering
    \tikzstyle{every pin edge}=[<-,shorten <=1pt]
    \tikzstyle{neuron}=[circle,fill=black!25,minimum size=10pt,inner sep=0pt]
    \tikzstyle{input neuron}=[neuron, fill=
    black];
    \tikzstyle{output neuron}=[neuron, fill=
    black];
    \tikzstyle{hidden neuron}=[neuron, fill=black!50
    ];
    \tikzstyle{annot} = [
    text centered
    ]

    \foreach \name / \y in {1,...,4}
        \node[input neuron, pin=left:\footnotesize{$x^{(\y)}$}, 
        xshift=1cm 
        ] (I-\name) at (0,-1.4*\y) {};

    \foreach \name / \y in {1,...,7}
        \path[
        yshift=0.5cm
        ]
            node[hidden neuron] (H-\name) at (\layersep,-\y cm) {};
            
    \foreach \name / \y in {1,...,7}
        \path[yshift=1cm]
            node[hidden neuron, right of = H-\name] (H2-\name) {};

    \foreach \name / \y in {1,...,7}
        \path[yshift=1cm]
            node[hidden neuron, right of = H2-\name] (H3-\name) {};

    \node[output neuron,pin={[pin edge={->}]right:\footnotesize{$f(\bold{x})$}}, right of=H3-4, xshift=-0.8cm] (O) {};

    \foreach \source in {1,...,4}
        \foreach \dest in {1,...,7}
            \path (I-\source) edge (H-\dest);
            
            \path (H-2) edge (H2-1);
        \path (H-4) edge (H2-1);
        \path (H-1) edge (H2-2);
        \path (H-2) edge (H2-2);
        \path (H-2) edge (H2-3);
        \path (H-5) edge (H2-3);
        \path (H-1) edge (H2-4);
        \path (H-7) edge (H2-4);
        \path (H-1) edge (H2-5);
        \path (H-5) edge (H2-5);
        \path (H-3) edge (H2-6);
        \path (H-4) edge (H2-6);
        \path (H-2) edge (H2-7);
        \path (H-6) edge (H2-7);

        \path (H2-2) edge (H3-1);
        \path (H2-7) edge (H3-1);
        \path (H2-1) edge (H3-2);
        \path (H2-4) edge (H3-2);
        \path (H2-3) edge (H3-3);
        \path (H2-5) edge (H3-3);
        \path (H2-1) edge (H3-4);
        \path (H2-6) edge (H3-4);
        \path (H2-5) edge (H3-5);
        \path (H2-7) edge (H3-5);
        \path (H2-2) edge (H3-6);
        \path (H2-4) edge (H3-6);
        \path (H2-3) edge (H3-7);
        \path (H2-7) edge (H3-7);

    \foreach \source in {1,...,7}
        \path (H3-\source) edge (O);

    \node[annot,above of =H-1, xshift=3.9cm, node distance=1cm] (hl) {\footnotesize{Hidden layers}};
    \node[annot,above of=I-1, node distance = 1cm] {\footnotesize{Input}};
   \node[annot,above of=O, yshift=1.5cm, node distance = 1cm] {\footnotesize{Output}};
\end{tikzpicture}
\caption{The above network after random pruning with $r=2$}
\label{fig2}
\end{figure}

We will use uniform distributions in order to initialize our weights.
Our choice is motivated by results concerning the approximation
properties and the generalization error of neural networks,
and we  choose the weights between the input and the first hidden
layer uniformly distributed on same large interval $[-A,A]$, and the
weights between all the hidden layers uniformly distributed
on some moderate interval $[-B,B]$. All output weights will be initially
set to zero.

Our basic idea for choosing the step sizes is to choose the maximal $\lambda_t$
 such that
\begin{equation}
\label{inteq2}
F_n(\bw^{(t+1)}) \leq F_n(\bw^{(t)})
-
\frac{\lambda_t}{2} \cdot \left\|
\nabla_\bw F_n( \bw^{(t)} )
\right\|^2
\quad (t=0,1, \dots, t_n-1)
\end{equation}
holds, which leads to a version of the well-known
Armijo line search procedure from optimization
for the choice of the step sizes during gradient descent
(cf., Armijo (1966)).
It turns out that in order to be able to show a
nice error bound for the estimate it suffices to impose 
the following additional conditions on the non-constant step sizes
$\lambda_0, \dots, \lambda_{t_n-1} \geq 0$ and by assuming that the number of
gradient descent steps $t_n$ satisfy:
\[
1 \leq \sum_{s=0}^{t_n-1}  \lambda_s \leq 2
\quad \mbox{and} \quad
\max_{s=0, \dots, t_n -1} \lambda_s \leq \frac{1}{n}.
\]
An algorithm which computes  data-dependent non-constant
step sizes satisfying the above conditions is presented below as Algorithm \ref{alg2}.
Here we stop checking condition (\ref{inteq2}) as soon as the step size is
below some threshold, since in this case the step size will fulfill condition (7) automatically with
high probability as we will see in the proof of our main results.

\begin{algorithm}
  \caption{Pseudo code for the data-dependent choice of non-constant
    step sizes for gradient descent. \label{alg2}}
  \KwData{$(x_1,y_1)$, \dots, $(x_n,y_n)$ \\
   $L$, $r_n$, $r$, $A$, $B$ \\
    $\lambda_{min}=\frac{1}{ (\log n)^{2L+3} \cdot r_n^{3/2} }$, $\const[lambda]=1$
  }
  \Begin{
      $t=0$ \\
    $\bw^{(0)}=InitializeWeights(L,r_n,r,A,B)$\\
      \Repeat{
$\sum_{s=0}^{t-1} \lambda_s \geq \const[lambda]$
      }{
      $\lambda_t=\frac{2}{n}$ \\       

        \Repeat{$F_n (\bw^{(t+1)}) \leq  F_n (\bw^{(t)}) - \frac{\lambda_t}{2} \cdot
\left\|
\nabla_\bw F_n (\bw^{(t)})
\right\|^2$ or $\lambda_t \leq \lambda_{min}$}
       {
         $\lambda_t =\lambda_t /2$\\
$\bw^{(t+1)}
=
\bw^{(t)}
-
\lambda_t \cdot \nabla_{\bw} F_n(\bw^{(t)})$\\
         } 

       $t=t+1$} 
      $t_n=t$ \\
          \KwResult{$f_{\bw^{(t_n)}}$}
} 
\end{algorithm}

Our estimate leads to simple rules for otherwise difficult decisions
in the applications of deep neural network regression estimates
concerning the topology of the network, the
initialization and the choices of the parameters of the gradient
descent.
The topology of our network is a simple fully connected neural network
which is initially randomly pruned. After pruning it is still
drastically overparameterized in the sense that it has many 
more parameters than data points,
but due to our special structure of the
pruning, the special initialization of the weights, and the special
choice of the step sizes and the number of gradient descent steps
our estimate does not overfit the data.

We are able to derive bounds on the expected $L_2$ error of this estimate
which imply that it achieves asymptotically up to some logarithmic factor
the optimal minimax rate of convergence for estimation of smooth
functions. In addition, we show that the estimate is able to
circumvent the so--called curse of dimensionality in case
 the predictor is concentrated in the neighborhood of
a low dimensional manifold.
Furthermore, we illustrate
the finite sample size behaviour of our estimate (and in particular our
data-dependent choice
of the non-constant step sizes) by applying it to
simulated data from an univariate
regression problem.

The results in this paper are based on earlier theoretical results
developed by the authors and co-authors (cf., Subsection \ref{se1sub6}). The main new
contributions in this article are:

\begin{enumerate}
\item
We extend the previous theory such that it covers more general
network topologies (initially randomly pruned fully connected deep neural networks
instead of linear combinations of small parallel deep neural networks) and non-constant
step sizes during gradient descent.

\item We prove rates of convergence which are optimal up to logarithmic factor
 instead of optimal up to small constants in the exponent of the rates
 (like, e.g., in Kohler (2026)).
  
 \item We extend previous known results from manifold data to 
 the case that the predictor is contained in a neighborhood of a
 manifold  which is important since it occurs often in applications that
 the data is close to manifolds, but not necessarily contained in
 a manifold (cf., e.g., Carlsson (2009)).
 
\end{enumerate}

\subsection{Discussion of related results}
\label{se1sub6}

In this paper we use a version of the Armijo line-search
during gradient descent. Concerning results of this method
in connection with optimization of general functions we refer
to Vaswani and Babanezhod (2026) and the literture cited therein.
Applications of it towards deep learning are described in
Kenneweg, Kenneweg and Hammer (2024).

We also use pruning, which is a standard method in learning
of deep neural networks. A review of various pruning methods
can be found in Cheng, Zhang and Shi (2024).
In our paper we apply pruning before the training of the
networks, so we basically use it to define a (random) class
of sparse deep neural networks. This helps us to simultaneously 
control the approximation error, the generalization error and
the optimization error of our estimate.

Deep learning has been studied from several complementary mathematical perspectives. 
Approximation-theoretic results quantify the expressive power of the neural networks. 
Examples include the classical universal approximation theorem for sigmoidal networks
(cf. Cybenko (1989) and Hornik, Stinchcombe and White (1989))
and approximation results for deep neural networks
like Yarotsky (2017, 2018), Yarotsky and Zhevnerchuck (2019), Langer (2020)
or Lu et al. (2020).

Often least squares neural network estimates have been analyzed.
Early works for shallow (one hidden layer) neural networks include
Barron (1994). The rate of convergence results
for deep neural networks which include dimension
reduction in case that the regression function is a composition of
low-dimensional smooth functions have been shown by Kohler and
Krzy\.zak (2017), Bauer and Kohler (2019), Schmidt-Hieber (2020)
and Kohler and Langer (2021). Kohler, Langer and Reif (2023)
demonstrated that least squares neural network estimates can achieve
a dimension reduction in case that the predictor is contained
in a low dimensional manifold. In practice, data is often just
close, but not exactly contained in a low dimensional manifold.
In this case Jiao et al. (2023) showed that the deep neural network least squares
estimates can also achieve a low dimensional rate of convergence 
and can hence circumvent the curse of dimensionality. Our second main result
can be considered as a generalization of this result to the case
that the estimate is learned by gradient descent (as it is always
done in applications since there least squares estimates are
computationally infeasible).

Another line
 of work studies the optimization of deep neural networks. E.g., Du et al. (2019) and Allen-Zhu, Li and Song (2019) consider gradient-based algorithms for the minimization of the highly non-convex empirical loss of an overparameterized network.  Statistical analyses, in turn, investigate the population prediction error and its dependence on the sample size and the regularity of the regression function. These questions are related and need to be studied simultaneously, since a small approximation error does not imply that an optimization algorithm finds a good network, and a small empirical training error does not by itself imply good generalization. In fact, Kohler and Krzy\.zak (2021) show that overparameterized networks can attain an almost minimal empirical risk without achieving an optimal nonparametric convergence rate.

One popular approach 
to study optimization and generalization simultaneously
is the neural tangent kernel (NTK) introduced by Jacot, Gabriel and Hongler (2018) for shallow neural networks. They showed that, in the infinite-width limit, the training dynamics of the network output are governed by a deterministic kernel that remains constant during gradient flow. Building on this result, Lee et al. (2019) showed that the parameters and outputs of a sufficiently wide network trained by gradient descent remain close to those of the first-order Taylor expansion around the initialization. Their result therefore provides an explicit linear-model interpretation of the NTK regime and extends the analysis to discrete gradient descent.

The NTK approach has been used to obtain the optimization and generalization results for overparameterized networks. Arora et al. (2019), for instance, study both optimization and generalization through a kernel associated with the network. In the setting of nonparametric regression, Nitanda and Suzuki (2021) derive minimax optimal convergence rates for averaged stochastic gradient descent in the NTK regime. Nguyen and M\"ucke (2024) obtain refined bounds for early-stopped gradient descent and quantify the network width required to approximate the limiting kernel estimator.

Another asymptotic description is given by mean-field theory. For a wide shallow network, the empirical distribution of the neuron parameters is represented by a probability measure. Under an appropriate mean-field scaling, its evolution is described by a nonlinear partial differential equation or a Wasserstein gradient flow, see Mei, Montanari and Nguyen (2018), Chizat and Bach (2018), and Sirignano and Spiliopoulos (2020). In contrast to the classical NTK regime, the parameter distribution can evolve nonlinearly, so the mean-field model can describe changes of the features during training. Mei, Misiakiewicz and Montanari (2019) additionally derive finite-particle approximation bounds and show how a kernel regime can be recovered as a special limit of the mean-field description. Mean-field frameworks have also been extended to multilayer networks by Nguyen and Pham (2023).

Both of those approaches provide powerful descriptions of limiting training dynamics often restricted to shallow neural networks.
Our approach focuses on deep neural networks with  a specific network architecture and logistic activation function.
Here we study our network estimator directly without asymptotic descriptions, which allows for a statistically guided construction of  concrete regression estimates.

The works most closely related to the present paper analyze neural network regression estimators trained by gradient descent directly in a statistical regression setting. Braun et al.\ (2024) derive a convergence rate of order $n^{-1/2}$, up to a logarithmic factor, for shallow neural network estimators with logistic activation and a suitable random initialization, under an appropriate decay condition on the Fourier transform of the regression function. Drews and Kohler (2024) establish universal consistency for an overparameterized deep neural network estimator. Quantitative bounds on the expected $L_2$ error and corresponding convergence rates are obtained by Kohler and Krzy\.zak (2025a, 2025b), Drews and Kohler (2026), and Kohler (2026). The analyses in these works jointly control approximation, optimization, and generalization for concrete finite-width neural network estimators trained by gradient descent. Our article relies in particular
on new techniques which have been introduced in Kohler (2026): The approximation of the
gradient descent by the  gradient descent applied to linear Taylor polynomials
of the deep neural network and approximation theorems for deep neural
network with bounded weights, which enable the use of metric entropy
bounds to control the generalization error of the estimates. Kohler, Krzy\.zak and
Molinero R\"omer (2026) derive the rate of convergence results for deep
neural network regression estimates learned by gradient descent in case
of dependent data. There it is also shown that these estimates achieve
a dimension reduction in case that the predictor is exactly contained
on a low dimensional manifold. All rates of convergence results shown in the above
papers are only optimal up to some arbitrarily small additive term
in the exponent of the rate of convergence, which is in contrast to
the results for the least squares neural network
estimates where it is usually shown that the rates of convergence are
optimal up to some logarithmic factor.

The present paper continues this line of research. Its main distinguishing features are a network topology
obtained by random pruning of an initially fully connected architecture, the data-dependent selection of potentially non constant step sizes  and the dimension reduction in case of 
predictors contained in a neighborhood of a manifold.
Moreover, for (p,C)-smooth regression functions
(where, roughly speaking, the regression function is $p$--times continuously
differentiable, cf., Definition \ref{se3de3} below), the resulting estimator achieves the minimax rate $n^{-2p/(2p+d)}$ up to a logarithmic factor.

\subsection{Notation}
\label{se1sub7}
  The sets of natural numbers, real numbers and nonnegative real numbers
  are denoted by $\N$, $\R$ and $\R_+$, respectively.
  For $z \in \R$, we denote
the smallest integer greater than or equal to $z$ by
$\lceil z \rceil$. 
The Euclidean norm of $x \in \Rd$
is denoted by $\|x\|$, and the scalar product of $x,y \in \R^d$ is denoted
by $<x,y>$.  And
$\|x\|_\infty$
denotes the supremum norm of $x \in \R^d$.
For $f:\R^d \rightarrow \R$
\[
\|f\|_\infty = \sup_{x \in \R^d} |f(x)|
\]
is its supremum norm.

If $X$ is an $\R^d$ valued random variable, we denote its support by
\[
supp(\PROB_X)
=
\left\{
x \in \R^d \, : \,
\PROB\{
\| X - x\| > \epsilon \} >0 \mbox{ for all } \epsilon>0
\right\}.
\]

Let $\F$ be a set of functions $f:\R^d \rightarrow \R$.
A finite collection $f_1, \dots, f_N:\Rd \rightarrow \R$
  is called an $L_p$ $\varepsilon$--covering of $\F$ on $x_1^n$
  if for all $f \in \F$
  \[
  \min_{1 \leq j\leq N}
  \left(
  \frac{1}{n} \sum_{k=1}^n |f(x_k)-f_j(x_k)|^p
  \right)^{1/p} \leq \varepsilon
  \]
  holds.
  The $L_p$ $\varepsilon$--covering number of $\F$ on $x_1^n$
  is the  size $N$ of the smallest $L_p$ $\varepsilon$--covering
  of $\F$ on $x_1^n$ and is denoted by $\Nu_p(\varepsilon,\F,x_1^n)$.

For $z \in \R$ and $\beta \geq 0$ we define
$T_\beta z = \max\{-\beta, \min\{\beta,z\}\}$. If $f:\R^d \rightarrow
\R$
is a function  then we set
$
(T_{\beta} f)(x)=
T_{\beta} \left( f(x) \right)$, and if $\F$ is a class of functions
$f:\R^d \rightarrow \R$ we set
$T_\beta \F = \{ T_\beta f \, : \, f \in \F \}$.

\subsection{Outline}
\label{se1sub8}
Section \ref{se2} contains the definition of the estimate.
The main results are presented in Section \ref{se3}.
In Section \ref{se4} we illustrate the
finite sample size performance of the estimate by applying it to
simulated data. The proofs are given in Section \ref{se5}.

\section{Definition of the estimate}
\label{se2}
Throughout this paper
we use the so--called logistic squasher
(\ref{inteq1})
as our activation function. For $L, r_n \in \N$
let $\F(L,r_n)$ be the class of all multilayer deep neural networks
with $L$ layers and $r_n$ neurons per layer and logistic squasher as
activation function defined by (\ref{se2eq1})--(\ref{se2eq3}).
In the sequel we describe
how our estimate fits such a network to the data.

We start by initializing the weight vector $\bw$ of this network. To do this,
we set all output weights to zero, i.e., we set
\begin{equation}
  \label{se2eq4}
  (\bw^{(0)})^{(L)}_{1,i} = 0 \quad \mbox{for } i=1, \dots, r_n.
  \end{equation}
Next we initialize all inner weights of the network by choosing
their values uniformly distributed from some interval $[-B,B]$ (where
$B>0$ is a parameter of the estimate), i.e., we choose
\begin{equation}
  \label{se2eq5}
  (\bw^{(0)})^{(l)}_{i,j} \sim Unif([-B,B])
   \quad \mbox{for } l=1, \dots, L-1,
  \end{equation}
and we initialize all input weights of the network by choosing
their values uniformly distributed from some interval $[-A_n,A_n]$
(where
$A_n>0$ is another parameter of the estimate),
i.e., we choose
\begin{equation}
  \label{se2eq6}
  (\bw^{(0)})^{(0)}_{i,j} \sim Unif([-A_n,A_n]).
  \end{equation}
Here all random values are chosen from the independent distributions.

After that we randomly prune the network depending on some $r \in \N$,
where we assume that $r_n \geq r$.
To do this, we choose
for each neuron of the network (except for the input neurons) $r$
neurons from the previous level, and set all weights connecting neurons
from the previous level and the considered neuron to zero except the
$r$ neurons selected. Here the $r$ neurons are randomly selected from
the uniform distribution on the $r_n$ neurons of the previous level,
and all considered distributions are independent and also independent
from the distribution used to define the initial values of the weights.
And for the rest of the computation of the estimate we keep all those weights, 
which we initially have set to zero, at the value zero and do not change them.

Next we train the resulting network using gradient descent.
Here we perform $t_n \in \N$ gradient descent steps with step sizes
$\lambda_0, \dots, \lambda_{t_{n}-1} >0$. To define this formally, let $\bw$ and $\bw^{(0)}$
be vectors of weights which contain only those weights which have
not been set to zero during the pruning step. Let $f_\bw$ be the network
corresponding to the weight vector $\bw$, and denote its empirical
$L_2$ risk by
\begin{equation}
  \label{se2eq7}
F_n(\bw) = \frac{1}{n} \sum_{i=1}^n | f_\bw (X_i) - Y_i |^2.
\end{equation}
Then we compute
\begin{equation}
  \label{se2eq8}
  \bw^{(t+1)} = \bw^{(t)} - \lambda_t \cdot \nabla_\bw F_n(\bw^{(t)})
  \quad
  \mbox{for } t=0, \dots, t_n-1,
\end{equation}
where the step sizes
and the number of gradient descent steps
are chosen as in Algorithm \ref{alg2}.
We define our estimate as truncated network with weight vector
$\bw^{(t_n)}$.
More precisely, we 
set $\beta_n = \const[c1] \cdot \log n$ and
\begin{equation}
  \label{se2eq10}
  m_n(x) = T_{\beta_n} f_{\bw^{(t_n)}}(x) \quad (x \in \R^d),
  \end{equation}
where $T_\beta z = \max\{ \min\{ z, \beta\}, - \beta\}$ for
$z \in \R$ and $\beta \geq 0$.

\section{Main results}
\label{se3}
In our main results we will impose the following smoothness
assumption on the regression function.

\begin{definition}
\label{se3de3}
  Let $p=q+s$ for some $q \in \N_0$ and $0< s \leq 1$.
A {\bf function} $m:\R^d \rightarrow \R$ is called
{\bf $(p,C)$-smooth}, if for every $\alpha=(\alpha_1, \dots, \alpha_d) \in
\N_0^d$
with $\sum_{j=1}^d \alpha_j = q$ the partial derivative
$\frac{
\partial^q m
}{
\partial x_1^{\alpha_1}
\dots
\partial x_d^{\alpha_d}
}$
exists and satisfies
\[
\left|
\frac{
\partial^q m
}{
\partial x_1^{\alpha_1}
\dots
\partial x_d^{\alpha_d}
}
(x)
-
\frac{
\partial^q m
}{
\partial x_1^{\alpha_1}
\dots
\partial x_d^{\alpha_d}
}
(z)
\right|
\leq
C
\cdot
\| x-z \|^s
\]
for all $x,z \in \R^d$, where $\Vert\cdot\Vert$ denotes the Euclidean norm.
\end{definition}

\noindent
Let $p,C, \const[c2], \alpha>0$. We will impose the following
assumptions on the data to which we apply our estimate.
\begin{itemize}
\item[(A1)]
  $(X,Y)$, $(X_1,Y_1)$, \dots, $(X_n,Y_n)$ are independent
  and identically distributed $\R^d \times \R$--valued random vectors.
  \vspace*{0.1cm}
  
\item[(A2)]
  $\PROB\{ X \in [-\alpha,\alpha]^d\}=1$.
  \vspace*{0.1cm}
  
\item[(A3)]
  $\EXP \{ e^{\const[c2] \cdot Y^2} \} < \infty$.
  \vspace*{0.1cm}
  
\item[(A4)]
  The regression function $m(x)=\EXP\{Y|X=x\}$ is $(p,C)$--smooth.
  \end{itemize}

Under the above assumptions we can prove the following bound on the
expected $L_2$ error of our estimate.

\begin{theorem}
  \label{th3}
    Let $p,C, \const[c2], \alpha>0$ be arbitrary and assume that $(A1)$--$(A4)$ hold.
  Choose $\const[c1]>0$ such that $\const[c1] \cdot \const[c2] \geq 2$, let
  $\const[c3]$, $\const[c_4]$ and $\const[c6]$ be sufficiently large and choose
  $r \in \N$ such that
  \[
r \geq 4 \cdot (p+d)^2,
  \]
and let
  $L, r_n, A_n, B$ be such that
   \begin{equation}
    \label{th1eq1}
L \geq \lceil \log_2(p+d) \rceil, B=\const[c3], A_n= \const[c_4] \cdot(\log n) \cdot n^{\frac{1}{2p+d}},
   \end{equation}
   \begin{equation}
    \label{th1eq2}
\frac{r_n}{n^{\const[c5]}} \rightarrow 0 \quad (n \rightarrow \infty)
   \end{equation}
   for some $\const[c5]>0$
   and
   \begin{equation}
    \label{th1eq3}
    \frac{r_n}{n^{8 \cdot L \cdot
        r^{L-1} \cdot (d+1) + 10}} \rightarrow \infty \quad (n \rightarrow \infty)
   \end{equation}
   hold.
     Set $\beta_n = \const[c1] \cdot \log n$ and define the estimate as in
   Section \ref{se2} with step sizes and number of gradient descent steps chosen
   data-dependent as in Algorithm \ref{alg2}. Then for any $n \in \N$
   it holds
   \[
   \EXP \int |m_n(x)-m(x)|^2 \PROB_X (dx)
   \leq
   \const[c6] \cdot (\log n)^{4 \cdot d \cdot L+3} \cdot n^{- \frac{2p}{2p+d}}.
   \]
   
  \end{theorem}

It follows from Stone (1982) that the above rate of convergence is optimal
up to some logarithmic factor. For $d$ large this error bound 
will converge rather slowly to zero.
Since the above rate of convergence is optimal up to a logarithmic
factor, the only way to circumvent the so--called curse of dimensionality is to
impose additional assumptions on the distribution of $(X,Y)$.

Next we show that our estimate achieves a better rate of convergence
in  the case of manifold data, where we
assume that the predictor $X$ takes on values on some manifold in $\R^d$
of dimension
$d^*<d$, which we define as follows:

\begin{definition}
  \label{se3de1} 
  Let $\M \subseteq \Rd$ be compact and let $d^* \in \{1, \dots, d\}$.

  \noindent
      {\bf a)} We say that $U_1,\dots,U_s$ is
      an {\em open covering  of $\M$}, if $U_1,\dots,U_s \subset \Rd$
      are open (with respect to the Euclidean topology on $\Rd$)
      and satisfy
      \[
\M \subseteq \bigcup_{l=1}^s U_l.
\]

  \noindent
      {\bf b)} We say that
\[
\psi_1, \dots, \psi_s: [0,1]^{d^*}\rightarrow \Rd
\]
are {\em bi-Lipschitz functions}, if there exists $0 < C_{\psi,1} 
\leq C_{\psi,2}
< \infty$ such that
\begin{equation}
\label{se3de1eq1}
C_{\psi,1}  \cdot \|\bx_1-\bx_2\|
\leq
\| \psi_l(\bx_1)-\psi_l(\bx_2) \|
\leq
C_{\psi,2}  \cdot \|\bx_1-\bx_2\|
\end{equation}
holds for any $\bx_1,\bx_2 \in  [0,1]^{d^*}$ and any
$l \in \{1, \dots, s\}$.

        \noindent
            {\bf c)}
            We 
say that $\M$ is a {\em $d^*$-dimensional Lipschitz-manifold} if 
there exist 
bi-Lipschitz functions $\psi_i : [0,1]^{d^*} \to \R^d$ $(i \in \{
  1,\dots,s \})$, 
and an open covering $U_1,\dots,U_s$ of $\M$ such that 
\[
\psi_l\left( (0,1)^{d^*} \right) = \M \cap U_l
\]
holds for all $l \in \{1, \dots, s\}$.
Here we call $\psi_1, \dots, \psi_s$ the {\em parametrizations}
of the manifold.
\end{definition}

In applications, often the data is just close to a manifold, but not contained in it
(cf., e.g., Carlsson (2009)), so it is more reasonable to assume that the data is contained in
a neighborhood of a manifold. To describe this, we use the next definition.

\begin{definition}
\label{se3de2}
Let $\delta \geq 0$. We say that the support of $X$ is contained in an approximate $d^*$-dimensional manifold
of order $\delta$, if there a exists a $d^*$-dimensional Lipschitz manifold $\M$ such that the support
of $X$ is contained in the $\delta$-neighborhood of $\M$, i.e., such that
\[
supp(\PROB_X) \subseteq
\M_\delta:=
\left\{
x \in \R^d \, : \,
\inf\{
\| x-y\| \, : \, y \in \M
\}
\leq \delta
\right\}.
\]
\end{definition}

Our main result in this setting is the following theorem.

\begin{theorem}
  \label{th4} 
  Let $p,C, \const[c2], \alpha>0$ be arbitrary.
  Choose $\const[c1]>0$ such that $\const[c1] \cdot \const[c2] \geq 2$, let
  $\const[c3]$ and $\const[c_4]$ be sufficiently large and choose
  $L, r \in \N, A_n,B \in \R$ such that
  \[
r \geq 4 \cdot (p+d)^2,
L \geq \lceil \log_2(p+d) \rceil, B=\const[c3], A_n= \const[c_4] \cdot(\log n) \cdot n^{\frac{1}{2p+d^*}},
  \]
and let
$r_n$ be such that
(\ref{th1eq2}) and (\ref{th1eq3})
hold.
   Set $\beta_n = \const[c1] \cdot \log n$ and define the estimate as in
   Section \ref{se2} with step sizes and number of gradient descent steps chosen
   data-dependent as in Algorithm \ref{alg2}. Then for any $n \in \N$
   and any distribution of $(X,Y)$ where (A1), (A3) and (A4) hold and where
  the support of $X$ is contained in an approximate $d^*$-dimensional manifold of order
  \[
  \delta_n = n^{-\frac{1}{2p+d^*}} \cdot (\log n)^{-4L-1}
  \]
   we have
   \[
   \EXP \int |m_n(x)-m(x)|^2 \PROB_X (dx)
   \leq
   \const[c6] \cdot (\log n)^{4 \cdot d^* \cdot L+d+3 } \cdot n^{- \frac{2p}{2p+d^*}}.
   \] 
  \end{theorem}

\noindent
{\bf Remark 1.} Our estimate uses a polynomial number of neurons (cf., (\ref{th1eq2}))
and a data-dependent number of gradient descent steps determined by Algorithm \ref{alg2}.
Since the step size is in each step bounded from below by $\lambda_{min}/2$, the
definition of $\lambda_{min}$ and (\ref{th1eq2}) imply that also the number of gradient
descent steps is bounded from above by some polynomial in the sample size.

\section{Application to simulated data}
\label{se4}

In order to illustrate the finite sample size performance
of our newly proposed estimate, we apply it to the following univariate
regression problem introduced in Chapter 1 in Gy\"orfi et al. (2002).

We define the distribution of an $\R \times \R$ valued random vector $(X,Y)$
by
\[
Y=m(X) + \sigma \cdot s(X) \cdot N,
\]
where $X$ is standard normal restricted to $[-1,1]$,
\[
m(x)= \begin{cases}
(x+2)^2/2 & \mbox{if } -1 \leq x < -0.5, \\
x/2 + 0.875 & \mbox{if } -0.5 \leq x < 0, \\
-5 \cdot (x-0.2)^2 + 1.075 & \mbox{if } 0 \leq x < 0.5, \\
x+0.125 & \mbox{if }  0.5 \leq x \leq 1, 
\end{cases}
\]
\[
s(x)=0.2-0.1 \cdot \cos ( 2 \cdot \pi \cdot x),
\]
$\sigma \in \{0.5,1\}$
and 
$N$ is a standard normally distributed random variable
independent of $X$. 

We construct an independent sample $\D_n$ of sample size $n \in \{100,200,400\} $
from this distribution, and apply our estimate 
defined as in Algorithm \ref{alg2} to estimate the regression
function $m$ from this sample.

In our first simulation we analyze the stopping criteria in the loop
in Algorithm \ref{alg2} where we stop whenever
\[
\sum_{s=0}^{t-1} \lambda_s \geq \const[lambda],
\]
and we consider $\const[lambda]>0$ as a parameter of the estimate. We apply
our estimate,
which we have implemented in R,
for various values of this parameter to $50$
independent data sets with $\sigma=1$ and sample size $n=100$,
and compute numerically the $L_2$ errors of the estimates. The medians and the IQRs 
of the $L_2$ errors are reported in Table \ref{se4tab1}.

\begin{table}
\begin{center}
\begin{tabular}{|l|l|l|}
\hline
Value of $\const[lambda]$ & median $L_2$ error (IQR)  \\
\hline
0.1 &  0.0095 (0.0030) \\
0.2 & 0.0068 (0.0029) \\
0.3 & 0.0069 (0.0037) \\
0.4 &   0.0074 (0.0038) \\
0.5 & 0.0065 (0.0024) \\
0.6 &  0.0077 (0.0037) \\
0.7 & 0.0080 (0.0029) \\
0.8 &  0.0080 (0.0040) \\
0.9 & 0.0089 (0.0041) \\
1.0 & 0.0086 (0.0029)\\
\hline
\end{tabular}
\end{center}
\caption{Median $L_2$ errors (and IQRs)
  in $50$  simulations with $\sigma=1$,
  $n=100$, $r_n=2000$,                           $L=4$, $r=8$,
  $A=1000$ and $B=20$. \label{se4tab1}}
\end{table}

We see that $\const[lambda]$ acts like a smoothing parameter
for the estimate: if it is too small or too large the error
gets larger than if it maintains a moderate value. Since
we do not know what the optimal value for this constant is
in an application, we consider it together with the value of
$A$ as smoothing parameters (for $B$ we will use a moderately
large fixed value in the sequel), and will choose both values
by splitting of the sample.

We apply our estimate (NNprun)
defined in Algorithm \ref{alg2}, which we have implemented in $R$,
with parameters
$r_n=2000$, $L=4$, $r=8$ and $B=20$.
We split the data into a learning sample of size $n_l = 0.8 * n$
and a testing sample of size $n_t =0.2 * n$, compute the estimate
for $A \in \{50, 100, 200, 500\}$ and all numbers $t$ of
gradient descent steps with $\sum_{s=0}^{t-1} \lambda_s \leq 1$,
and choose the value of $A$ and the number of gradient descent steps
by minimizing the empirical $L_2$ risk on the testing data.

We compare our estimate with feedforward neural network estimates
NN1, NN3 and NN5
with $1$, $3$ and $5$ hidden layers, respectively.
Each layer has a common width $r_n$ and we use the logistic squasher as activation function as well as a linear output layer with a bias term.
We again split the data in a learning sample of size $n_l=0.8*n$ and a testing sample of size $n_t=0.2*n$.
Here the number $r_n \in \{10, 25, 50, 100, 200\}$ of hidden neurons
and the number $t_n \in \{500, 1000, 1500, 2000, 5000\}$ of gradient descent steps are jointly selected to minimize the empirical $L_2$ risk on the testing data.
All network weights are initialized using the Glorot uniform initialization, whereas all bias terms are initialized at zero. 
The estimates are implemented in Python using the package
{\it tensorflow} and trained by gradient descent using full-batch
Adam with a base learning rate of 0.001.

Furthermore, we compare the estimate with a
smoothing spline estimate $smooth-spline$  as implemented in R by the procedure
{\it Tps()} from the library {\it fields}. Here the smoothing parameter
of this estimate is chosen data dependent by generalized cross validation
as implemented in {\it Tps()}.

As before we apply the estimates to $50$ independent samples
with $n \in \{100,200,400\}$ and $\sigma \in \{0.5,1\}$, compute for each
estimate numerically its $L_2$ errors on these $50$ samples and report the medians
and the IQRs in Table \ref{se4tab5}.

\begin{table}[htbp]                                  
\centering
\makebox[\textwidth][c]{                                     
\begin{tabular}{|c|c|c|c|c|c|c|}
\hline
\textit{$\sigma$} & \multicolumn{3}{|c|}{$0.5$} & \multicolumn{3}{|c|}{$1$}\\
\hline                                         
\textit{sample size} & \multicolumn{1}{|c|}{$n=100$} & \multicolumn{1}{|c|}{$n=200$}& \multicolumn{1}{|c|}{$n=400$} & \multicolumn{1}{|c|}{$n=100$} & \multicolumn{1}{|c|}{$n=200$} & \multicolumn{1}{|c|}{$n=400$}\\ 
\hline                                         
\textit{approach} & median  & median  & median  & median & median  & median \\
                  & (IQR) & (IQR) & (IQR) & (IQR) & (IQR) & (IQR) \\                  
\hline                                                                             
NNprun & 0.0040  & 0.0025  & 0.0018  & 0.0086   & 0.0056  & 0.0035  \\
 & (0.0022)  & (0.0006)  & (0.0004)  & (0.0046)    & (0.0022)  & (0.0012)  \\
\hline                                                                             
NN1 & 0.0194  & 0.0181  & 0.0179  & 0.0199   & 0.0193  & 0.0178  \\
 & (0.0019)  & (0.0022)  & (0.0015)  & (0.0023)    & (0.0014)  & (0.0032)  \\
\hline                                                                             
NN3 & 0.0045  & 0.0024  & 0.0019  & 0.0136   & 0.0067  & 0.0037  \\
 & (0.0042)  & (0.0014)  & (0.0008)  & (0.0114)    & (0.0061)  & (0.0030)  \\
\hline                                                                             
NN5 & 0.0044  & 0.0025  & 0.0014  & 0.0146   & 0.0075  & 0.0039  \\
 & (0.0027)  & (0.0017)  & (0.0005)  & (0.0093)    & (0.0079)  & (0.0018)  \\
\hline                                                                             
smooth- & 0.0035  & 0.0023  & 0.0016  & 0.0076   & 0.0053  & 0.0033  \\
spline & (0.0016)  & (0.0079)  & (0.0003)  & (0.0047)    & (0.0027)  & (0.0012)  \\
\hline                                                                          \end{tabular}
}                                  
\caption{Median and IQR of the $L_2$ errors of the estimates in the univariate regression problem}                
\label{se4tab5}                            
\end{table}

As we can see from Table \ref{se4tab5}, our estimate leads in most cases to
results which are either better than or rather close to the standard
fully connected neural networks NN1, NN3 and NN5, and for large values
of $n$ it achieves error bounds similar to the smoothing spline estimate,
which can be considered as gold standard for univariate regression
estimation problems. This shows that (at least in this example) our completely
theoretically motivated deep neural network estimate with data-dependent
non-constant step sizes during gradient achieves a reasonable performance.

\section{Proofs}
\label{se5}

\subsection{Neural network approximation}
\label{se5sub1}

We will use the following approximation result from Kohler (2026)
to show that our randomly initialized weights lead
with high probability to networks with
good approximation properties if we properly modify the weights
in the output layer.

\begin{lemma}
  \label{le1}
Let $d \in \N$,  $p=q+s$ where
$s \in (0,1]$ and $q \in \N_0$, $C>0$,
$\bar{\alpha} \geq 1$
and
$A_n, B_n, \gamma_n^* \geq 1$.
Let $\sigma$ be the logistic squasher.
For $L,r,K \in \N$
let $\F$ be the set of all networks $f_{\bw}$ defined by
\begin{equation}
f_\bw(x) = \sum_{j=1}^{r} w_{1,1,j}^{(L)} \cdot f_{j,1}^{(L)}(x)
\end{equation}
for some $w_{1,1,1}^{(L)}, \dots, w_{1,1,r}^{(L)} \in \mathbb{R}$, where
$f_{j,1}^{(L)}=f_{\bw,j,1}^{(L)}$ are recursively defined by
\begin{equation}
  \label{le1eq2}
f_{k,i}^{(l)}(x) = 
f_{\bw,k,i}^{(l)}(x) = 
\sigma\left(\sum_{j=1}^{r} w_{k,i,j}^{(l-1)}\cdot f_{k,j}^{(l-1)}(x) + w_{k,i,0}^{(l-1)} \right)
\end{equation}
for some $w_{k,i,0}^{(l-1)}, \dots, w_{k,i, r}^{(l-1)} \in \mathbb{R}$
$(l=2, \dots, L)$
and
\begin{equation}
  \label{le1eq3}
f_{k,i}^{(1)}(x) = 
f_{\bw,k,i}^{(1)}(x) = 
\sigma \left(\sum_{j=1}^d w_{k,i,j}^{(0)}\cdot x^{(j)} + w_{k,i,0}^{(0)} \right)
\end{equation}
for some $w_{k,i,0}^{(0)}, \dots, w_{k,i,d}^{(0)} \in \mathbb{R}$, where
the weight vector satisfies
\[
|w_{k,i,j}^{(0)}| \leq A_n, \quad
|w_{k,i,j}^{(l)}| \leq B_n \quad \mbox{and} \quad
|w_{k,i,j}^{(L)}| \leq \gamma_n^*
\]
for all $l \in \{1, \dots, L-1\}$ and all $k,i,j$, and set
\[
\HH = \left\{ \sum_{k=1}^{K^{d}} f_k \quad : \quad f_k \in \F \quad (k=1, \dots, K^{d})
\right\}.
\]
Let $L,r \in \N$ with
\[
L \geq \lceil \log_2(q+d) \rceil
\quad
\mbox{and}
\quad
r \geq 4 \cdot (p+d) \cdot (q+d),
\]
and set
\[
 A_n =\bar{\alpha} \cdot K \cdot \log K, \quad B_n=\const
\quad \mbox{and} \quad 
\gamma_n^*=\const \cdot K^{q+d}.
\]
Assume $K \geq \const[c1th3]$ for $\const[c1th3]>0$ sufficiently large.
Then there exists for any $(p,C)$--smooth $f:\R^d \rightarrow \R$
a neural network $h \in \HH$ such that
\[
\sup_{x \in [-\bar{\alpha},\bar{\alpha})^d} |f(x)-h(x)|
\leq
\frac{\const}{K^{p}}.
\]
\end{lemma}

\noindent
    {\bf Proof.} See Theorem 3 in Kohler (2026).
    \hfill $\Box$

    In the sequel we show that our initial network contains with
    high probability subnetworks close to the subnetworks occurring
    in the network in Lemma \ref{le1}. Our main result with this respect
    is the following result.

    \begin{lemma}
      \label{le2}
      Let $p,C, \alpha >0$ and assume that $(A4)$ holds.
      Choose the parameters of the estimate as in Theorem \ref{th3}.
      Let $\bw^{(0)}$ be the randomly initialized and pruned
      weight vector of the network. 
      Then outside of an event whose probability is bounded from above by
      \[
n^{\const} \cdot \exp \left( -  (\log n)^2 \right) 
      \]
      there exists (random) $w_1^*, \dots, w_{r_n}^* \in \R$  such that
      \[
      \sup_{x \in [-\alpha, \alpha]^d}
      \left|
\sum_{k=1}^{r_n} w_{k}^* \cdot f^{(L)}_{\bw^{(0)},k }(x) - m(x)
  \right|
  \leq \const  \cdot n^{-\frac{p}{2p+d}}
  \]
  and
  \[
\sum_{k=1}^{r_n} | w_{k}^*|^2 \leq \frac{\const \cdot n^{4L \cdot r^{L-1} \cdot (d+1)+4}}{r_n}.
  \]
      \end{lemma}

    In the proof we will need the following auxiliary result.

    \begin{lemma}
      \label{le3} 
      Let $\alpha , B\geq 1$, and $L,r \in \N$,
      let $\sigma$ be the logistic squasher and define $f_{\bw,1}^{(L)}$
      by (\ref{se2eq2}) and (\ref{se2eq3}). Assume that $\bw$ and $\tilde{\bw}$ satisfy
      \begin{equation}
        \label{le3eq1}
        |\bw_{i,j}^{(l)}| \leq B \quad (l=1, \dots, L-1)
      \end{equation}
      and
      \begin{equation}
        \label{le3eq2}
        \max_{i=1, \dots, r_n}
        | \{ 1 \leq j \leq r_n \, : \, w_{i,j}^{(l)} \neq 0 \quad \mbox{or}
        \quad \tilde{w}_{i,j}^{(l)} \neq 0 \}|
        \leq r
         \quad (l=1, \dots, L-1).
      \end{equation}
      Then we have for any $x \in [-\alpha,\alpha]^d$
      \[
      | f_{\bw,1}^{(L)}(x) - f_{\tilde{\bw},1}^{(L)}(x)|
      \leq
      (d+1) \cdot \alpha \cdot (2r+1)^{L-1} \cdot B^{L-1} \cdot
      \| \bw - \tilde{\bw}\|_\infty.
      \]
      \end{lemma}

    \noindent
        {\bf Proof.} The logistic squasher is Lipschitz continuous with
        Lipschitz constant $1$, which implies
        \begin{eqnarray*}
          | f_{\bw,i}^{(1)}(x) - f_{\tilde{\bw},i}^{(1)}(x)|
          &\leq&
          | \sum_{j=1}^d (w_{i,j}^{(0)}-\tilde{w}_{i,j}^{(0)}) \cdot x^{(j)}
          + w_{i,0}^{(0)}-\tilde{w}_{i,0}^{(0)}|
          \\
          &\leq&
          (d+1) \cdot \alpha \cdot \| \bw - \tilde{\bw}\|_\infty
          \end{eqnarray*}
    and for $l \in \{2, \dots, L\}$
        \begin{eqnarray*}
          | f_{\bw,i}^{(l)}(x) - f_{\tilde{\bw},i}^{(l)}(x)|
          &\leq&
          |
          \sum_{j=1}^{r_n} w_{i,j}^{(l-1)}\cdot f_{\bw,j}^{(l-1)}(x) + w_{i,0}^{(l-1)}
          -
          \sum_{j=1}^{r_n} \tilde{w}_{i,j}^{(l-1)}\cdot f_{\tilde{\bw},j}^{(l-1)}(x) -
          \tilde{w}_{i,0}^{(l-1)}
          |
          \\
          & \leq &
          |
          \sum_{j=1}^{r_n} w_{i,j}^{(l-1)}\cdot (f_{\bw,j}^{(l-1)}(x)-  f_{\tilde{\bw},j}^{(l-1)}(x))|
          +
          | \sum_{j=1}^{r_n} (w_{i,j}^{(l-1)}-\tilde{w}_{i,j}^{(l-1)})\cdot f_{\tilde{\bw},j}^{(l-1)}(x)|
          \\
          &&
          +
          | w_{i,0}^{(l-1)}- \tilde{w}_{i,0}^{(l-1)}|
          \\
          & = &
          |
          \sum_{j=1, \dots, r_n: w_{i,j}^{(l-1)} \neq 0} w_{i,j}^{(l-1)}\cdot (f_{\bw,j}^{(l-1)}(x)-  f_{\tilde{\bw},j}^{(l-1)}(x))|
          \\
          &&
          +
          | \sum_{j=1, \dots, r_n: w_{i,j}^{(l-1)}-\tilde{w}_{i,j}^{(l-1)} \neq 0} (w_{i,j}^{(l-1)}-\tilde{w}_{i,j}^{(l-1)})\cdot f_{\tilde{\bw},k,j}^{(l-1)}(x)|
          \\
          &&
          +
          | w_{i,0}^{(l-1)}- \tilde{w}_{i,0}^{(l-1)}|
          \\
          & \leq & r \cdot B \cdot \max_{j=1, \dots, r_n} |f_{\bw,j}^{(l-1)}(x)-  f_{\tilde{\bw},j}^{(l-1)}(x)|
          +
          (r+1) \cdot \| \bw - \tilde{\bw}\|_\infty
        \end{eqnarray*}
        where the last inequality follows from (\ref{le3eq1}) and (\ref{le3eq2}).
        Using induction we conclude for any $l \in \{1, \dots, L\}$
        \[
        | f_{\bw,i}^{(l)}(x) - f_{\tilde{\bw},i}^{(l)}(x)|
        \leq
        (d+1) \cdot \alpha \cdot (2r+1)^{l-1} \cdot B^{l-1} \cdot \| \bw - \tilde{\bw}\|_\infty,
        \]
        which implies the assertion. \hfill $\Box$

        \noindent
            {\bf Proof of Lemma \ref{le2}.}
            W.l.o.g. we assume that $n$ is so large that $A_n \geq B$ holds.
            Set
            $\tilde{K}_n = \lceil n^{\frac{1}{2p+d}} \rceil$ and  
            \[
            I_n =\lceil \frac{r_n}{n^{4L \cdot r^{L-1} \cdot (d+1) +1}} \rceil.
            \]
            In the {\it first step of the proof} we show that there exists
            a neural network
            \begin{equation}
              \label{ple2eq1}
            f_{\bw}(x) = \sum_{j=1}^{I_n \cdot r \cdot \tilde{K}_n^d}
            w_j \cdot f_{\bw_j,j,1}^{(L)}(x)
            \end{equation}
            where $f_{\bw_j,j,1}^{(L)}$ are recursively defined by
            (\ref{le1eq2}) and (\ref{le1eq3}) such that
            \[
      \sup_{x \in [-\alpha, \alpha]^d}
      \left|
      \sum_{j=1}^{I_n \cdot r \cdot \tilde{K}_n^d}
      w_j \cdot f_{\bw_j,j,1}^{(L)}(x)
      - m(x)
  \right|
  \leq \const  \cdot n^{-\frac{p}{2p+d}}
  \quad \mbox{and} \quad
  \sum_{j=1}^{I_n \cdot r \cdot \tilde{K}_n^d}
            w_j^2 \leq \const \cdot \frac{n^3}{I_n}.
            \]
            To show this we use Lemma \ref{le1}
            with $\bar{\alpha}=\alpha+1$. This implies that if we use
            a linear combination of $r \cdot \tilde{K}_n^d$
            fully connected networks of depth $L$ and width $r$
            with a single neuron as output, then the corresponding
            network will approximate $m$ in supremum norm on the cube $ [-\alpha, \alpha]^d$ up to an error of order
\begin{equation}
  \label{ple2eq2}
\frac{1}{\tilde{K}_n^p} \leq n^{-\frac{p}{2p+d}}.
\end{equation}
Here the outer weights are bounded in absolute value by
\[
\const \cdot \tilde{K}_n^{q+d} \leq \const \cdot n.
\]
Hence if we use a sum of $I_n$ of these networks where we divide all outer weights by $I_n$, then the corresponding network (\ref{ple2eq1}) still approximates
$m$ with a supremum norm error of order (\ref{ple2eq2}), and in addition
it holds
\[
\sum_{j=1}^{I_n \cdot r \cdot \tilde{K}_n^d}
w_j^2
\leq
I_n \cdot r \cdot \tilde{K}_n^d \cdot
\left(
\frac{ \const \cdot n }{I_n}
\right)^2
=
\const \cdot r \cdot \tilde{K}_n^d \cdot n^2 \cdot \frac{1}{I_n}
\leq \const \cdot \frac{n^3}{I_n}.
\]

Set
\[
\epsilon_n= \frac{1}{n^3}
\quad \mbox{and} \quad
N_n
=
4^{
L \cdot r^{L-1} \cdot (d+1)
} \cdot \left\lceil
n^{4 \cdot L \cdot r^{L-1} \cdot (d+1) }
\cdot (\log n)^2
\right\rceil.
\]
            In the {\it second step of the proof} we show
            that outside of an event whose probability is bounded from
            above by
            \[
  I_n \cdot r \cdot \tilde{K}_n^d \cdot
              \exp\left(
              - N_n \cdot  \left( \frac{1}{4} \cdot\frac{\epsilon_n}{A_n} \right)^{L \cdot r^{L-1} \cdot
                (d+1)}
               \right)
            \]
            there exist (random) indices $j_1, \dots, j_{I_n \cdot r \cdot \tilde{K}_n^d}
            \in \{1, \dots, r_n\}$ such that
            \begin{eqnarray*}
              &&
            \sup_{x \in [-\alpha,\alpha]^d}
            |
              f_{\bw_k,k,1}^{(L)}(x)
              -
f_{\bw^{(0)},j_k}^{(L)}(x)              
            |
            \leq
            (d+1) \cdot \alpha \cdot (2r+1)^{L-1} \cdot B^{L-1} \cdot \epsilon_n\\
            &&
            \hspace*{8cm} (k=1, \dots, I_n \cdot r \cdot \tilde{K}_n^d).
            \end{eqnarray*}
            (Observe that here $\bw^{(0)}$ is the initial weight of the pruned network.)
            Note that we can rewrite all $f_{\bw_k,k,1}^{(L)}$ such that in each
            layer $l \in \{1, \dots, L-1\}$ each neuron
            is connected to exactly
            one neuron of layer $l+1$
            (which is possible by increasing the width of the network)
            and represent it by a network
            $\bar{f}_{\bar{\bw}_k,k}^{(L)}$ defined by (\ref{se2eq2}) and (\ref{se2eq3}).
            This implies in particular that each neuron in layer
            $l \in \{2, \dots, L\}$ is connected to $r$ different neurons in layer $l-1$,
            which all have no connections to other neurons in layer $l$ than to this neuron.
            Because $f_{\bw,k,1}^{(L)}$ has width $r$, the newly constructed
            network
            $\bar{f}_{\bar{\bw}_k,k}^{(L)}$
            has the property that its weight vector $\bw_k$ satisfies
            \[
        \max_{i=1, \dots, r_n}
        | \{ 1 \leq j \leq r_n \, : \, (\bar{\bw}_k)_{i,j}^{(l)} \neq 0  \}|
        \leq r
         \quad (l=1, \dots, L-1).
            \]
            Let $\bar{f}_{\bar{\bw}_k,k}^{(L)}$
            be networks of that type which satisfy 
            $\bar{f}_{\bar{\bw}_k,k}^{(L)}(x)= f_{\bw_k,k,1}^{(L)}(x)$ $(x \in \R^d)$ for all
            $k=1, \dots, I_n \cdot r \cdot \tilde{K}_n^d$.

            Because of Lemma \ref{le3} it suffices to show that there
            exist random indices $j_1, \dots, j_{I_n \cdot r \cdot \tilde{K}_n^d}
            \in \{1, \dots, r_n\}$
such that all the weights in $\bar{f}_{\bar{\bw}_k,k}^{(L)}$ 
            have distance at most $\epsilon_n$ to the corresponding
            weights in  $f_{\bw^{(0)},j_k}^{(L)}$. (Here we consider
            in $f_{\bw^{(0)},j_k}^{(L)}$
            only those weights which have not been set to zero during the initial pruning.)

            To show this, we consider a special way to choose the
            initial network. We consider the initialization of $f_{\bw,j}^{(L)}$
            successively for $j=1, \dots, r_n$. When we initialize 
            $f_{\bw,j}^{(L)}$ for some fixed $j$, we do this in a top down
            manner starting at level $L$: We select randomly $k$ neurons
            of the previous level which we connect with the current
            neuron, select randomly their weights from the corresponding
            uniform distributions and then consider all neurons of the
            previous level which have yet not been considered during 
            determination of the weights of $f_{\bw,1}^{(L)}$, \dots,
            $f_{\bw,j-1}^{(L)}$. Since in our initialization in Section
            \ref{se2} the pruning is done independently from the
            selection of the values of the weights and also the
            initialization and pruning are done independently for each
            weight and each neuron, this leads to the same distribution
            of the random initial network.

            If we assume that during consideration of neuron $j$
            in each layer at least half of the neurons have not been
            considered before, then the probability that
            $f_{\bw^{(0)},j}^{(L)}$ has the same topology as
            $\bar{f}_{\bar{\bw}_k,k}^{(L)}$ and simultaneously that all
            corresponding weights are at a distance at most
            $\epsilon_n$ from one another
            is lower bounded by
            \begin{eqnarray*}
              &&
              \left(
              \prod_{l=1}^{L-1} \left( \frac{1}{2} \cdot \frac{\epsilon_n}{2B} \right)^{r^l}
              \right) \cdot
              \left( \frac{1}{2} \cdot \frac{\epsilon_n}{2A_n} \right)^{r^{L-1} \cdot (d+1)}
              \geq  \left( \frac{1}{2} \cdot \frac{\epsilon_n}{2 A_n} \right)^{L \cdot r^{L-1} \cdot (d+1)}.
              \end{eqnarray*}
            Here the factor $1/2$ takes into account that we have to choose neurons
            which have not be chosen in the previous steps, i.e., we have
            to choose neurons from a set of neurons which contains at least
            half of all possible $r_n$ neurons, which occurs with probability
            at least $1/2$ since we choose the neurons from an uniform
            distribution on the set of all neurons.
            Consequently the probability that none of the first $N_n$ networks 
            $f_{\bw^{(0)},j,1}^{(L)}$ $(j=1, \dots, N_n)$ has the same topology as
            $\bar{f}_{\bar{\bw}_k,k}^{(L)}$ and simultaneously that not all
            corresponding weights are at a distance at most $\epsilon_n$
            from one another
            is bounded from above by
             \begin{eqnarray*}
               &&
               \left( 1-
               \left( \frac{1}{4} \cdot \frac{\epsilon_n}{A_n} \right)^{L \cdot r^{L-1} \cdot (d+1)} \right)^{N_n}
               \leq
               \exp\left(
- N_n \cdot                
\left( \frac{1}{4} \cdot \frac{\epsilon_n}{A_n} \right)^{L \cdot r^{L-1} \cdot (d+1)}
               \right)
              \end{eqnarray*}
             provided during selection of each of the $N_n$ networks there are in all layers at least
             half of the neurons not considered.

             Since w.l.o.g.
             \[
r_n \geq 2 \cdot N_n \cdot I_n \cdot r \cdot \tilde{K}_n^d  \cdot r^{L} ,
             \]
             the last condition is satisfied if we consider successively all
             networks $\bar{f}_{\bar{\bw}_k,k}^{(L)}$ $(k=1, \dots, I_n \cdot r \cdot \tilde{K}_n^d)$ and we can conclude: The probability that there exists
             $k \in \{ 1, \dots, I_n \cdot r \cdot \tilde{K}_n^d\}$ such that none of the networks $f_{\bw^{(0)},j,1}^{(L)}$ has the same topology as
            $\bar{f}_{\bar{\bw}_k,k}^{(L)}$ and simultaneously that not all
            corresponding weights are at a distance at most $\epsilon_n$ from one another
            is upper bounded by
            \begin{eqnarray*}
              &&
              I_n \cdot r \cdot \tilde{K}_n^d \cdot
              \exp\left(
- N_n \cdot                
\left( \frac{1}{4} \cdot \frac{\epsilon_n}{A_n} \right)^{L \cdot r^{L-1} \cdot (d+1)}
               \right).
              \end{eqnarray*}

            In the {\it third step of the proof} we show the assertion.
            Let
            \[
 f_{\bw}(x) = \sum_{j=1}^{I_n \cdot r \cdot \tilde{K}_n^d}
            w_j \cdot f_{\bw,j,1}^{(L)}(x)
            \]
            be the neural network of step 1, and let
$j_1, \dots, j_{I_n \cdot r \cdot \tilde{K}_n^d}
            \in \{1, \dots, K_n\}$ be the random indices from step 2,
            and set
            \[
            \bw_j^*=
            \begin{cases}
              w_k, & \mbox{if } j=j_k \mbox{ for some } k \in \{1, \dots I_n \cdot r \cdot \tilde{K}_n^d\},
              \\
              0, & else.
              \end{cases}
            \]
            Then
            \[
            \sum_{k=1}^{r_n} |w_k^*|^2
            =
  \sum_{j=1}^{I_n \cdot r \cdot \tilde{K}_n^d}
  w_j^2 \leq \const \cdot \frac{n^3}{I_n}
  \leq \frac{\const \cdot n^{4L \cdot r^{L-1} \cdot (d+1)+4}}{r_n},
            \]
            and outside of an event whose probability is bounded from above by
            \[
                          I_n \cdot r \cdot \tilde{K}_n^d \cdot
              \exp\left(
              - N_n  \cdot \left( \frac{1}{4} \cdot \frac{\epsilon_n}{A_n} \right)^{L \cdot r^{L-1} \cdot
                (d+1)}
              \right)
              \leq
               n^{\const} \cdot \exp \left(
-  (\log n)^2
              \right)
\]
we have
\begin{eqnarray*}
  &&
      \sup_{x \in [-\alpha, \alpha]^d}
      \left|
      \sum_{k=1}^{r_n} w_{k}^* \cdot f^{(L)}_{\bw^{(0)},k }(x) -
      m(x)
  \right|
  \\
  &&
  \leq
      \sup_{x \in [-\alpha, \alpha]^d}
      \left|
      \sum_{k=1}^{I_n \cdot r \cdot \tilde{K}_n^d} w_{k} \cdot f^{(L)}_{\bw^{(0)},j_k }(x) -
      \sum_{k=1}^{I_n \cdot r \cdot \tilde{K}_n^d}
            w_k \cdot f_{\bw_k,k,1}^{(L)}(x)
            \right|
            \\
            &&
            \quad
  +    \sup_{x \in [-\alpha, \alpha]^d}
      \left|
\sum_{k=1}^{I_n \cdot r \cdot \tilde{K}_n^d}
            w_k \cdot f_{\bw_k,k,1}^{(L)}(x) - m(x)
  \right|
  \\
  &&
  \leq
  I_n \cdot r \cdot \tilde{K}_n^d \cdot
\frac{ \const \cdot n }{I_n}
\cdot
            (d+1) \cdot \alpha \cdot (2r+1)^{l-1} \cdot B^{l-1} \cdot \epsilon_n
+
\const  \cdot n^{-\frac{p}{2p+d}}
\\
&&
\leq
\const  \cdot n^{-\frac{p}{2p+d}},
\end{eqnarray*}
which implies the assertion.
\hfill $\Box$

\subsection{Neural network generalization}
\label{se5sub2}

Our next lemma is our main tool to analyze the generalization error of our estimate.
It extends the previous result of Kohler (2026) to the more general topology considered
in our paper, and sharpens the bound there by computing explicitly the constant in 
the exponent of the upper bound on the covering number, which will allows
us in the proof of the main result to choose some parameter $k$ as a function
of the sample size. The main idea of the proof is to approximate over-parametrized
deep neural networks with bounded weights by the piecewise Taylor polynomials
and use the fact that the covering number of these piecewise Taylor polynomials
can be easily bounded. 
\begin{lemma}
  \label{le4}
  Let $L, r_n, r \in \N$, let $\alpha, \beta, A, B, C \geq 1$.
  Let $\sigma$ be the logistic squasher and let $\F$ be the set
  of all neural networks $f_\bw$ defined by (\ref{se2eq1})--(\ref{se2eq3}),
  where the weight vector $\bw$ satisfies
  \begin{equation}
    \label{le4eq1}
    \sum_{i=1}^{r_n} | w_{1,i}^{(L)}| \leq C,
    \end{equation}
  \begin{equation}
    \label{le4eq2}
    |w_{i,j}^{(l)}| \leq B \quad  \mbox{for all } i,j=1, \dots, r_n, l=1, \dots, L-1,
    \end{equation}
  \begin{equation}
    \label{le4eq3}
    |w_{i,j}^{(0)}| \leq A \quad  \mbox{for all }  i=1, \dots, r_n, j=1, \dots, d
  \end{equation}
  and
  \begin{equation}
    \label{le4eq4}
    | \{ j \in \{1, \dots, r_n\} \, : \, w_{i,j}^{(l)} \neq 0 \} | \leq r
    \quad \mbox{for all } i=1, \dots, r_n, l=1, \dots, L-1.
  \end{equation}
  Then we have for any $0<\epsilon \leq 1$,
  any $x_1, \dots, x_n \in [-\alpha,\alpha]^d$
 and any  $k \in \N$ 
  \[
  \Nu( \epsilon, T_\beta \F, x_1^n)
  \leq
3 \left(\frac{12e \cdot \beta}{\epsilon}\right)^
             {
            2 \cdot   (4ed)^d \cdot 2^{L \cdot d} \cdot (r+k)^{3 \cdot d \cdot L} 
               \cdot
\alpha^d \cdot A^d \cdot B^{(L-1) \cdot d}
\cdot
 \left( \frac{C}{\epsilon} \right)^{\frac{d}{k}}  
          +2
}.       
  \]
  \end{lemma}

\noindent
    {\bf Proof.} 
    In the {\it first step of the proof} we show for any $s \in \N$
    \[
\max_{x \in \R} | \sigma^{(s)}(x)| \leq 2^{s-1} \cdot (s+1)!
    \]
    To do this we show by induction
    \[
\sigma^{(s)}(x) = \sum_{l=1}^{s+1} a_{s,l} \cdot (\sigma(x))^l
    \]
    for some $a_{s,l} \in \R$ which satisfy
    \[
\max_{l=1, \dots, s+1} |a_{s,l}| \leq 2^{s-1} \cdot s!
    \]
    For $s=1$ we have
    \[
    \sigma^\prime(x)= (-1) \cdot (1+e^{-x})^{-2} \cdot e^{-x} \cdot (-1)
    = \frac{1}{1+e^{-x}} \cdot \frac{e^{-x}}{1+e^{-x}} = \sigma(x) \cdot (1-\sigma(x)) = \sigma(x) - (\sigma(x))^2,
    \]
    so the assertion trivially holds.

    Assume now that the assertion holds for some $s \in \N$. Then
    \begin{eqnarray*}
      && \sigma^{(s+1)}(x) \\
      &&= \sum_{l=1}^{s+1} a_{s,l} \cdot l \cdot (\sigma(x))^{l-1} \cdot \sigma^\prime(x)
      \\
      &&=
      \sum_{l=1}^{s+1} a_{s,l} \cdot l \cdot (\sigma(x))^{l-1} \cdot (\sigma(x) - (\sigma(x))^2)
      \\
      &&=
      \sum_{l=1}^{s+1} a_{s,l} \cdot l \cdot (\sigma(x))^{l}
      -
      \sum_{l=1}^{s+1} a_{s,l} \cdot l \cdot (\sigma(x))^{l+1}
      \\
      &&=
      a_{s,1} \cdot \sigma(x) + \sum_{l=2}^{s+1} (a_{s,l} \cdot l
- a_{s,l-1} \cdot (l-1))  
\cdot (\sigma(x))^{l}
-
a_{s,s+1} \cdot (s+1) \cdot (\sigma(x))^{s+2},
      \end{eqnarray*}
    which implies
    \[
\sigma^{(s+1)}(x) = \sum_{l=1}^{s+2} a_{s+1,l} \cdot (\sigma(x))^s
    \]
    where $ a_{s+1,l}$ satisfy
    \[
    \max_{l=1, \dots, s+2} |a_{s+1,l}| \leq
2 \cdot (s+1) \cdot  \max_{l=1, \dots, s+1} |a_{s,l}|
\leq
2 \cdot (s+1) \cdot
    2^{s-1} \cdot s!
    =
    2^{s} \cdot (s+1)!
    \]

    In the {\it second step of the proof} we show
      for any $f_{\bw} \in \F$, any $x \in \Rd$, any $k \in \N$ and any $s_1, \dots, s_k \in \{1, \dots, d\}$
    \begin{equation}
      \label{ple4eq1}
      \left|
      \frac{\partial^k f_{\bw}}{\partial x^{(s_1)} \dots \partial x^{(s_k)}} (x)
      \right| \leq (r+k)^{3 \cdot k \cdot L} \cdot \left( 2^{k-1} \right)^{L} \cdot C \cdot B^{(L-1) \cdot k} \cdot A^k =: c
      .
    \end{equation}
    Because of (\ref{le4eq1}) and
    \[
    \frac{\partial^k f_{\bw}}{\partial x^{(s_1)} \dots \partial x^{(s_k)}} (x)
    =
    \sum_{i=1}^{r_n} w_{1,i}^{(L)} \cdot
    \frac{\partial^k f_{\bw,i}^{(L)}}{\partial x^{(s_1)} \dots \partial x^{(s_k)}} (x)
    \]
    it suffices to show for all $i \in \{1, \dots, r_n\}$
    \begin{equation}
      \label{ple4eq2}
      \left|
      \frac{\partial^k f_{\bw,i}^{(L)}}{\partial x^{(s_1)} \dots \partial x^{(s_k)}} (x)
      \right| \leq (r+k)^{3 \cdot k \cdot L}  \cdot \left( 2^{k-1}  \right)^{L}  \cdot B^{(L-1) \cdot k} \cdot A^k
        .
    \end{equation}
    To show (\ref{ple4eq2}) we will demonstrate for any
    $i \in \{1, \dots, r_n\}$,
    any $s \in \{1, \dots, L\}$,
    any $k \in \N$ and any $s_1, \dots, s_k \in \{1, \dots, d\}$
    \begin{equation}
      \label{ple4eq3}
      \left|
      \frac{\partial^k f_{\bw,i}^{(s)}}{\partial x^{(s_1)} \dots \partial x^{(s_k)}} (x)
      \right| \leq (r+k)^{3 \cdot k \cdot s}  \cdot \left( 2^{k-1}  \right)^{s}  \cdot B^{(s-1) \cdot k} \cdot A^k
        .
    \end{equation}
    For $s=1$ we have
    \[
    \frac{\partial^k f_{\bw,i}^{(1)}}{\partial x^{(s_1)} \dots \partial x^{(s_k)}} (x)
    =
    \prod_{j=1}^k w_{i,s_j}^{(0)} \cdot
      \sigma^{(k)} \left(\sum_{t=1}^d w_{i,t}^{(0)} \cdot x^{(t)} + w_{i,0}^{(0)} \right)
    \]
    from which we get (\ref{ple4eq3}) for all $k \in \N$ by (\ref{le4eq3}) and the assertion of the
    first step of the proof.

    Let $l \in \{2, \dots, L\}$ and assume that (\ref{ple4eq3}) holds
    for $s=1, \dots, l-1$, all
$i \in \{1, \dots, r_n\}$,
    all $k \in \N$ and all $s_1, \dots, s_k \in \{1, \dots, d\}$. We have
    \begin{eqnarray*}
      &&
      \frac{\partial f_{\bw,i}^{(l)}}{\partial x^{(s)}}(x)
      \\
      &&
      =
      \sigma^\prime \left(\sum_{t=1}^{r_n} w_{i,t}^{(l-1)} \cdot f_{\bw,t}^{(l-1)}(x) + w_{i,0}^{(l-1)} \right)
      \cdot
      \sum_{j=1}^{r_n} w_{i,j}^{(l-1)} \cdot \frac{\partial
f_{\bw,j}^{(l-1)}
      }{\partial x^{(s)}}(x)
      \\
      &&
      =
      \sum_{j=1, \dots, r_n: w_{i,j}^{(l-1)} \neq 0 } w_{i,j}^{(l-1)} 
      \cdot
      \sigma^\prime \left(\sum_{t=1}^{r_n} w_{i,t}^{(l-1)} \cdot f_{\bw,t}^{(l-1)}(x) + w_{i,0}^{(l-1)} \right)
      \cdot
      \frac{\partial
f_{\bw,j}^{(l-1)}
      }{\partial x^{(s)}}(x).
    \end{eqnarray*}
    Because of (\ref{le4eq4}) the sum on the right hand side above has at
    most $r$ terms. Using the product rule of derivation we can conclude
    for any $k \in \N$ that 
       \begin{equation}
      \label{ple4eq4}
    \frac{\partial^k f_{\bw, i}^{(l)}}{\partial x^{(s_1)} \dots \partial x^{(s_k)}} (x)
       \end{equation}
 is a sum of at most $r \cdot (r+k)^{k-1} \leq (r+k)^k$ terms of the form
          \begin{eqnarray}
      \label{ple4eq4}
      &&
w
\cdot
\sigma^{(s)} \left(\sum_{j=1}^{r_n} w_{i,j}^{(l-1)} \cdot f_{\bw, j}^{(l-1)}(x) + w_{i,0}^{(l-1)} \right)
      \cdot
      \frac{\partial^{t_{1}}
f_{\bw, j_1}^{(l-1)}
      }{\partial x^{(r_{1,1})} \dots \partial x^{(r_{1,t_{1}})}}(x)
      \cdot
\nonumber      \\
      &&
      \hspace*{5cm}
\dots
      \cdot
      \frac{\partial^{t_{s}}
f_{\bw, j_s}^{(l-1)}
      }{\partial x^{(r_{s,1})} \dots \partial x^{(r_{s,t_{s}})}}(x)
    \end{eqnarray}
    where we have $s \in \{1, \dots,k\}$,
    $|w| \leq  B^{s}$ and $t_1+ \dots + t_s = k$.
    The induction hypothesis
and the first step of the proof
    imply that the absolute value of
    (\ref{ple4eq4}) is upper bounded by
    \begin{eqnarray*}
      &&
      B^{s} \cdot
      2^{s-1} \cdot (s+1)!
      \cdot
\prod_{j=1}^s
(r+t_j)^{3 \cdot t_j \cdot (l-1)}  \cdot \left( 2^{t_j-1}  \right)^{l-1}  \cdot B^{(l-2) \cdot t_j} \cdot A^{t_j}
\\
&&
\leq
      B^{k} \cdot
      2^{k-1} \cdot (k+1)!
      \cdot
\prod_{j=1}^s
(r+k)^{3 \cdot t_j \cdot (l-1)}  \cdot \left( 2^{t_j-1}  \right)^{l-1}  \cdot B^{(l-2) \cdot t_j} \cdot A^{t_j}
\\
&&
\leq
(r+k)^{3 \cdot k \cdot (l-1)} \cdot (k+1)! \cdot (2^{k-1})^l
\cdot B^{k \cdot (l-1)} \cdot A^{k}, 
      \end{eqnarray*}
    hence
    \begin{eqnarray*}
      &&
      \left|
    \frac{\partial^k f_{\bw, i}^{(l)}}{\partial x^{(s_1)} \dots \partial x^{(s_k)}} (x)
    \right|
    \\
    &&
    \leq
    (r+k)^{k}
     \cdot
     (r+k)^{3 \cdot k \cdot (l-1)} \cdot (k+1)!  \cdot (2^{k-1})^l
     \cdot
     B^{k \cdot (l-1)} \cdot A^{k}
     \\
     &&
     \leq
     (r+k)^{3 \cdot k \cdot l}  \cdot (2^{k-1})^l
     \cdot
     B^{k \cdot (l-1)} \cdot A^{k},     
    \end{eqnarray*}
    where we have used
    \[
(k+1)! \leq (k+1)^{k+1} \leq (r+k)^{2k}.
    \]
    
Let
$\G$ be the set of all polynomials of degree less than or equal to $k-1$,
let $\Pi$ be a partition of $[-\alpha,\alpha]^d$ into cubes of sidelength
\[
\left(\frac{k!}{2 d^k} \cdot \frac{\epsilon}{c}\right)^{1/k}
\leq 1,
\]
and let
$\G\circ \Pi$ be the class of all functions
whose restriction on each cube in $\Pi$ lies in $\G$ and
which are zero outside of $[-\alpha,\alpha]^d$.

In the {\it fourth step of the proof} we show
\begin{equation}
	\label{ple4eq5}
	\Nu_p \left(
	\epsilon, \{T_{\beta} f \, : \, f \in \F \}, x_1^n
	\right)
	\leq
	\Nu_p  \left(
	\frac{\epsilon}{2}, T_\beta \G\circ \Pi, x_1^n
	\right).     
\end{equation}

Let  $(T f_\bw)_{k-1,u}$ be the multivariate Taylor polynomial of $f_\bw$ of degree $k-1$ around $u\in \R^d$, i.e.
\begin{align*}
      &(Tf_\bw)_{k-1,u}(x)
      \\
      &
            =
      \sum_{j_1, \dots, j_d \in \N_0, \atop
        j_1 + \dots + j_d \leq k-1}
            \frac{1}{j_1! \cdots j_d!}\cdot
      \frac{
\partial^{j_1+\dots + j_d} f_\bw
      }{
\partial^{j_1} x^{(1)} \dots \partial^{j_d} x^{(d)}
      }
      (u)
      \cdot
      (x^{(1)}-u^{(1)})^{j_1}  \cdots  (x^{(d)}-u^{(d)})^{j_d}.
\end{align*}
For each $I\in \Pi$ fix some $u\in I$. By the multivariate Taylor theorem  we get
as in the proof of Lemma 1 in Kohler (2014)
for each $x\in I$
\begin{eqnarray*}
  &&
  |f_\bw(x)-(T f_\bw)_{k-1,u}(x)|
  \\
  &&
  =
  \Bigg|
  f_\bw(x)-(T f_\bw)_{k-2,u}(x)
  \\
  &&
  \quad
  -
  \sum_{j_1, \dots, j_d \in \N_0, \atop
        j_1 + \dots + j_d = k-1}
            \frac{1}{j_1! \cdots j_d!}\cdot
      \frac{
\partial^{j_1+\dots + j_d} f_\bw
      }{
\partial^{j_1} x^{(1)} \dots \partial^{j_d} x^{(d)}
      }
      (u)
      \cdot
      (x^{(1)}-u^{(1)})^{j_1}  \cdots  (x^{(d)}-u^{(d)})^{j_d}
  \Bigg|
  \\
  &&
  =
  \Bigg|
\sum_{j_1, \dots, j_d \in \N_0, \atop
        j_1 + \dots + j_d = k-1}
\frac{k-1}{j_1! \cdots j_d!}\cdot
\int_0^1
(1-t)^{k-2}
\cdot
      \frac{
\partial^{j_1+\dots + j_d} f_\bw
      }{
\partial^{j_1} x^{(1)} \dots \partial^{j_d} x^{(d)}
      }
      (u+t \cdot (x-u)) \, dt
      \\
      &&
      \hspace*{4cm}
      \cdot
      (x^{(1)}-u^{(1)})^{j_1}  \cdots  (x^{(d)}-u^{(d)})^{j_d}
  \\
  &&
  \quad
  -
  \sum_{j_1, \dots, j_d \in \N_0, \atop
        j_1 + \dots + j_d = k-1}
            \frac{1}{j_1! \cdots j_d!}\cdot
      \frac{
\partial^{j_1+\dots + j_d} f_\bw
      }{
\partial^{j_1} x^{(1)} \dots \partial^{j_d} x^{(d)}
      }
      (u)
      \cdot
      (x^{(1)}-u^{(1)})^{j_1}  \cdots  (x^{(d)}-u^{(d)})^{j_d}
      \Bigg|
  \\
  &&
  =
  \Bigg|
\sum_{j_1, \dots, j_d \in \N_0, \atop
        j_1 + \dots + j_d = k-1}
\frac{k-1}{j_1! \cdots j_d!}
\cdot
      (x^{(1)}-u^{(1)})^{j_1}  \cdots  (x^{(d)}-u^{(d)})^{j_d}
\\
&&
\quad
\cdot
\int_0^1
(1-t)^{k-2}
\cdot
\left(
\frac{
\partial^{j_1+\dots + j_d} f_\bw
      }{
\partial^{j_1} x^{(1)} \dots \partial^{j_d} x^{(d)}
      }
(u+t \cdot (x-u))
-
      \frac{
\partial^{j_1+\dots + j_d} f_\bw
      }{
\partial^{j_1} x^{(1)} \dots \partial^{j_d} x^{(d)}
      }
      (u)
\right)
\, dt
\Bigg|
\\
&&
\leq
\sum_{j_1, \dots, j_d \in \N_0, \atop
        j_1 + \dots + j_d = k-1}
\frac{k-1}{j_1! \cdots j_d!}
\cdot
      |x^{(1)}-u^{(1)}|^{j_1}  \cdots  |x^{(d)}-u^{(d)}|^{j_d}
\\
&&
\quad
\cdot
\int_0^1
(1-t)^{k-2}
\cdot
\left|
\frac{
\partial^{j_1+\dots + j_d} f_\bw
      }{
\partial^{j_1} x^{(1)} \dots \partial^{j_d} x^{(d)}
      }
(u+t \cdot (x-u))
-
      \frac{
\partial^{j_1+\dots + j_d} f_\bw
      }{
\partial^{j_1} x^{(1)} \dots \partial^{j_d} x^{(d)}
      }
      (u)
\right|
\, dt.
  \end{eqnarray*}
By the mean value theorem and
(\ref{ple4eq1}) we get for any $j_1, \dots, j_d \in \N_0$
with $j_1 + \dots + j_d = k-1$
\begin{eqnarray*}
  &&
\left|
\frac{
\partial^{j_1+\dots + j_d} f_\bw
      }{
\partial^{j_1} x^{(1)} \dots \partial^{j_d} x^{(d)}
      }
(u+t \cdot (x-u))
-
      \frac{
\partial^{j_1+\dots + j_d} f_\bw
      }{
\partial^{j_1} x^{(1)} \dots \partial^{j_d} x^{(d)}
      }
      (u)
      \right|
      \\
      &&
      \leq
      \sum_{i=1}^d
      \Bigg|
\frac{
\partial^{j_1+\dots + j_d} f_\bw
      }{
\partial^{j_1} x^{(1)} \dots \partial^{j_d} x^{(d)}
      }
(u^{(1)}, \dots, u^{(i-1)},  u^{(i)}+t \cdot (x^{(i)}-u^{(i)}), \dots,
u^{(d)}+t \cdot (x^{(d)}-u^{(d)}))
\\
&&
\quad
-
\frac{
\partial^{j_1+\dots + j_d} f_\bw
      }{
\partial^{j_1} x^{(1)} \dots \partial^{j_d} x^{(d)}
      }
(u^{(1)}, \dots, u^{(i)},  u^{(i+1)}+t \cdot (x^{(i+1)}-u^{(i+1)}), \dots,
u^{(d)}+t \cdot (x^{(d)}-u^{(d)}))
\Bigg|
\\
&&
\leq \sum_{i=1}^d c \cdot |u^{(i)}+t \cdot (x^{(i)}-u^{(i)}) - u^{(i)}|
\\
&&
= c \cdot t \cdot \sum_{i=1}^d |x^{(i)}-u^{(i)}|.
  \end{eqnarray*}
Using
\[
      \int_0^1 (k-1) \cdot t \cdot (1-t)^{k-2} dt= \frac{1}{k}
\]
we get
\begin{align*}
  &|f_\bw(x)-(T f_\bw)_{k-1,u}(x)|
  \\
    & 
    \leq
    \sum_{j_1, \cdot \dots \cdot, j_d \in \N_0, \atop
    	j_1+\dots + j_d = k-1}
    	\frac{k-1}{j_1!\cdots j_d!}
    	\cdot
    	c
        \cdot
        \sum_{i=1}^d \left|x^{(i)}-u^{(i)}\right|
        \cdot
        \left|x^{(1)}-u^{(1)}\right|^{j_1} \cdots  \left|x^{(d)}-u^{(d)}\right|^{j_d}
        \\
        &
        \hspace*{3cm}
        \cdot
        \int_0^1 t \cdot (1-t)^{k-2} dt 
      \\
      &
      =
\frac{c}{k} \cdot 
\sum_{i=1}^d \left|x^{(i)}-u^{(i)}\right|
\cdot
\frac{1}{(k-1)!} \cdot
    \sum_{j_1, \cdot \dots \cdot, j_d \in \N_0, \atop
    	j_1+\dots + j_d = k-1}
    	\frac{(k-1)!}{j_1!\cdots j_d!}
        \cdot
        \left|x^{(1)}-u^{(1)}\right|^{j_1} \cdots  \left|x^{(d)}-u^{(d)}\right|^{j_d}
        \\
        &
        =
        \frac{c}{k!} \cdot 
\sum_{i=1}^d \left|x^{(i)}-u^{(i)}\right|
\cdot
\left( \sum_{j=1}^d \left|x^{(j)}-u^{(j)}\right| \right)^{k-1}
        \\
        &
        =
        \frac{c}{k!} \cdot \left(
        \sum_{i=1}^d \left|x^{(i)}-u^{(i)}\right|
        \right)^k
        \\
        &
        \leq
        c \cdot \frac{d^k}{k!} \cdot \|x-u\|_\infty^k
        \leq
        c \cdot \frac{d^k}{k!} \cdot  \frac{k!}{2 \cdot d^k} \cdot \frac{\epsilon}{c}
        =
        \frac{\epsilon}{2},
\end{align*}
where for the proof of the second
equality above we have used the multinomial theorem.
By repeating this argument for every cube $I\in\Pi$ we
see that for each $f_\bw$ we can find $g \in \mathcal{G}\circ \Pi$ such that
\[
|f_\bw(x)-g(x)|\leq \frac{\epsilon}{2}
\]
  holds for all $x\in \R^d$, which implies (\ref{ple4eq5}).

    In the {\it third step} of the proof we show the assertion of Lemma \ref{le4}.
    $\mathcal{G}\circ \Pi$ is a linear vector space of dimension less than or equal to \[
    k^d \cdot
    \left(
    \frac{4 \alpha}{
\left(\frac{k!}{2 d^k} \cdot \frac{\epsilon}{c}\right)^{1/k}
    }
    \right)^d
    =
4^d \cdot 2^{d/k} \cdot d^d \cdot \frac{k^d}{(k!)^{d/k}} \cdot    \alpha^d \cdot \left( \frac{c}{\epsilon} \right)^{\frac{d}{k}}.
\]
Since
\[
e^k \geq \frac{k^k}{k!}, \quad i.e., \quad
k! \geq \frac{k^k}{e^k}
\]
this in turn is bounded from above by
\[
(4ed)^d \cdot 2^{d/k} \cdot    \alpha^d \cdot \left( \frac{c}{\epsilon} \right)^{\frac{d}{k}}.
\]
We conclude from Theorem 9.4 and Theorem 9.5 in Gy\"orfi et al. (2002),
    	\begin{align*}
    	\mathcal{N}(\frac{\epsilon}{2}, T_\beta \mathcal{G} \circ \Pi, x_1^n)
    	\leq 3 \left(\frac{2e(2\beta)}{(\epsilon/2)}\log\left(\frac{3e(2 \beta)}{(\epsilon/2)}\right)\right)^{
  (4ed)^d \cdot 2^{d/k} \cdot    \alpha^d \cdot \left( \frac{c}{\epsilon} \right)^{\frac{d}{k}}          +1
}.
    \end{align*}
        We have
        \begin{eqnarray*}
        c^{\frac{d}{k}}
        &=&
        \left(
        (r+k)^{3 \cdot k \cdot L} \cdot \left( 2^{k-1}  \right)^{L}
\right)^{\frac{d}{k}}
\cdot C^{\frac{d}{k}} \cdot B^{(L-1) \cdot d} \cdot A^d
\\
&=&
(r+k)^{3 \cdot d \cdot L} \cdot 2^{L \cdot d}
\cdot 2^{- L \cdot \frac{d}{k}}
\cdot A^d \cdot B^{(L-1) \cdot d} 
\cdot C^{\frac{d}{k}} ,
        \end{eqnarray*}
        hence
 \begin{align*}
    	\mathcal{N}(\frac{\epsilon}{2}, T_\beta \mathcal{G} \circ \Pi, x_1^n)
    	\leq 3 \left(\frac{12e \cdot \beta}{\epsilon}\right)^
             {
               2 \cdot        (4ed)^d
               \cdot 2^{L \cdot d} \cdot (r+k)^{3 \cdot d \cdot L} 
               \cdot
\alpha^d \cdot A^d \cdot B^{(L-1) \cdot d}
\cdot
 \left( \frac{C}{\epsilon} \right)^{\frac{d}{k}}  
          +2
}.       
 \end{align*}
 
\noindent    
    Together with (\ref{ple4eq5}) this implies the assertion.

    \hfill $\Box$

\subsection{Neural network optimization}
\label{se5sub3}

Our next lemma is our main tool for analyzing the optimization
error of the estimate. It extends results from Kohler (2026) and
Kohler, Krzy\.zak and Molinero R\"omer (2026) to the case
of non-constant step sizes.

\begin{lemma}
  \label{le11}
  Let $d, J_n \in \N$, and for
\[
\bw=(w_1, \dots, w_{J_n}) \in \R^{J_n}
\]
let
$
  f_{\bw}: \R^d \rightarrow \R
  $
  be a (deep) neural network with weight vector $\bw$.
  Let $F_n$ be defined by (\ref{se2eq7}),
  let $\bw^{(0)} \in \R^{J_n}$, $t_n \in \N$ and
  $\lambda_0, \dots, \lambda_{t_n-1}>0$,
  and set
  \begin{equation}
    \label{le11eq1}
  \bw^{(t+1)} = \bw^{(t)} - \lambda_t \cdot \nabla_{\bw} F_n (\bw^{(t)})
  \quad
  \mbox{for } t=0, \dots, t_n-1.
  \end{equation}
  Let $\bw^*\in\R^{J_n}$.
  Assume
 \begin{equation}
    \label{le11eq2}
    F_n( \bw^{(s+1)}) \leq F_n(\bw^{(s)}) - \frac{\lambda_s}{2}
    \cdot \left\|
\nabla_\bw F_n(\bw^{(s)})
    \right\|^2
    \end{equation}
for $s=0,1, \dots, t_n-1$,
\begin{equation}
  \label{le11eq3}
  \sum_{j=1}^{J_n}
  \left|
  \frac{\partial}{\partial w^{(j)}}
  f_{\bw_1}(x)
-
\frac{\partial}{\partial w^{(j)}}
f_{\bw_2}(x)
\right|^2
   \leq C_n^2 \cdot
  \| \bw_1 - \bw_2 \|^2
  \end{equation}
for all  $\bw_1, \bw_2 \in \R^{J_n}$ with
  $
\max\{  \| \bw^* - \bw_1\|,   \| \bw^* - \bw_2\| \}
  \leq
  \| \bw^* - \bw^{(0)}\|$
  and all
 $x \in \{X_1, \dots, X_n\}$,
  \begin{equation}
    \label{le11eq4}
    |Y_i| \leq \beta_n \quad (i=1, \dots, n),
    \end{equation}
  \begin{equation}
    \label{le11eq5}
    |f_{\bw^*}(X_i)| \leq \beta_n \quad (i=1, \dots, n)  
    \end{equation}
and
  \begin{equation}
    \label{le11eq6}
    C_n \cdot  \| \bw^* - \bw^{(0)}\|^2 \leq \beta_n.
  \end{equation}
  Then
  \begin{eqnarray*}
F_n (\bw^{(t_n)})
    & \leq &
    F_n( \bw^*)
    +
    \left(
5 \cdot \beta_n \cdot C_n + \frac{1}{2 \cdot \sum_{s=0}^{t_n -1} \lambda_s}
    \right)
    \cdot
    \| \bw^* - \bw^{(0)}\|^2
    \\
    &&
    +
    \frac{\max_{s=0, \dots, t_n-1} \lambda_s}{\sum_{s=0}^{t_n -1} \lambda_s} \cdot F_n( \bw^{(0)}).
    \end{eqnarray*}
\end{lemma}

In the proof of Lemma \ref{le11} we will need the following auxiliary result.

\begin{lemma}
  \label{le12}
  Let $\bw^{(0)}, \bw^* \in \R^{J_n}$,
  define $\bw^{(t+1)}$ by (\ref{le11eq1})
  for $t=0, \dots, t_n-1$
  and assume that (\ref{le11eq2}) holds for $s=0, \dots, t-1$.
  Let
   \begin{equation}\label{le12eq1}
    f_{lin, \bw, \tilde{\bw}}(x)
    =
    f_{\bw}(x)+
    \sum_{j=1}^{J_n}
    \frac{\partial f_{\bw}(x)}{ \partial \bw^{(j)}}
      \cdot
      (\tilde{\bw}^{(j)} - \bw^{(j)})
   \end{equation}
   be the linear Taylor polynomial of $f_{\tilde{\bw}}(x)$ around $\bw$,
  set
  \[
  F_{n,lin,\bw}(\bw^*)
  =
\frac{1}{n} \sum_{i=1}^n |Y_i - f_{lin,\bw,\bw^*}(X_i)|^2,  
\]
and assume
  \begin{equation}
    \label{le12eq2}
    \left|
f_{lin,\bw,\bw^*}(x)-f_{\bw^*}(x)
\right|
\leq
C_n \cdot \| \bw^* - \bw\|^2
    \end{equation}
  for all $x \in \{X_1, \dots, X_n \}$ and all $\bw \in \R^{J_n}$ with
  $
  \| \bw^* - \bw\|
  \leq
  \| \bw^* - \bw^{(0)}\|,
  $
and
\begin{equation}
  \label{le12eq3}
F_n( \bw^{(s+1)}) \geq F_{n,lin,\bw^{(s)}}(\bw^*)
\end{equation}
for $s=0,1, \dots, t-1$. Then
\[
\| \bw^{(s)} - \bw^{*}\|
\leq
\| \bw^{(0)} - \bw^{*}\|
\quad
\mbox{for } s=1, \dots, t.
\]
\end{lemma}

\noindent
    {\bf Proof.}
    The proof is a straightforward modification of the proof of
    Lemma 2 in Kohler, Krzy\.zak and Molinero R\"omer (2026).
    For the sake of completeness we nevertheless present a complete
    proof.
    
    Let $s \in \{0,1, \dots, t-1\}$ be arbitrary.
    We have
    \[
    F_{n,lin,\bw^{(s)}}(\bw^{(s)})=     F_{n}(\bw^{(s)})
    \quad \mbox{and} \quad
    \nabla_\bw F_{n,lin,\bw^{(s)}}(\bw^{(s)})=     \nabla_\bw F_{n}(\bw^{(s)})
    \]
    and (since $ \bw \mapsto  F_{n,lin,\bw^{(s)}}(\bw)$ is convex
    and differentiable)
    \[
    F_{n,lin,\bw^{(s)}}(\bw^{(s)})
    -
    F_{n,lin,\bw^{(s)}}(\bw^*)
    \leq
    \,
    <
    \nabla_\bw F_{n,lin,\bw^{(s)}}(\bw^{(s)}),
    \bw^{(s)} - \bw^*
    >,
    \]
    which implies
    \begin{eqnarray*}
    <
    \nabla_\bw F_{n}(\bw^{(s)}),
    \bw^{(s)} - \bw^*
    >
    &=&
     <
    \nabla_\bw F_{n,lin,\bw^{(s)}}(\bw^{(s)}),
    \bw^{(s)} - \bw^*
    >
    \\
    &\geq&
     F_{n,lin,\bw^{(s)}}(\bw^{(s)})
    -
    F_{n,lin,\bw^{(s)}}(\bw^*)
    \\
    &=&
     F_{n}(\bw^{(s)})
    -
    F_{n,lin,\bw^{(s)}}(\bw^*).
    \end{eqnarray*}
    Hence
    \begin{eqnarray*}
      &&
\| \bw^{(s+1)} - \bw^{*}\|^2
\\
&&
=
\left\|
\bw^{(s)}
-
\lambda_s
\cdot \nabla_{\bw} F_n( \bw^{(s)} )
- \bw^{*}
\right\|^2
\\
&&
=
\left\|
\bw^{(s)}
- \bw^{*}
\right\|^2
-
2 \cdot \lambda_s \cdot
    <
    \nabla_\bw F_{n}(\bw^{(s)}),
    \bw^{(s)} - \bw^*
    >
+
\lambda_s^2
\cdot
\left\|
 \nabla_{\bw} F_n( \bw^{(s)} )
\right\|^2
\\
&&
\leq
\left\|
\bw^{(s)}
- \bw^{*}
\right\|^2
-
2 \cdot \lambda_s
\left(
     F_{n}(\bw^{(s)})
    -
    F_{n,lin,\bw^{(s)}}(\bw^*)
    \right)
    \\
    &&
    \quad
    + 2 \cdot \lambda_s \cdot (F_{n}(\bw^{(s)})
    -F_{n}(\bw^{(s+1)}))
\\
&&
=
\left\|
\bw^{(s)}
- \bw^{*}
\right\|^2
-
2 \cdot \lambda_s
\left(
     F_{n}(\bw^{(s+1)})
    -
    F_{n,lin,\bw^{(s)}}(\bw^*)
    \right)
    \\
    &&
    \leq
    \left\|
    \bw^{(s)}
- \bw^{*}
\right\|^2.
    \end{eqnarray*}
    Here we have used (\ref{le11eq2}) in the first inequality
    and (\ref{le12eq3})
    in the second inequality.
    \hfill $\Box$

    \noindent
        {\bf Proof of Lemma \ref{le11}.}
        The proof is an extension of  the proof of
        Lemma 1 in Kohler, Krzy\.zak and Molinero R\"omer (2026).
        
        Set
  \[
  F_{n,lin,\bw}(\bw^*)
  =
\frac{1}{n} \sum_{i=1}^n |Y_i - f_{lin,\bw,\bw^*}(X_i)|^2, 
\]
where $f_{lin,\bw,\bw^*}$ is defined by (\ref{le12eq1}).

In the {\it first step of the proof} we show that (\ref{le11eq3})
implies (\ref{le12eq2}), and consequently the assumptions
of Lemma \ref{le12} are satisfied under the assumptions of Lemma \ref{le11}
provided (\ref{le12eq3}) holds.

To do this, we proceed as in the proof of Lemma 1 in Kohler (2026).
Let $\bw \in \R^{J_n}$ with
$\| \bw^* - \bw\| \leq \|\bw^* - \bw^{(0)} \|$.
Set
\[
H(s)=f_{\bw + s \cdot (\bw^*-\bw)}(x) \quad \mbox{for } s \in [0,1].
\]
Then
\begin{eqnarray*}
  &&
\max\{
\| \bw^* - \bw\|, \| \bw^* - (\bw + s \cdot (\bw^* - \bw))\|
\}
=
\max\{
\| \bw^* - \bw\|, (1-s) \cdot \| \bw^* - \bw \|
\}
\\
&&
\leq
 \|\bw^* - \bw^{(0)} \|
\end{eqnarray*}
for all $s \in [0,1]$, hence we can conclude from (\ref{le11eq3})
\begin{eqnarray*}
  &&
  |    f_{\bw^*}(x)-f_{lin,\bw,\bw^*}(x)|
  \\
  &&
  =
  |  f_{\bw^*}(x)-
    f_{\bw}(x)-
    \sum_{j=1}^{J_n}
    \frac{\partial f_{\bw}(x)}{ \partial \bw^{(j)}}
      \cdot
      ((\bw^*)^{(j)} - \bw^{(j)})
      |
      \\
      &&
      =
      |H(1)-H(0) - \sum_{j=1}^{J_n}
    \frac{\partial f_{\bw}(x)}{ \partial \bw^{(j)}}
      \cdot
       ((\bw^*)^{(j)} - \bw^{(j)})
      |
      \\
      &&
      =
      |
      \int_0^1 H^\prime(s) \, ds
      - \sum_{j=1}^{J_n}
    \frac{\partial f_{\bw}(x)}{ \partial \bw^{(j)}}
      \cdot
       ((\bw^*)^{(j)} - \bw^{(j)})
      |
      \\
      &&
      =
      |
      \int_0^1
      \sum_{j=1}^{J_n}
    \frac{\partial f_{\bw+s \cdot (\bw^*-\bw)}(x)}{ \partial \bw^{(j)}}
      \cdot
       ((\bw^*)^{(j)} - \bw^{(j)})
      \, ds
      - \sum_{j=1}^{J_n}
    \frac{\partial f_{\bw}(x)}{ \partial \bw^{(j)}}
      \cdot
       ((\bw^*)^{(j)} - \bw^{(j)})
      |
      \\
      &&
      =
      |
      \int_0^1
      \sum_{j=1}^{J_n}
      (    \frac{\partial f_{\bw+s \cdot (\bw^*-\bw)}(x)}{ \partial \bw^{(j)}}
      -
    \frac{\partial f_{\bw}(x)}{ \partial \bw^{(j)}})  
      \cdot
       ((\bw^*)^{(j)} - \bw^{(j)})
      \, ds
      |
      \\
      &&
      \leq
      \int_0^1
      \sum_{j=1}^{J_n}
      |    \frac{\partial f_{\bw+s \cdot (\bw^*-\bw)}(x)}{ \partial \bw^{(j)}}
      -
    \frac{\partial f_{\bw}(x)}{ \partial \bw^{(j)}}|  
      \cdot
      | (\bw^*)^{(j)} - \bw^{(j)}|
      \, ds
      \\
      &&
      \leq
        \int_0^1
        \sqrt{
 \sum_{j=1}^{J_n}
      |    \frac{\partial f_{\bw+s \cdot (\bw^*-\bw)}(x)}{ \partial \bw^{(j)}}
      -
    \frac{\partial f_{\bw}(x)}{ \partial \bw^{(j)}}|^2 
        }
        \cdot
        \| \bw^* - \bw\|  \, ds
        \\
        &&
        \leq
        \int_0^1
        \sqrt{
          C_n^2 \cdot \| \bw+s \cdot (\bw^*-\bw) - \bw\|^2}
         \cdot
         \| \bw^* - \bw\|  \, ds
         \\
         &&
         \leq
         C_n \cdot  \| \bw^* - \bw\|^2 \cdot \int_0^1 s \, ds
         =
         \frac{1}{2} \cdot  C_n \cdot  \| \bw^* - \bw\|^2.
  \end{eqnarray*}

In the {\it second step of the proof} we show
that for any
$\bw \in \R^{J_n}$ with
  $
  \| \bw^* - \bw\|
  \leq
  \| \bw^* - \bw^{(0)}\|,
  $
  we have
  \begin{equation}
    \label{ple5eq1}
    \left|
  F_{n,lin,\bw}(\bw^*)
-  F_{n}(\bw^*)
    \right|
    \leq
    5 \cdot \beta_n \cdot C_n \cdot \| \bw^* - \bw \|^2.
  \end{equation}
  Using (\ref{le11eq3})--(\ref{le11eq6}) and the first step of the proof
  we get
  \begin{eqnarray*}
    &&
    \left|
  F_{n,lin,\bw}(\bw^*)
-  F_{n}(\bw^*)
    \right|
    \\
    &&
    =
        \left|
  \frac{1}{n} \sum_{i=1}^n |Y_i - f_{lin,\bw,\bw^*}(X_i)|^2
-  \frac{1}{n} \sum_{i=1}^n |Y_i - f_{\bw^*}(X_i)|^2
    \right|
    \\
    &&
    \leq
    \frac{1}{n} \sum_{i=1}^n
    | 2 \cdot Y_i - 2 \cdot f_{\bw^*}(X_i) +  f_{\bw^*}(X_i) -    f_{lin,\bw,\bw^*}(X_i)|
    \cdot | f_{lin,\bw,\bw^*}(X_i) - f_{\bw^*}(X_i)|
    \\
    &&
    \leq
    \frac{1}{n} \sum_{i=1}^n (2 \cdot \beta_n + 2 \cdot \beta_n + C_n \cdot \|\bw^*-\bw\|^2)
    \cdot C_n \cdot \|\bw^*-\bw\|^2
    \\
    &&
    \leq
    5 \cdot \beta_n \cdot C_n \cdot \|\bw^*-\bw\|^2.
    \end{eqnarray*}

  In the {\it third step of the proof} we show that the assertion holds
  in case that we have for some $s \in \{0,1, \dots, t_n-1\}$
  \begin{equation}
    \label{ple11eq2}
F_n( \bw^{(s+1)}) < F_{n,lin,\bw^{(s)}}(\bw^*).
  \end{equation}

  So assume that (\ref{ple11eq2}) holds for some $s \in \{0,1, \dots, t_n-1\}$.
  By choosing $s$ minimal with this property we can assume
  \[
F_n( \bw^{(t+1)}) \geq F_{n,lin,\bw^{(t)}}(\bw^*)
  \]
  for all $t \in \{0,1, \dots, s-1\}$. By Lemma \ref{le12}
  we can conclude
  \[
\| \bw^{(s)} - \bw^{*}\|
\leq
\| \bw^{(0)} - \bw^{*}\|.
\]
By
using
(\ref{le11eq2})
(which implies $F_n(\bw^{(t_n)}) =  \min_{t=0, \dots, t_n} F_n( \bw^{(t)})$)
and the result of the second step of the proof we get
\begin{eqnarray*}
F_n (\bw^{(t_n)})
  & \leq &
  F_n( \bw^{(s+1)})
  \\
  &
  <&
  F_{n,lin,\bw^{(s)}}(\bw^*)
  \\
  &
  =&
  F_n( \bw^*) +  F_{n,lin,\bw^{(s)}}(\bw^*) - F_n( \bw^*)
  \\
  &
  \leq&
  F_n(\bw^*) +     5 \cdot \beta_n \cdot C_n \cdot \| \bw^* - \bw^{(s)} \|^2
  \\
  &
  \leq&
  F_n(\bw^*) +     5 \cdot \beta_n \cdot C_n \cdot \| \bw^* - \bw^{(0)} \|^2
.
  \end{eqnarray*}

In the {\it fourth step of the proof} we show
the assertion in case that (\ref{ple11eq2}) does not hold for all
$s \in \{0,1, \dots, t_n-1\}$.

In this case we have
\[
F_n( \bw^{(s+1)}) \geq F_{n,lin,\bw^{(s)}}(\bw^*)
\]
for all $s \in \{0,1, \dots, t_n-1\}$, so we can conclude from Lemma \ref{le12}
\[
\| \bw^{(t)} - \bw^{*}\|
\leq
\| \bw^{(0)} - \bw^{*}\|
\]
for all $t \in \{0,1, \dots, t_n\}$.

Using
(\ref{le11eq2})
(which implies $F_n(\bw^{(t_n)}) =  \min_{t=0, \dots, t_n} F_n( \bw^{(t)})$)
and
the result of the second step of the proof we
conclude
\begin{eqnarray*}
  &&
  F_n( \bw^{(t_n)}) - F_n (\bw^*)
  \\
  &&
  \leq
  \frac{1}{\sum_{s=0}^{t_n-1} \lambda_s } \sum_{t=0}^{t_n-1} \lambda_t \cdot F_n( \bw^{(t)}) - F_n (\bw^*)
  \\
  &&
  =
  \frac{1}{\sum_{s=0}^{t_n-1} \lambda_s } \sum_{t=0}^{t_n-1}  \lambda_t
  \cdot \left( F_n( \bw^{(t)}) -
F_{n,lin, \bw^{(t)}} (\bw^*)
  \right)
  \\
  &&
  \quad
  +
    \frac{1}{\sum_{s=0}^{t_n-1} \lambda_s } \sum_{t=0}^{t_n-1}  \lambda_t
  \cdot  \left(
  F_{n,lin, \bw^{(t)}} (\bw^*)
    - F_n (\bw^*)
    \right)
    \\
    &&
    \leq
  \frac{1}{\sum_{s=0}^{t_n-1} \lambda_s } \sum_{t=0}^{t_n-1}  \lambda_t
  \cdot \left( F_n( \bw^{(t)}) -
F_{n,lin, \bw^{(t)}} (\bw^*)
\right)
\\
&&
\quad
  +
      \frac{1}{\sum_{s=0}^{t_n-1} \lambda_s } \sum_{t=0}^{t_n-1}  \lambda_t
  \cdot
  5 \cdot \beta_n \cdot C_n \cdot \| \bw^* - \bw^{(t)} \|^2
    \\
    &&
    \leq
  \frac{1}{\sum_{s=0}^{t_n-1} \lambda_s } \sum_{t=0}^{t_n-1}  \lambda_t
  \cdot \left( F_n( \bw^{(t)}) -
F_{n,lin, \bw^{(t)}} (\bw^*)
  \right)
  +
  5 \cdot \beta_n \cdot C_n \cdot \| \bw^* - \bw^{(0)} \|^2.
  \end{eqnarray*}
Since
    \[
    F_{n,lin,\bw^{(t)}}(\bw^{(t)})=     F_{n}(\bw^{(t)})
    \quad \mbox{and} \quad
    \nabla_\bw F_{n,lin,\bw^{(t)}}(\bw^{(t)})=     \nabla_\bw F_{n}(\bw^{(t)})
    \]
    and $ \bw \mapsto  F_{n,lin,\bw^{(t)}}(\bw)$ is convex
    and differentiable we get furthermore
    \begin{eqnarray*}
      &&
  \frac{1}{\sum_{s=0}^{t_n-1} \lambda_s } \sum_{t=0}^{t_n-1}  \lambda_t
  \cdot \left( F_n( \bw^{(t)}) -
F_{n,lin, \bw^{(t)}} (\bw^*)
  \right)
  \\
  &&
  =
    \frac{1}{\sum_{s=0}^{t_n-1} \lambda_s } \sum_{t=0}^{t_n-1}  \lambda_t
  \cdot   \left( F_{n,lin, \bw^{(t)}}( \bw^{(t)}) -
F_{n,lin, \bw^{(t)}} (\bw^*)
  \right)
  \\
  &&
  \leq
      \frac{1}{\sum_{s=0}^{t_n-1} \lambda_s } \sum_{t=0}^{t_n-1}  \lambda_t
  \cdot
  <
\nabla_\bw F_{n,lin, \bw^{(t)}}
( \bw^{(t)}), \bw^{(t)} -  \bw^{*}
>
\\
&&
=
      \frac{1}{\sum_{s=0}^{t_n-1} \lambda_s } \sum_{t=0}^{t_n-1}  \lambda_t
  \cdot
  <
\nabla_\bw F_{n}
( \bw^{(t)}), \bw^{(t)} - \bw^{*}
>
\\
&&
=
      \frac{1}{2 \cdot \sum_{s=0}^{t_n-1} \lambda_s } \sum_{t=0}^{t_n-1}
2 \cdot
  <
\lambda_t \cdot \nabla_\bw F_{n}
( \bw^{(t)}), \bw^{(t)} -  \bw^{*}
>
\\
&&
=
      \frac{1}{2 \cdot \sum_{s=0}^{t_n-1} \lambda_s } \sum_{t=0}^{t_n-1}
\Bigg(
\| \bw^{(t)} -  \bw^{*} \|^2
-
\|
 \bw^{(t)} -  \bw^{*} -\lambda_t \cdot \nabla_\bw F_{n}
( \bw^{(t)})
 \|^2
 \\
 &&
 \hspace*{3cm}
 +
 \lambda_t^2 \cdot \|  \nabla_\bw F_{n}
( \bw^{(t)})\|^2
 \Bigg)
 \\
 &&
 \leq
      \frac{1}{2 \cdot \sum_{s=0}^{t_n-1} \lambda_s } \sum_{t=0}^{t_n-1}
\left(
\| \bw^{(t)} -  \bw^{*} \|^2
-
\|
 \bw^{(t+1)} - \bw^*
 \|^2
 \right)
 \\
 &&
 \quad
 +
 \frac{\max_{s=0, \dots, t_n-1} \lambda_s}{\sum_{s=0}^{t_n-1} \lambda_s}
\cdot
 \sum_{t=0}^{t_n-1} \frac{\lambda_t}{2} \cdot \|  \nabla_\bw F_{n}
( \bw^{(t)})\|^2
 \\
 &&
 \leq \frac{1}{2 \cdot \sum_{s=0}^{t_n-1} \lambda_s}
 \cdot
 \left(
\| \bw^{(0)} -  \bw^{*} \|^2
-
\|
 \bw^{(t_n)} - \bw^*
 \|^2
 \right)
 \\
 &&
 \quad
+
\frac{\max_{s=0, \dots, t_n-1} \lambda_s}{\sum_{s=0}^{t_n-1} \lambda_s}
\cdot
\sum_{t=0}^{t_n-1}   (F_n(\bw^{(t)}) - F_n(\bw^{(t+1)}))
 \\
 &&
 \leq
 \frac{1}{2 \cdot \sum_{s=0}^{t_n-1} \lambda_s}
 \cdot
 \| \bw^{(0)} -  \bw^{*} \|^2
+
\frac{\max_{s=0, \dots, t_n-1} \lambda_s}{\sum_{s=0}^{t_n-1} \lambda_s}
\cdot F_n(\bw^{(0)})
,
    \end{eqnarray*}
    where the second to last inequality follows from assumption (\ref{le11eq2}).
        \hfill $\Box$

Our next lemma will enable us to show in the proofs of our main results
that the weights during gradient descent do not move too far away from the
initial weights. Together with weight bounds from the initialization
this will enable us to compute upper bounds on the weights
needed for the application of Lemma \ref{le4}.

        \begin{lemma}
          \label{le13}
            Let $F_n$ be defined by (\ref{se2eq7}), set
  \[
  \bw^{(s+1)} = \bw^{(s)} - \lambda_s \cdot \nabla_{\bw} F_n (\bw^{(s)})
  \quad
  \mbox{for } s=0, \dots, t-1
  \]
  and assume
 \begin{equation}
    \label{le13eq1}
    F_n( \bw^{(s+1)}) \leq F_n(\bw^{(s)}) - \frac{\lambda_s}{2}
    \cdot \left\|
\nabla_\bw F_n(\bw^{(s)})
    \right\|^2
    \end{equation}
 for $s=0,1, \dots, t-1$. Then
 \[
 \| \bw^{(t)} - \bw^{(0)} \| \leq
 \sqrt{\sum_{s=1}^t \lambda_{s-1}} \cdot \sqrt{2 \cdot F_n(\bw^{(0)})}.
 \]
        \end{lemma}

        \noindent
            {\bf Proof.} Using the triangle inequality,
            the Cauchy-Schwartz inequality
            and (\ref{le13eq1}) we get
            \begin{eqnarray*}
              &&
              \| \bw^{(t)}-\bw^{(0)} \|
              \\
              &&
              \leq 
              \sum_{s=1}^t
              \| \bw^{(s)}-\bw^{(s-1)} \|
              \\
              &&
              =
              \sum_{s=1}^t
              \sqrt{\lambda_{s-1}} \cdot
              \sqrt{\lambda_{s-1}} \cdot
              \| \nabla_\bw F_n( \bw^{(s-1)}) \|
                \\
                &&
                \leq
                \sqrt{
 \sum_{s=1}^t \lambda_{s-1}
                }
                \cdot
                \sqrt{
 \sum_{s=1}^t \lambda_{s-1} \cdot  \| \nabla_\bw F_n( \bw^{(s-1)}) \|^2
 }
 \\
 &&
 \leq
        \sqrt{
 \sum_{s=1}^t \lambda_{s-1}
                }
                \cdot
                \sqrt{
                   \sum_{s=1}^t
 2 \cdot (F_n(\bw^{(s-1)}) - F_n(\bw^{(s)}))
                }
                \\
                &&
                \leq
        \sqrt{
 \sum_{s=1}^t \lambda_{s-1}
                }
                \cdot
                \sqrt{
 2 \cdot F_n(\bw^{(0)})
                }.
            \end{eqnarray*}

            \hfill $\Box$

        In the sequel we derive auxiliary results which we will use
        to verify the assumptions of Lemma \ref{le11} in the proof
        of Theorem \ref{th3}. Therefore we introduce some additional
        notation to describe our initially pruned deep neural network.
        We describe the connections in our pruned neural network by the set
        \begin{eqnarray*}
          \I &\subseteq& \cup_{l=1}^{L-1} \{ (l,i,k) \, : \, i \in \{1, \dots, r_n\}, k \in \{0, \dots, r_n\} \}
          \\
          &&
        \cup \{ (0,i,k) \; : \; i \in \{1, \dots, r_n\}, k \in \{0, \dots, d\} \}
        \cup
        \{ (L,1,k) \; : \; k \in \{0, \dots, r_n\} \}
        .
        \end{eqnarray*}
        Here $(l,i,k) \in \I$ means that neuron $k$ of level $l$ is connected to
        neuron $i$ of level $l+1$.
        Using this set the pruned neural network is described by
       \begin{equation}
  \label{se5eq1}
f_{\bw}(x)
=
\sum_{k=1}^{r_n} w_{1,k}^{(L)} \cdot f_{\bw,k}^{(L)}(x),  
  \end{equation}
\begin{equation}
  \label{se5eq2}
f_{\bw,i}^{(l)}(x)
=
\sigma \left(
\sum_{k=1, \dots, r_n: (l-1,i,k) \in \I} w_{i,k}^{(l-1)} \cdot f_{\bw,k}^{(l-1)}(x) + w_{i,0}^{(l-1)}
\cdot 1_{\{ (l-1,i,0) \in \I \} }
\right)
\end{equation}
for $ i \in \{1, \dots, r_n\}$ and $l \in \{2, \dots, L\}$, and  
\begin{equation}
  \label{se5eq3}
f_{\bw,i}^{(1)}(x)
=
\sigma \left(
\sum_{k=1}^d w_{i,k}^{(0)} \cdot x^{(i)} + w_{i,0}^{(0)}
\right)  
\end{equation}
for $i \in \{1, \dots, r_n\}$.

Our next lemma concerns inequality (\ref{le11eq3}).

\begin{lemma}
  \label{le6}
  Let $\sigma$ be the logistic squasher.
  Let $a,B_n, \gamma_n^* \geq 1$, $L, r_n,r \in \N$ and define the deep neural network
  $f_\bw:\R^d \rightarrow \R$ with weight vector $\bw$
  by (\ref{se5eq1})--(\ref{se5eq3}) where $\I$ satisfies
  \begin{equation}
    \label{le6eq*}
    | \{ 1 \leq k \leq r_n \, : \, (l,i,k) \in \I \} | \leq r
    \quad \mbox{for all }
    l \in \{1, \dots, L-1\}, i \in \{1, \dots, r_n\}.
  \end{equation}
  Assume that the weight
  vectors $\bw_1$ and $\bw_2$ satisfy
  \[
  \sum_{k=1}^{r_n} |w_{1,k}^{(L)}|^2 \leq (\gamma_n^*)^2
  \quad \mbox{and} \quad
  |w_{i,j}^{(l)}| \leq B_n 
  \]
  for all $l \in \{1, \dots, L-1\}$.
  Then we have for any $x \in [-a,a]^d$
  \[
  \sum_{i,j,l: (l,i,j) \in \I \; or \; l \in \{0,L\}}
  \left|
  \frac{\partial}{\partial w_{i,j}^{(l)}}
  f_{\bw_1}(x)
-
\frac{\partial}{\partial w_{i,j}^{(l)}}
f_{\bw_2}(x)
\right|^2
   \leq \const[c8]  \cdot r_n \cdot B_n^{4L} \cdot (\gamma_n^*)^2  \cdot
  \| \bw_1 - \bw_2 \|^2
  \]
  for some $\const[c8]=\const[c8](d,L,r,a)>0$.
\end{lemma}

\noindent
    {\bf Proof.} The proof is a modification of the proof
    of Lemma 2 in Kohler (2026).
    
    We have
    \begin{eqnarray*}
      &&
  \sum_{i,j,l: (l,i,j) \in \I \; or \; l \in \{0,L\}}
  \left|
  \frac{\partial}{\partial w_{i,j}^{(l)}}
  f_{\bw_1}(x)
-
\frac{\partial}{\partial w_{i,j}^{(l)}}
f_{\bw_2}(x)
\right|^2
\\
&&
\leq
\sum_{k=1}^{r_n}
|f_{\bw_1,k}^{(L)}(x) - f_{\bw_2, k}^{(L)}(x)|^2
\\
&&
\quad
+
\sum_{i,j,l : l<L, (l,i,j) \in \I \; or \; l=0}
\left|
\sum_{k=1}^{r_n}
\left((\bw_1)_{1,k}^{(L)} \cdot \frac{\partial}{\partial \bw_{i,j}^{(l)}}
  f_{\bw_1,k}^{(L)}(x)
  -
(\bw_2)_{1,k}^{(L)} \cdot \frac{\partial}{\partial \bw_{i,j}^{(l)}}
  f_{\bw_2,k}^{(L)}(x)  
  \right)
  \right|^2\\
  &&
\leq
\sum_{k=1}^{r_n}
|f_{\bw_1,k}^{(L)}(x) - f_{\bw_2, k}^{(L)}(x)|^2
\\
&&
\quad
+
r_n \cdot
\sum_{k=1}^{r_n}
\sum_{i,j,l : l<L, (l,i,j) \in \I \; or \; l=0}
\left|
(\bw_1)_{1,k}^{(L)} \cdot \frac{\partial}{\partial \bw_{i,j}^{(l)}}
  f_{\bw_1,k}^{(L)}(x)
  -
(\bw_2)_{1,k}^{(L)} \cdot \frac{\partial}{\partial \bw_{i,j}^{(l)}}
  f_{\bw_2,k}^{(L)}(x)  
  \right|^2\\
  &&
  \leq
\sum_{k=1}^{r_n}
|f_{\bw_1,k}^{(L)}(x) - f_{\bw_2, k}^{(L)}(x)|^2
\\
&&
\quad
+ 2 \cdot r_n \cdot
\sum_{k=1}^{r_n}
\sum_{i,j,l: l<L,  (l,i,j) \in \I \; or \; l=0}
\left|
(\bw_1)_{1,k}^{(L)}
  -
(\bw_2)_{1,k}^{(L)}  
  \right|^2
\cdot |  \frac{\partial}{\partial \bw_{i,j}^{(l)}} f_{\bw_1,k}^{(L)}(x)|^2
  \\
    &&
\quad
+ 2 \cdot r_n \cdot
\sum_{k=1}^{r_n}
\sum_{i,j,l: l<L,  (l,i,j) \in \I \; or \; l=0}
\left| (\bw_2)_{1,k}^{(L)} \right|^2 \cdot
\left|
\frac{\partial}{\partial \bw_{i,j}^{(l)}}
  f_{\bw_1,k}^{(L)}(x)
  -
 \frac{\partial}{\partial \bw_{i,j}^{(l)}}
  f_{\bw_2,k}^{(L)}(x)  
  \right|^2\\
  &&
  =
\sum_{k=1}^{r_n}
|f_{\bw_1,k}^{(L)}(x) - f_{\bw_2, k}^{(L)}(x)|^2
\\
&&
\quad
+ 2 \cdot r_n \cdot
\sum_{k=1}^{r_n}
\left|
(\bw_1)_{1,k}^{(L)}
  -
(\bw_2)_{1,k}^{(L)}  
  \right|^2
\cdot \sum_{i,j,l: l<L,  (l,i,j) \in \I \; or \; l=0}
|  \frac{\partial}{\partial \bw_{i,j}^{(l)}} f_{\bw_1,k}^{(L)}(x)|^2
  \\
    &&
\quad
+ 2 \cdot r_n \cdot
\sum_{k=1}^{r_n}
\left| (\bw_2)_{1,k}^{(L)} \right|^2 \cdot
\sum_{i,j,l: l<L,  (l,i,j) \in \I \; or \; l=0}
\left|
\frac{\partial}{\partial \bw_{i,j}^{(l)}}
  f_{\bw_1,k}^{(L)}(x)
  -
 \frac{\partial}{\partial \bw_{i,j}^{(l)}}
  f_{\bw_2,k}^{(L)}(x)  
  \right|^2.
    \end{eqnarray*}
    By construction of our randomly pruned weight vector we know that
    $f_{\bw,k}^{(L)}$ depends on at most
    \[
    \sum_{l=1}^{L-1} r^{L-l-1} \cdot (r+1) + r^{L-1} \cdot (d+1)
    \leq
    L \cdot (r+d)^L
    \]
    many weights $w_{i,j}^{(l)}$, hence for any fixed
    $k \in \{1, \dots, r_n\}$ there are at most $L \cdot (r+d)^L$
    many terms in the sums
    \[
\sum_{i,j,l: l<L,  (l,i,j) \in \I \; or \; l=0}
  | \frac{\partial}{\partial \bw_{i,j}^{(l)}} f_{\bw_1,k}^{(L)}(x)|^2
    \]
    and
    \[
\sum_{i,j,l: l<L,  (l,i,j) \in \I \; or \; l=0}
\left|
\frac{\partial}{\partial \bw_{i,j}^{(l)}}
  f_{\bw_1,k}^{(L)}(x)
  -
 \frac{\partial}{\partial \bw_{i,j}^{(l)}}
  f_{\bw_2,k}^{(L)}(x)  
  \right|^2
    \]
    which are not zero.

     The chain rule implies
    \begin{eqnarray}
      &&
      \frac{\partial f_{\bw,k}^{(L)}}{\partial w_{i,j}^{(l)}}(x)
        =
        \sum_{s_{l+2}=1, \dots, r_n: (l+1,s_{l+2},i) \in \I}
        \sum_{s_{l+3}=1, \dots,r_n: (l+2, s_{l+3},s_{l+2}) \in \I} \dots
        \nonumber \\
        && \quad
          \sum_{s_{L-1}=1, \dots,r_n: (L-2, s_{L-1},s_{L-2}) \in \I} 
  f_{\bw, j}^{(l)}(x)
  \cdot
\sigma^\prime \left(\sum_{t=0, \dots, r_n: (l,i,t) \in \I} w_{i,t}^{(l)} \cdot f_{\bw, t}^{(l)}(x)  \right)
  \nonumber \\
  && \quad
  \cdot
  w_{s_{l+2},i}^{(l+1)} \cdot
\sigma^\prime \left(\sum_{t=0, \dots, r_n: (l+1,s_{l+2},t) \in \I} w_{s_{l+2},t}^{(l+1)} \cdot f_{\bw, t}^{(l+1)}(x) \right)
  \cdot
  w_{s_{l+3},s_{l+2}}^{(l+2)}
   \nonumber \\
  && \quad
  \cdot
\sigma^\prime \left(\sum_{t=0, \dots, r_n: (l+2, s_{l+3},t) \in \I} w_{s_{l+3},t}^{(l+2)} \cdot f_{\bw,t}^{(l+2)}(x)  \right)
  \cdots
  w_{s_{L-1},s_{L-2}}^{(L-2)}
 \nonumber \\
  && \quad
  \cdot
\sigma^\prime \left(\sum_{t=0, \dots, r_n: (L-2,s_{L-1},t) \in \I} w_{s_{L-1},t}^{(L-2)} \cdot f_{\bw,t}^{(L-2)}(x)  \right)
   \cdot
   w_{k,s_{L-1}}^{(L-1)}
 \nonumber \\
  && \quad
   \cdot
   \sigma^\prime \left(\sum_{t=0, \dots, r_n: (L-1,1,t) \in \I} w_{k,t}^{(L-1)} \cdot
   f_{\bw, t}^{(L-1)}(x)  \right),
     \label{ple6eq1}
      \end{eqnarray}
where we have used the abbreviations
\[
f_{\bw, j}^{(0)}(x)
=
\left\{
\begin{array}{ll}
  x^{(j)} & \mbox{if } j \in \{1,\dots,d\} \\
  1 & \mbox{if } j=0
\end{array}
\right.
\]
and
\[
f_{\bw, 0}^{(l)}(x)=1 \quad (l=1, \dots, L-1).
\]
If $f_{i,1}$, \dots, $f_{i,L}$ are real--valued functions defined on $\R^{J_n}$
where $f_{i,l}$ is bounded in absolute value by $B_{i,l} \geq 1$ and Lipschitz
continuous (w.r.t. $\|\cdot\|_\infty$) with Lipschitz constant $C_{i,l} \geq 1$ $(i=1, \dots, r)$
then
\begin{eqnarray*}
  &&
| \sum_{i=1}^r \prod_{l=1}^L f_{i,l}(\bw_1)
-
\sum_{i=1}^r \prod_{l=1}^L f_{i,l}(\bw_2)|
\\
&&
\leq
\sum_{i=1}^r
\sum_{j=1}^L
\prod_{l=1}^{j-1} |f_{i,l}(\bw_1)|
\cdot
|f_{j,l}(\bw_1) - f_{j,l}(\bw_2)|
\prod_{l=j+1}^L |f_{i,l}(\bw_2)|
\\
&&
\leq
r \cdot L \cdot \max_{i=1, \dots, r}  \prod_{l=1}^L B_{i,l}
\cdot
\max_{i=1, \dots, r} \max_{l=1, \dots, L} C_{i,l}
\cdot \|\bw_1-\bw_2\|_\infty.
\end{eqnarray*}
Using this, (\ref{le6eq*}),
\[
0 \leq \sigma(x) \leq 1 \quad \mbox{and} \quad
|\sigma^\prime(x)|=|\sigma(x) \cdot (1-\sigma(x))| \leq 1
\]
and 
\begin{eqnarray*}
&&|  f_{\bw_1,j}^{(l)}(x)-  f_{\bw_2,j}^{(l)}(x)|\\
&& \leq \const[c9] \cdot a \cdot (\max\{ 2r, d\}+1)^l \cdot B_n^{l-1}  \cdot
\| ((\bw_1)_{\bar{i},\bar{j}}^{(\bar{l})})_{\bar{i},\bar{j},\bar{l}} -
((\bw_2)_{\bar{i},\bar{j}}^{(\bar{l})})_{\bar{i},\bar{j},\bar{l}}\|_\infty\\
&& \leq \const[c9] \cdot a \cdot (\max\{2r,d\}+1)^l \cdot B_n^{l-1}  \cdot
\| ((\bw_1)_{\bar{i},\bar{j}}^{(\bar{l})})_{\bar{i},\bar{j},\bar{l}} -
((\bw_2)_{\bar{i},\bar{j}}^{(\bar{l})})_{\bar{i},\bar{j},\bar{l}}\|
\end{eqnarray*}
(cf., proof of Lemma \ref{le3})
we get 
    \begin{eqnarray*}
      &&
  \sum_{i,j,l: (l,i,j) \in \I \; or \; l \in \{0,L\}}
  \left|
  \frac{\partial}{\partial w_{k,i,j}^{(l)}}
  f_{\bw_1}(x)
-
\frac{\partial}{\partial w_{k,i,j}^{(l)}}
f_{\bw_2}(x)
\right|^2
\\
&&
\leq
\const \cdot a^2 \cdot (2r+d)^{2L} \cdot B_n^{2L}  \cdot \|\bw_1-\bw_2\|^2
\\
&&
\quad
+\const \cdot r_n  \cdot L \cdot
(r+d)^L \cdot r^{2L} \cdot a^2 \cdot B_n^{2L}  \cdot \|\bw_1-\bw_2\|^2
\\
&&
\quad
+\const \cdot r_n \cdot (\gamma_n^*)^2 \cdot L \cdot (r+d)^L
\cdot r^{2L} \cdot (3L)^2 \cdot a^4 \cdot B_n^{4L} 
\cdot (2r+d)^{2L} \cdot \|\bw_1-\bw_2\|^2
\\
&&
\leq \const \cdot r_n \cdot (\gamma_n^*)^2 \cdot L^3 \cdot (2r+d)^{5L}
\cdot B_n^{4L} \cdot a^4 \cdot \|\bw_1-\bw_2\|^2.
\end{eqnarray*}
    \hfill $\Box$

    Our next two lemmata will be needed to verify
    assumption (\ref{le11eq2}) of Lemma \ref{le11}
    in the proof of our main results.

    \begin{lemma}
      \label{le7}
        Let $\sigma$ be the logistic squasher.
        Let $a,\beta_n, B_n, \gamma_n^* \geq 1$, $L,r_n,r \in \N$ and
      define the deep neural network
  $f_\bw:\R^d \rightarrow \R$ with weight vector $\bw$
      by (\ref{se5eq1})--(\ref{se5eq3}) where $\I$ satisfies (\ref{le6eq*}),
      and assume
      \[
      X_i \in [-a,a]^d
            \quad
  \mbox{and} \quad
      |Y_i| \leq \beta_n \quad (i=1, \dots, n)
\]
and
  \[
  \sum_{k=1}^{r_n} |w_{1,k}^{(L)}|^2 \leq (\gamma_n^*)^2 \quad
  \mbox{and} \quad
  |w_{i,j}^{(l)}| \leq B_n 
  \]
  for all $l \in \{1, \dots, L-1\}$. Assume
  \[
r_n \cdot \gamma_n^* \geq \beta_n.
  \]

Then
\[
\| \nabla_\bw F_n(\bw) \| \leq \const[14] \cdot
r_n^{3/2} \cdot (\gamma_n^*)^2 \cdot B_n^L 
  \]
    for some $\const[14]=\const[14](d,L,r,a)>0$.
      \end{lemma}

    \noindent
        {\bf Proof.} The proof is a modification of Lemma 3
        in Kohler (2026).
        
        We have
        \[
|f_\bw(x)| \leq r_n \cdot \gamma_n^* ,
\]
which implies
        \begin{eqnarray*}
          &&
       \| \nabla_\bw F_n(\bw) \|^2    \\
      &&=
  \sum_{i,j,l: (l,i,j) \in \I \; or \; l \in \{0,L\}}
  \left|
  \frac{1}{n}
  \sum_{s=1}^n 2 \cdot (Y_s - f_{\bw}(X_s)) \cdot  
  \frac{\partial}{\partial w_{i,j}^{(l)}}
  f_{\bw}(X_s) \cdot (-1)
  \right|^2
  \\
  && \leq
  16 \cdot (r_n \cdot \gamma_n^*)^2 \cdot
  \sum_{i,j,l: (l,i,j) \in \I \; or \; l \in \{0,L\}}
  \max_{s=1, \dots, n}
  \left|
  \frac{\partial}{\partial w_{i,j}^{(l)}}
  f_{\bw}(X_s) 
  \right|^2
  \\
  && \leq
  16 \cdot (r_n \cdot \gamma_n^*)^2 \cdot \Bigg(
  \sum_{k=1}^{r_n}
  \max_{s=1, \dots, n}
|f_{\bw,k}^{(L)}(X_s)|^2
\\
&&
\quad \quad
+
\sum_{i,j,l: l<L, (l,i,j) \in \I \; or \; l=0 }
  \max_{s=1, \dots, n}
\left|
\sum_{k=1}^{r_n}
\bw_{1,k}^{(L)} \cdot \frac{\partial}{\partial \bw_{i,j}^{(l)}}
  f_{\bw,k}^{(L)}(X_s)
  \right|^2
  \Bigg)
  \\
  && \leq
  16 \cdot (r_n \cdot \gamma_n^*)^2 \cdot \Bigg(
  \sum_{k=1}^{r_n}
  \max_{s=1, \dots, n}
|f_{\bw,k}^{(L)}(X_s)|^2
\\
&&
\quad \quad
+
r_n \cdot
\sum_{k=1}^{r_n}
\sum_{i,j,l: l<L, (l,i,j) \in \I \; or \; l=0 }
  \max_{s=1, \dots, n}
\left|
\bw_{1,k}^{(L)} \cdot \frac{\partial}{\partial \bw_{i,j}^{(l)}}
  f_{\bw,k}^{(L)}(X_s)
  \right|^2
  \Bigg)
  \\
  && =
  16 \cdot (r_n \cdot \gamma_n^*)^2 \cdot \Bigg(
  \sum_{k=1}^{r_n}
  \max_{s=1, \dots, n}
|f_{\bw,k}^{(L)}(X_s)|^2
\\
&&
\quad \quad
+
r_n \cdot
\sum_{k=1}^{r_n}
|\bw_{1,k}^{(L)} |^2 \cdot 
\sum_{i,j,l: l<L, (l,i,j) \in \I \; or \; l=0 }
  \max_{s=1, \dots, n}
\left|
\frac{\partial}{\partial \bw_{i,j}^{(l)}}
  f_{\bw,k}^{(L)}(X_s)
  \right|^2
  \Bigg)
  \\
  &&
  \leq
  16 \cdot (r_n \cdot \gamma_n^*)^2 \cdot \Bigg(
  r_n \cdot 1 \\
  && \quad +   r_n \cdot (\gamma_n^*)^2
 \cdot   L  \cdot (r+d)^L   \cdot
  \max_{s=1, \dots, n}
  \max_{k=1, \dots, r_n}
  \max_{i,j,l: \atop l<L, (l,i,j) \in \I} \left|
\frac{\partial}{\partial \bw_{i,j}^{(l)}}
  f_{\bw,k}^{(L)}(X_s)
  \right|^2
  \Bigg)
  \\
  &&
  \leq
   16 \cdot (r_n \cdot \gamma_n^*)^2 \cdot \left(
  r_n \cdot 1 + r_n \cdot (\gamma_n^*)^2 \cdot  L \cdot (r+d)^L 
  \cdot  r^{2L}
  \cdot a^2 \cdot B_n^{2L} 
  \right),
\end{eqnarray*}        
        where the last inequality follows from (\ref{ple6eq1}), the assumptions
        on the weights and the bounds on the logistic squasher
        mentioned in the proof of Lemma \ref{le6}.
        \hfill $\Box$

    \begin{lemma}
      \label{le9}
      Let $\sigma: \R \rightarrow \R$ be the logistic squasher,
      let $L, r_n, r \in \N$,
      let $f_\bw$
      be defined by (\ref{se5eq1})--(\ref{se5eq3}) where $\I$ satisfies (\ref{le6eq*}), and let
      $F_n$ be defined by (\ref{se2eq7}).
     Let $a \geq 1$,
     $\gamma_n^* \geq 1$, $B_n \geq 1$,  and assume
     $X_i \in [-a,a]^d$ and $|Y_i| \leq \beta_n$  $(i=1, \dots, n)$,
     \begin{equation}
     	\label{le9eq1}
     	\max\{ \sum_{k=1}^{r_n} |(\bw_1)_{1,k}^{(L)}|^2,  \sum_{k=1}^{r_n} |(\bw_2)_{1,k}^{(L)}|^2\} \leq (\gamma_n^*)^2, 
     \end{equation}
     \begin{equation}
     	\label{le9eq2}
     	\max\{|(\bw_1)_{i,j}^{(l)}|,|(\bw_2)_{i,j}^{(l)}|\} \leq B_n
     	\quad
     	\mbox{for } l=1, \dots, L-1
     \end{equation}
     and
     \begin{equation}
       \label{le9eq3}
r_n \cdot \gamma_n^* \geq \beta_n.
  \end{equation}
     Then we have
     \begin{eqnarray*}
     	&&
     	\| (\nabla_\bw F_n)(\bw_1) - (\nabla_\bw F_n)(\bw_2) \|
     	\leq
     	\const[c15]   \cdot r_n^{3/2} \cdot B_n^{2L} \cdot (\gamma_n^*)^2 \cdot \|\bw_1-\bw_2\|
     \end{eqnarray*}
     for some $\const[c15]=\const[c15](d,L,r,a) >0$.
\end{lemma}

\noindent
    {\bf Proof.}    The proof is a modification
    of the proof of Lemma 5 in Kohler (2026).
    
We have
    \begin{eqnarray*}
      &&
      \| \nabla_\bw F_n (\bw_1) -  \nabla_\bw F_n (\bw_2) \|^2
      \\
      &&
      =
      \sum_{i,j,l: (l,i,j) \in \I \; or \; l \in \{0,L\} }
      \Bigg(
      \frac{2}{n}
      \sum_{s=1}^n
      (f_{\bw_1} (X_s) - Y_s) 
      \cdot
      \frac{\partial f_{\bw_1}}{\partial w_{i,j}^{(l)}}(X_s)
      \\
      &&
      \hspace*{6cm}
      -
      \frac{2}{n}
      \sum_{s=1}^n
      (f_{\bw_2} (X_s)- Y_s) 
      \cdot
      \frac{\partial f_{\bw_2}}{\partial w_{i,j}^{(l)}}(X_s)
      \Bigg)^2
      \\
      &&
      \leq
      16 \cdot
      \sum_{i,j,l: (l,i,j) \in \I \; or \; l \in \{0,L\} }
      \max_{s=1,\dots,n}
      \left(
      \frac{\partial f_{\bw_1}}{\partial w_{i,j}^{(l)}}(X_s)
      \right)^2 
    \cdot
      \frac{1}{n}
      \sum_{s=1}^n
      (f_{\bw_2} (X_s) - f_{\bw_1} (X_s))^2
      \\
      &&
      \quad
      +
      16  \cdot
      \frac{1}{n}
      \sum_{s=1}^n
      (Y_s-f_{\bw_2} (X_s) )^2 
      \cdot
      \sum_{i,j,l: (l,i,j) \in \I \; or \; l \in \{0,L\} }
      \left(
      \frac{\partial f_{\bw_1}}{\partial w_{i,j}^{(l)}}(X_s)
      -
           \frac{\partial f_{\bw_2}}{\partial w_{i,j}^{(l)}}(X_s)
           \right)^2 
      .
    \end{eqnarray*}
    From the proof of Lemma \ref{le7} we can conclude
    \begin{eqnarray*}
      &&
      \sum_{i,j,l: (l,i,j) \in \I \; or \; l \in \{0,L\} }
      \max_{s=1,\dots,n}
      \left(
      \frac{\partial f_{\bw_1}}{\partial w_{i,j}^{(l)}}(X_s)
      \right)^2
      \\
         &&
         \leq
      \sum_{i,j,l: (l,i,j) \in \I \; or \; l \in \{0,L\} }
      \max_{s=1,\dots,n}
      \left(
          \sum_{k=1}^{r_n}
\bw_{1,k}^{(L)} \cdot \frac{\partial}{\partial \bw_{i,j}^{(l)}}
f_{\bw,k}^{(L)}(X_s)
            \right)^2
      \\ 
           &&
         \leq
         \sum_{k=1}^{r_n}
|\bw_{1,k}^{(L)}|^2 \cdot 
      \sum_{i,j,l: (l,i,j) \in \I \; or \; l \in \{0,L\} }
      \sum_{k=1}^{r_n}
      \max_{s=1,\dots,n}
      \left(
          \frac{\partial}{\partial \bw_{i,j}^{(l)}}
f_{\bw,k}^{(L)}(X_s)
            \right)^2
      \\      
      &&
      \leq
      \const \cdot (\gamma_n^*)^2  \cdot L \cdot
 (r+d)^L \cdot r_n \cdot r^{2L}  \cdot B_n^{2L} \cdot a^2,      
    \end{eqnarray*}
    from the proof of Lemma \ref{le6} we know
    \begin{eqnarray*}
      &&
      \frac{1}{n}
      \sum_{s=1}^n
      (f_{\bw_2} (X_s) - f_{\bw_1} (X_s))^2
      \\
      &&
      \leq
      \frac{2}{n}
      \sum_{s=1}^n
      \sum_{k=1}^{r_n} r_n \cdot |(\bw_2)_{1,k}^{(L)}|^2
      \cdot
      (f_{\bw_2,k}^{(L)} (X_s) - f_{\bw_1,k}^{(L)} (X_s))^2\\
      &&
      \quad
      +
        \frac{2}{n}
      \sum_{s=1}^n
      \sum_{k=1}^{r_n} r_n \cdot |(\bw_2)_{1,k}^{(L)} - (\bw_1)_{1,k}^{(L)}|^2
        \cdot
      |f_{\bw_1,k}^{(L)} (X_s)|^2
      \\
      &&
      \leq
      \const \cdot r_n \cdot
      (\gamma_n^*)^2 \cdot (2r+d)^{2L}
      \cdot B_n^{2L} \cdot a^2  \cdot \|\bw_1-\bw_2\|^2
      \end{eqnarray*}
    and (\ref{le9eq1}) and (\ref{le9eq3}) imply
    \begin{eqnarray*}
      &&
            \frac{1}{n}
      \sum_{s=1}^n
      (Y_s-f_{\bw_2} (X_s) )^2 
      \leq
      4 \cdot r_n^2 \cdot (\gamma_n^*)^2.
    \end{eqnarray*}
    And by Lemma \ref{le6} we can conclude for any $s \in \{1, \dots, n\}$
    \[
\sum_{i,j,l: (l,i,j) \in \I \; or \; l \in \{0,L\} }
      \left(
      \frac{\partial f_{\bw_1}}{\partial w_{i,j}^{(l)}}(X_s)
      -
           \frac{\partial f_{\bw_2}}{\partial w_{i,j}^{(l)}}(X_s)
           \right)^2
           \leq
           \const  \cdot r_n \cdot B_n^{4L} \cdot (\gamma_n^*)^2  \cdot
  \| \bw_1 - \bw_2 \|^2.
  \]
  Summarizing the above results we get the assertion.
    \hfill $\Box$

\subsection{Proof of Theorem \ref{th3}}
\label{se5sub7}

In the {\it first step of the proof} we show that if
\begin{equation}
  \label{apth3eq1}
|Y_i| \leq \const \cdot \log n \quad (i=1, \dots, n)
\end{equation}
holds, then we have
\begin{equation}
  \label{apth3eq2}
F_n( \bw^{(s+1)}) \leq F_n(\bw^{(s)}) - \frac{\lambda_s}{2}
    \cdot \left\|
\nabla_\bw F_n(\bw^{(s)})
    \right\|^2
\end{equation}
for $s=0,1, \dots, t_n-1$.

To show this we have to demonstrate that if Algorithm \ref{alg2}
stops its inner loop because of $\lambda_s \leq \lambda_{min}$,
then (\ref{apth3eq2}) holds.
Therefore it suffices to show for all
$t \in \{0,1, \dots, t_n-1\}$
that if (\ref{apth3eq2}) holds for all $s \in \{0,1, \dots, t-1\}$
and if $\lambda \leq \lambda_{min}$, then
\[
  \bar{\bw} = \bw^{(t)} - \lambda \cdot \nabla_{\bw} F_n (\bw^{(t)})
\]
satisfies
\[
F_n( \bar{\bw} ) \leq F_n(\bw^{(t)}) - \frac{\lambda}{2}
    \cdot \left\|
\nabla_\bw F_n(\bw^{(t)})
    \right\|^2.
\]
So assume that  (\ref{apth3eq2}) holds for all $s \in \{0,1, \dots, t-1\}$.
Then we can conclude from Lemma \ref{le13}
and (\ref{apth3eq1})
that we have
\[
\| \bw^{(s)} - \bw^{(0)} \| \leq \const \cdot \log n
\quad (s=1, \dots, t).
\]

Using the bounds on $\bw^{(0)}$ we can conclude from
Lemma \ref{le7}
(which we apply with $\gamma_n^* = B_n = \const \cdot \log n$)
that we have
\[
\left\|
\nabla_\bw F_n(\bw^{(t)})
    \right\| \leq \const \cdot (\log n)^{L+2} \cdot r_n^{3/2},
\]
hence for any $\tau \in [0,1]$ we have
\begin{eqnarray*}
  \| \bw^{(t)} + \tau \cdot (\bar{\bw} - \bw^{(t)}) - \bw^{(0)} \|
  &\leq&
    \| \bw^{(t)}  - \bw^{(0)} \|
    +  \|\bar{\bw} - \bw^{(t)}\|
    \\
    &\leq&
    \| \bw^{(t)}  - \bw^{(0)} \|
    +  \lambda \cdot \left\|
\nabla_\bw F_n(\bw^{(t)})
\right\|
\\
    &\leq&
    2 \cdot \log n
    +  \lambda_{min} \cdot \left\|
\nabla_\bw F_n(\bw^{(t)})
\right\|
\\
&\leq&
\const \cdot \log n.
  \end{eqnarray*}
Application of 
Lemma \ref{le9}
(which we apply with $\gamma_n^* = B_n = \const \cdot \log n$)
yields
\[
     	\| (\nabla_\bw F_n)(\bw^{(t)} + \tau \cdot (\bar{\bw} - \bw^{(t)})) - (\nabla_\bw F_n)(\bw^{(t)})) \|
     	\leq
        \const \cdot r_n^{3/2} \cdot (\log n)^{2L+2} \cdot
        \| \tau \cdot (\bar{\bw} - \bw^{(t)}) )\|.
\]
For $\tau \in [0,1]$ set
   \[
H(\tau)=F_n(\bw^{(t)} + \tau \cdot (\bar{\bw}-\bw^{(t)})).
\]

    Then the fundamental theorem of calculus, the chain rule and
    the Cauchy-Schwarz inequality imply
    \begin{eqnarray*}
      &&F_n (\bar{\bw})-F_n (\bw^{(t)}) = H(1)-H(0) 
      = \int_0^1 H^\prime (\tau) \, d\tau \\
      &&= \int_0^1 (\nabla_\bw F_n)(\bw^{(t)} + \tau \cdot (\bar{\bw}-\bw^{(t)}))
      \cdot (\bar{\bw}-\bw^{(t)}) \, d\tau\\
      &&=\int_0^1 \left(
      (\nabla_\bw F_n)(\bw^{(t)} + \tau \cdot (\bar{\bw}-\bw^{(t)}))
      -
      (\nabla_\bw F_n)(\bw^{(t)})
      \right)
      \cdot (\bar{\bw}-\bw^{(t)}) \, d\tau\\
      && \quad
      +
     \int_0^1 
      (\nabla_\bw F_n)(\bw^{(t)})
      \cdot (\bar{\bw}-\bw^{(t)}) \, d\tau
      \\
      && \leq \int_0^1 \big\|
      (\nabla_\bw F_n)(\bw^{(t)} + \tau \cdot (\bar{\bw}-\bw^{(t)}))
      -
      (\nabla_\bw F_n)(\bw^{(t)})
      \big\|
      \cdot \| \bar{\bw}-\bw^{(t)}\| \, d\tau\\
      && \quad
      +
     \int_0^1 
      (\nabla_\bw F_n)(\bw^{(t)})
      \cdot (\bar{\bw}-\bw^{(t)}) \, d\tau
      \\
      &&
      \leq
      \int_0^1         \const[c102] \cdot r_n^{3/2} \cdot (\log n)^{2L+2} \cdot
 \|\tau \cdot (\bar{\bw}-\bw^{(t)}) \| \cdot
      \|\bar{\bw}-\bw^{(t)}\| \, d\tau\\
      && \quad +
      (\nabla_\bw F_n)(\bw^{(t)}) \cdot
      (\bar{\bw}-\bw^{(t)})
      \\
      &&
      =
      \frac{        \const[c102] \cdot r_n^{3/2} \cdot (\log n)^{2L+2} 
}{2} \cdot \|\bar{\bw}-\bw^{(t)}\|^2
      +
      (\nabla_\bw F)(\bw^{(t)}) \cdot
      (\bar{\bw}-\bw^{(t)}).
    \end{eqnarray*}
    Using $\lambda \leq \lambda_{min}$ and the definitions of
    $\lambda_{min}$ and $\bar{\bw}$ we get
    \begin{eqnarray*}
      &&    F_n(\bar{\bw})-F_n(\bw^{(t)})
      \\
    &&\leq
    \frac{ \const[c102] \cdot r_n^{3/2} \cdot (\log n)^{2L+2}}{2} \cdot \lambda^2 \cdot \|  (\nabla_\bw F)(\bw^{(t)})\|^2
    - \lambda \|  (\nabla_\bw F)(\bw^{(t)})\|^2\\
    &&\leq
    \left(\frac{ \const[c102] \cdot r_n^{3/2} \cdot (\log n)^{2L+2}}{2} \cdot \lambda_{min} - 1\right)
    \cdot \lambda
    \cdot \|  (\nabla_\bw F)(\bw^{(t)})\|^2
    \\
    &&
    \leq 
- \frac{1}{2} \cdot \lambda
      \cdot \|(\nabla_\bw F)(\bw^{(t)})\|^2,
    \end{eqnarray*}
    which follows from the fact  that we can assume w.l.o.g. that $n$ is so large
    that
    \[
    \const[c102] \cdot r_n^{3/2} \cdot (\log n)^{2L+2} \cdot \lambda_{min}
    \leq 1
    \]
    holds.
    
In the {\it second step of the proof} we show
\begin{equation}
  \label{apth1eq9}
  \| \bw^{(s)} - \bw^{(0)} \| \leq \const \cdot \log n
  \quad (s=0,1, \dots, t_n)
\end{equation}
provided (\ref{apth3eq1}) holds. This follows directly
from Lemma \ref{le13} and the first step of the proof.

In the {\it third step of the proof} we decompose the $L_2$ error
of the estimate in a sum of several terms.

Let $E_n$ be the event that $|Y_i|^2 \leq \const \cdot \log n$ holds
for $i=1, \dots, n$ and set
\[
m_{\beta_n}(x)=\EXP\{ T_{\beta_n} Y | X=x \}.
\]
Then we have
\begin{eqnarray*}
&&
\int | m_n(x)-m(x)|^2 \PROB_X (dx)
\\
&&
=
\left(
\EXP \left\{ |m_n(X)-Y|^2 | \D_n \right\}
-
\EXP \{ |m(X)-Y|^2\}
\right)
\cdot 1_{E_n}
\\
&&
\quad
+
\int | m_n(x)-m(x)|^2 \PROB_X (dx)
\cdot 1_{E_n^c}
\\
&&
=
\Big[
\EXP \left\{ |m_n(X)-Y|^2 | \D_n \right\}
-
\EXP \{ |m(X)-Y|^2\}
\\
&&
\hspace*{2cm}
- \left(
\EXP \left\{ |m_n(X)-T_{\beta_n} Y|^2 | \D_n \right\}
-
\EXP \{ |m_{\beta_n}(X)- T_{\beta_n} Y|^2\}
\right)
\Big] \cdot 1_{E_n}
\\
&&
\quad +
\Big[
\EXP \left\{ |m_n(X)-T_{\beta_n} Y|^2| \D_n \right\}
-
\EXP \{ |m_{\beta_n}(X)- T_{\beta_n} Y|^2\}
\\
&&
\hspace*{2cm}
-
2 \cdot \frac{1}{n} \sum_{i=1}^n
\left(
|m_n(X_i)-T_{\beta_n} Y_i|^2
-
|m_{\beta_n}(X_i)- T_{\beta_n} Y_i|^2
\right)
\Big] \cdot 1_{E_n}
\\
&&
\quad
+\Big[
2 \cdot \frac{1}{n} \sum_{i=1}^n
|m_n(X_i)-T_{\beta_n} Y_i|^2
-
2 \cdot \frac{1}{n} \sum_{i=1}^n
|m_{\beta_n}(X_i)- T_{\beta_n} Y_i|^2
\\
&&
\hspace*{2cm}
- \left(
2 \cdot \frac{1}{n} \sum_{i=1}^n
|m_n(X_i)-Y_i|^2
-
2 \cdot \frac{1}{n} \sum_{i=1}^n
|m(X_i)- Y_i|^2
\right)
\Big] \cdot 1_{E_n}
\\
&&
\quad
+
\Big[
2 \cdot \frac{1}{n} \sum_{i=1}^n
|m_n(X_i)-Y_i|^2
-
2 \cdot \frac{1}{n} \sum_{i=1}^n
|m(X_i)- Y_i|^2
\Big] \cdot 1_{E_n}
\\
&&
\quad
+
\int | m_n(x)-m(x)|^2 \PROB_X (dx)
\cdot 1_{E_n^c}
\\
&&
=: \sum_{j=1}^5 T_{j,n}.
\end{eqnarray*}

In the {\it fourth step of the proof} we show
\[
\EXP \{T_{j,n}\} \leq \const \cdot \frac{\log n}{n} \quad
\mbox{for } j \in \{1,3\}.
\]
This follows from the proof of Lemma 1 in Bauer and Kohler (2019).

\noindent
In the {\it fifth step of the proof} we show
\[
\EXP \{T_{5,n}\} \leq \const \cdot \frac{(\log n)^2}{n^2}.
\]
W.l.o.g we can assume that $n$ is so large that
$\|m\|_\infty\leq \beta_n$ holds, and hence
\begin{align*}
    \int |m_n(x)-m(x)|^2 \PROB_X (dx)
    \leq \int |2\cdot \beta_n|^2 \PROB_X (dx)
    \leq
\const \cdot (\log
n)^2.
\end{align*}
Consequently the assertion follows from
\begin{eqnarray*}
\PROB(E_n^c)
&\leq&
\PROB\{ \max_{i=1, \dots, n} Y_i^2 > \beta_n
\}
 \leq
n \cdot\PROB\{  Y^2 > \beta_n
\}
\\
&=&
n\cdot \PROB\{\exp(\const[c2] \cdot Y^2)>\exp(\const[c2] \cdot\beta_n)\}
 \leq 
n \cdot
\frac{\EXP\{ \exp(\const[c2] \cdot Y^2)\}}{\exp( \const[c2] \cdot \const[c1] \cdot \log n)}\\
 &\leq &
\frac{\const}{n^2},
\end{eqnarray*}
where the last inequality followed from assumption $\const[c1] \cdot \const[c2]   \geq 2$
and assumption $(A3)$.

\noindent
In the {\it sixth step of the proof} we show
\[
\EXP \{T_{2,n}\} \leq
\const \cdot (\log n)^{4 \cdot d \cdot L+3} \cdot n^{-\frac{2p}{2p+d}}.
\]
Let $\W_n$ be the set of all weight vectors
$\bw=(w_{i,j}^{(l)})_{i,j,l} $ which satisfy
\[
| w_{1,k}^{(L)}| \leq \const \cdot \log n \quad (k=1, \dots, K_n),
\]
\[
|w_{i,j}^{(l)}| \leq \const \cdot \log n \quad (l=1, \dots, L-1),
\]
\[
|w_{i,j}^{(0)}| \leq \const \cdot (\log n) \cdot n^{\frac{1}{2p+d}}
\]
and
\[
|\{j \in \{1, \dots, r_n\} \, : \, w_{i,j}^{(l)} \neq 0 \}| \leq r
\quad
\mbox{for all}
\quad
l \in \{1, \dots, L-1\}, i \in \{1, \dots, r_n\}.
\]
By the second step of the proof and our assumptions on the initialization
of the weights we know $\bw^{(t_n)} \in \W_n$ on $E_n$, hence we have
\[
m_n= f
\quad \mbox{for some} \quad
f \in \F_n = \left\{ T_{\beta_n} f_\bw \quad : \quad \bw \in \W_n \right\}
\]
on $E_n$.
This implies for any $u > 0$
\begin{eqnarray*}
&&
\PROB \{ T_{2,n} > u \}
\\
&&
\leq
\PROB \Bigg\{
\exists f \in \F_n :
\EXP \left(
\left|
\frac{f(X)}{\beta_n} - \frac{T_{\beta_n}Y}{\beta_n}
\right|^2
\right)
-
\EXP \left(
\left|
\frac{m_{\beta_n}(X)}{\beta_n} - \frac{T_{\beta_n}Y}{\beta_n}
\right|^2
\right)
\\
&&\hspace*{3cm}-
\frac{1}{n} \sum_{i=1}^n
\left(
\left|
\frac{f(X_i)}{\beta_n} - \frac{T_{\beta_n}Y_i}{\beta_n}
\right|^2
-
\left|
\frac{m_{\beta_n}(X_i)}{\beta_n} - \frac{T_{\beta_n}Y_i}{\beta_n}
\right|^2
\right)
\Bigg\}
\\
&&\hspace*{2cm}
> \frac{1}{2} \cdot
\left(
\frac{u}{\beta_n^2}
+
\EXP \left(
\left|
\frac{f(X)}{\beta_n} - \frac{T_{\beta_n}Y}{\beta_n}
\right|^2
\right)
-
\EXP \left(
\left|
\frac{m_{\beta_n}(X)}{\beta_n} - \frac{T_{\beta_n}Y}{\beta_n}
\right|^2
\right)
\right) \Bigg\}.
\end{eqnarray*}
By Lemma \ref{le4} we get
\begin{eqnarray*}
&&
\Nu_1 \left(
\delta , \left\{
\frac{1}{\beta_n} \cdot f : f \in \F_n
\right\}
, x_1^n
\right)
\leq
\Nu_1 \left(
\delta \cdot \beta_n , \F_n
, x_1^n
\right)
\\
&&
\leq
3 \left(\frac{12e \cdot \beta_n}{\delta \cdot \beta_n }\right)^
             {
            2 \cdot   (4ed)^d \cdot 2^{L \cdot d} \cdot (r+k)^{3 \cdot d \cdot L} 
               \cdot
\alpha^d \cdot ((\log n) \cdot n^{\frac{1}{2p+d}})^d \cdot (\const \cdot \log n)^{(L-1) \cdot d}
\cdot
 \left( \frac{r_n \cdot \const \cdot \log n}{\delta \cdot \beta_n} \right)^{\frac{d}{k}}  
          +2
}
\end{eqnarray*}
for any $k \in \N$. Choosing
\[
k= \lceil (\const[c5]+2) \cdot \log n \rceil
\]
we conclude from (\ref{th1eq2}) for $\delta>1/n^2$
\begin{eqnarray*}
  &&
\Nu_1 \left(
\delta , \left\{
\frac{1}{\beta_n} \cdot f : f \in \F_n
\right\}
, x_1^n
\right)
\\
&&
\leq
3 \left(12e \cdot n\right)^
             {
               \const \cdot
               (\log n)^{3 \cdot d \cdot L} \cdot (\log n)^d 
               \cdot n^{\frac{d}{2p+d}} \cdot (\log n)^{(L-1) \cdot d}
\cdot
 \left( n^{\const[c5]+2}  \right)^{\frac{d}{(\const[c5]+2) \cdot \log n }}  
}
             \\
             &&
             \leq
             3 \left(12e \cdot n\right)^
             {
               \const \cdot
               (\log n)^{4 \cdot d \cdot L} \cdot n^{\frac{d}{2p+d}}}. 
\end{eqnarray*}

This together with Theorem 11.4 in Gy\"orfi et al. (2002) leads for $u
\geq 1/n$ to
\[
\PROB\{T_{2,n}>u\}
\leq
14 \cdot
\const[c32] \cdot n^{ \const[c33]   \cdot
               (\log n)^{4 \cdot d \cdot L } \cdot n^{\frac{d}{2p+d}}}
\cdot
\exp \left(
- \frac{n}{5136 \cdot \beta_n^2} \cdot u
\right).
\]
For $\epsilon_n \geq 1/n$ we can conclude
\begin{eqnarray*}
\EXP \{ T_{2,n} \}
& \leq &
\epsilon_n + \int_{\epsilon_n}^\infty \PROB\{ T_{2,n}>u \} \, du
\\
& \leq &
\epsilon_n
+
14 \cdot
\const[c32] \cdot n^{
\const[c33]   \cdot
               (\log n)^{4 \cdot d \cdot L  } \cdot n^{\frac{d}{2p+d}}}
\cdot
\exp \left(
- \frac{n}{5136 \cdot \beta_n^2} \cdot \epsilon_n
\right)
\cdot
\frac{5136 \cdot \beta_n^2}{n}.
\end{eqnarray*}
Setting
\[
\epsilon_n = \frac{5136 \cdot \beta_n^2}{n}
\cdot
\const[c33]   \cdot
               (\log n)^{4 \cdot d \cdot L  } \cdot n^{\frac{d}{2p+d}}
\cdot \log n
=
\frac{5136 \cdot \beta_n^2}{n}
\cdot
\log
\left(
n^{
\const[c33]   \cdot
               (\log n)^{4 \cdot d \cdot L  } \cdot n^{\frac{d}{2p+d}}
}
\right)
\]
yields the assertion of the sixth step of the proof.

In the {\it seventh step of the proof} we show
\[
\EXP \{T_{4,n}\} \leq
\const \cdot  n^{-\frac{2p}{2p+d}}.
\]
Since $|T_{\beta_n} z - y| \leq |z-y|$ for $|y| \leq \beta_n$ we have
\begin{eqnarray*}
T_{4,n}
&\leq&
\Big[
2 \cdot \frac{1}{n} \sum_{i=1}^n
|m_n(X_i)-Y_i|^2
-
2 \cdot \frac{1}{n} \sum_{i=1}^n
|m(X_i)- Y_i|^2
\Big] \cdot 1_{E_n}
\\
&\leq&
\Big[
2 \cdot \frac{1}{n} \sum_{i=1}^n
|f_{\bw^{(t_n)}} (X_i)-Y_i|^2
-
2 \cdot \frac{1}{n} \sum_{i=1}^n
|m(X_i)- Y_i|^2
\Big] \cdot 1_{E_n}
\\
& \leq &
\Big[
  2 \cdot F_n( \bw^{(t_n)})
  -
2 \cdot \frac{1}{n} \sum_{i=1}^n
|m(X_i)- Y_i|^2
\Big] \cdot 1_{E_n}.
\end{eqnarray*}
We will use Lemma \ref{le11} to derive an upper bound on $ F_n( \bw^{(t_n)})$
on $E_n$.

By the first step of the proof we know that (\ref{le11eq2}) holds.

Next we choose  $w_1^*, \dots, w_{r_n}^*$ as in Lemma \ref{le2} and define $\bw^*$
by
\[
(\bw^*)_{i,j}^{(l)}=
\begin{cases}
  w_j^* & \mbox{ if } l=L,i=1, \\
  (\bw^{(0)})_{i,j}^{(l)} &  \mbox{ if } l \in \{1, \dots, L-1\}
  \end{cases}
\]
if $E_n$ holds and set $\bw^*=\bw^{(0)}$ else.
Then we have
\[
\| \bw^* - \bw^{(0)} \|^2
\leq
\sum_{k=1}^{r_n} | w_{k}^*|^2 \leq \frac{\const \cdot n^{4L \cdot r^{L-1} \cdot (d+1)+4}}{r_n}.
\]
By Lemma \ref{le6} we know that (\ref{le11eq3}) holds with
$C_n=\const \cdot (\log n)^{2L+1} \cdot \sqrt{r_n}$, where we apply Lemma \ref{le6} with
$B_n=\gamma_n^*=\const \cdot \log n$, which is possible because
of the result of the second step of the proof and our initial choice
of $\bw^{(0)}$. Here we have used that for any $\bw$ with
$\| \bw^* - \bw \| \leq \| \bw^* - \bw^{(0)}\|$ we have
\[
\sum_{t=1}^{r_n} | w_{1,t}^{(L)}|^2
\leq 2 \cdot \| \bw^* - \bw^{(0)} \|^2 + 2
\cdot \sum_{t=1}^{r_n} | (\bw^*)_{1,t}^{(L)}|^2
\leq \const \cdot (\log n)^2.
\]
This implies also that
(\ref{le11eq6}) holds, and
since $f_{\bw^*}$ approximates $m$ in supremum norm well on $E_n$
and $m$ is bounded, we can assume w.l.o.g. that $n$ is so large
that (\ref{le11eq5}) holds.

So the assumptions of Lemma \ref{le11} are satisfied on $E_n$, and we can
conclude
\begin{eqnarray*}
  &&
\Big[
  2 \cdot F_n( \bw^{(t_n)})
  -
2 \cdot \frac{1}{n} \sum_{i=1}^n
|m(X_i)- Y_i|^2
\Big] \cdot 1_{E_n}\\
&&
\leq
\Big[
  2 \cdot
    F_n( \bw^*)
    +
    2 \cdot \left(
5 \cdot \beta_n \cdot C_n + \frac{1}{2 \cdot 1}
    \right)
    \cdot
    \| \bw^* - \bw^{(0)}\|^2
    +
    2 \cdot F_n( \bw^{(0)}) \cdot
    \frac{\frac{1}{n}}{1}
    \\
    &&
    \quad
  -
2 \cdot \frac{1}{n} \sum_{i=1}^n
|m(X_i)- Y_i|^2
\Big] \cdot 1_{E_n}\\
&&    
\end{eqnarray*}
This yields
\begin{eqnarray*}
T_{4,n}
&\leq&
2 \cdot \frac{1}{n} \sum_{i=1}^n
|f_{\bw^*}(X_i)-Y_i|^2
-
2 \cdot \frac{1}{n} \sum_{i=1}^n
|m(X_i)- Y_i|^2
+
\const \cdot \frac{\log n}{n}
\\
&&
+
2 \cdot \frac{1}{n} \sum_{i=1}^n
|m(X_i)- Y_i|^2
\cdot 1_{E_n^c}.
\end{eqnarray*}
Because of
\begin{eqnarray*}
&&  \EXP \left\{
\frac{1}{n} \sum_{i=1}^n
|f_{\bw^*}(X_i)-Y_i|^2
-
\frac{1}{n} \sum_{i=1}^n
|m(X_i)- Y_i|^2
\right\}
\\
&&=
\EXP \int |f_{\bw^*}(x) - m(x)|^2 \PROB_X (dx)
\\
&&\leq
\const \cdot n^{-\frac{2p}{2p+d}} + \const  \cdot
n^{\const} \cdot \exp \left( -  (\log n)^2 \right)
\\
&&
\leq
\const \cdot n^{-\frac{2p}{2p+d}} 
  \end{eqnarray*}
and
\begin{eqnarray*}
  &&
  \EXP \left\{
\frac{1}{n} \sum_{i=1}^n
|m(X_i)- Y_i|^2
\cdot 1_{E_n^c}
\right\}
\\
&&
\leq
\frac{1}{n} \sum_{i=1}^n
\sqrt{
 \EXP \left\{ |m(X_i)- Y_i|^4 \right\}
}
\cdot
\sqrt{\PROB(E_n^c)}
\leq
\frac{\const}{n}
  \end{eqnarray*}
this implies the assertion.
\hfill $\Box$

\subsection{Proof of Theorem \ref{th4}}
\label{se5sub8}

In the proof we will need the following auxiliary results.

    \begin{lemma}
      \label{le16}
      Let $p,C, \alpha >0$ and assume that $(A4)$ holds.
      Choose the parameters of the estimate as in Theorem \ref{th4}.
      Let $\bw^{(0)}$ be the randomly initialized and pruned
      weight vector of the network. 
      Then outside of an event whose probability is bounded from above by
      \[
n^{\const} \cdot \exp \left( -  (\log n)^2 \right) 
      \]
      there exists (random) $w_1^*, \dots, w_{r_n}^* \in \R$  such that
      \[
      \sup_{x \in [-\alpha, \alpha]^d}
      \left|
\sum_{k=1}^{r_n} w_{k}^* \cdot f^{(L)}_{\bw^{(0)},k }(x) - m(x)
  \right|
  \leq \const  \cdot n^{-\frac{p}{2p+d^*}}
  \]
  and
  \[
\sum_{k=1}^{r_n} | w_{k}^*|^2 \leq 
\frac{\const \cdot n^{4L \cdot r^{L-1} \cdot (d+1)+4}}{r_n}
  \]
      \end{lemma}

\noindent
{\bf Proof.}
The result follows by a straightforward modification of the proof of Lemma \ref{le2},
therefore we only provide the outline of the proof.

   W.l.o.g. we assume that $n$ is so large that $A_n \geq B$ holds.
            Set
            $\tilde{K}_n = \lceil n^{\frac{1}{2p+d^*}} \rceil$ and  
            \[
            I_n =\lceil \frac{r_n}{n^{4L \cdot r^{L-1} \cdot (d+1) +1}} \rceil.
            \]
            In the {\it first step of the proof} we show that there exists
            a neural network
            \begin{equation}
              \label{ple16eq1}
            f_{\bw}(x) = \sum_{j=1}^{I_n \cdot r \cdot \tilde{K}_n^d}
            w_j \cdot f_{\bw_j,j,1}^{(L)}(x)
            \end{equation}
            where $f_{\bw_j,j,1}^{(L)}$ are recursively defined by
            (\ref{le1eq2}) and (\ref{le1eq3}) such that
            \[
      \sup_{x \in [-\alpha, \alpha]^d}
      \left|
      \sum_{j=1}^{I_n \cdot r \cdot \tilde{K}_n^d}
      w_j \cdot f_{\bw_j,j,1}^{(L)}(x)
      - m(x)
  \right|
  \leq \const  \cdot n^{-\frac{p}{2p+d^*}}
  \quad \mbox{and} \quad
  \sum_{j=1}^{I_n \cdot r \cdot \tilde{K}_n^d}
            w_j^2 \leq \const \cdot \frac{n^3}{I_n}.
            \]
            To show this we proceed as in the first step of the proof of Lemma \ref{le2}.
            Observe that due to the different choice of $\tilde{K}_n$ here all inner weights
            are within the range of the uniform distributions used in Theorem \ref{th4} during the initialization
            of the weights.

Set
\[
\epsilon_n= \frac{1}{n^3}
\quad \mbox{and} \quad
N_n
=
2^{
L \cdot r^{L-1} \cdot (d+1)
} \cdot \left\lceil
n^{4 \cdot L \cdot r^{L-1} \cdot (d+1) }
\cdot (\log n)^2
\right\rceil.
\]
Next we observe
            that outside of an event whose probability is bounded from
            above by
            \[
  I_n \cdot r \cdot \tilde{K}_n^d \cdot
              \exp\left(
              - N_n \cdot \frac{1}{2} \cdot \left( \frac{\epsilon_n}{A_n} \right)^{L \cdot r^{L-1} \cdot
                (d+1)}
               \right)
            \]
            there exists (random) indices $j_1, \dots, j_{I_n \cdot r \cdot \tilde{K}_n^d}
            \in \{1, \dots, r_n\}$ such that
            \[
            \sup_{x \in [-\alpha,\alpha]^d}
            |
              f_{\bw_k,k,1}^{(L)}(x)
              -
f_{\bw^{(0)},j_k}^{(L)}(x)              
            |
            \leq
            (d+1) \cdot \alpha \cdot (2r+1)^{L-1} \cdot B^{L-1} \cdot \epsilon_n
            \quad (k=1, \dots, I_n \cdot r \cdot \tilde{K}_n^d).
            \]
This follows from the second step of the proof of Lemma \ref{le2}.

From this we get the assertion as in the third step of the proof of Lemma \ref{le2}.

\hfill $\Box$

In order to prove an extension of our bound on the covering number
of over-parametrized deep neural network to the case of manifold data,
we need the following auxiliary result.

\begin{lemma}
  \label{le14}
  Let $\M$ be a $d^*$-dimensional Lipschitz-manifold.
  Let $h \in (0,1]$ and let $\P$ be a partition of $\R^d$ into cubes with side length $h$.
      Then 
      \[
      | \{ C \in \P \, : \, C \cap \M \neq \emptyset \} |
      \leq
      \const[c1le0] \cdot
      \left( \frac{1}{h} \right)^{d^*},
      \]
      where
      $\const[c1le0]{} = s \cdot (2 \cdot C_{\psi,2} \cdot \sqrt{d^*} +4)^{d}$.\\
  \end{lemma}

\noindent
    {\bf Proof.} See Lemma 11 in Kohler, Krzy\.zak and Molinero R\"omer (2026).
    \hfill $\Box$

    Using the above lemma we are able to show the following modification
    of the bound on the covering number in Lemma \ref{le4}.

\begin{lemma}
  \label{le15}
  Let $L, r_n, r \in \N$, let $\beta, A, B, C \geq 1$.
  Let $\sigma$ be the logistic squasher and let $\F$ be the set
  of all neural networks $f_\bw$ defined by (\ref{se2eq1})--(\ref{se2eq3})
  where the weight vector $\bw$ satisfies
  (\ref{le4eq1})--(\ref{le4eq3}).
  Let $0<\epsilon \leq 1$, let $k \in \N$,
  let $\M$ be a $d^*$-dimensional Lipschitz-manifold, set
  \[
  \delta= 
  \frac{1}{2} \cdot
  \left(\frac{k!}{2 d^k} \cdot \frac{\epsilon}{(r+k)^{3 \cdot k \cdot L} \cdot \left( 2^{k-1} \right)^{L} \cdot C \cdot B^{(L-1) \cdot k} \cdot A^k}\right)^{1/k}
  \]
  and let $\M_{\delta}$ be the $\delta$--neighborhood
  of $\M$ introduced in Definition \ref{se3de2}. Let $\alpha \geq 1$ be such that
  $\M_{\delta} \subseteq [-\alpha,\alpha]^d$ holds.
  Then we have for any $x_1, \dots, x_n  \in \M_{\delta}$
  \[
  \Nu( \epsilon, T_\beta \F, x_1^n)
  \leq
3 \cdot  \left(\frac{12e \cdot \beta}{\epsilon}\right)^
             {
             \const[le15c1b]
             \cdot
             k^{3 \cdot L \cdot d^* + d}
           \cdot A^{d^*} \cdot B^{(L-1) \cdot d^*}
\cdot
 \left( \frac{C}{\epsilon} \right)^{\frac{d^*}{k}}  
}
  \]
  for some constant
  $\const[le15c1b]
=\const[le15c1b](d,r,L,C_{\psi,2},s)>0.$
  \end{lemma}

  \noindent
  {\bf Proof.}
Let
$\G$ be the set of all polynomials of degree less than or equal to $k-1$,
let $\Pi$ be a partition of $[-\alpha,\alpha]^d$ into cubes of sidelength
\[
h=\left(\frac{k!}{2 d^k} \cdot \frac{\epsilon}{(r+k)^{3 \cdot k \cdot L} \cdot \left( 2^{k-1} \right)^{L} \cdot C \cdot B^{(L-1) \cdot k} \cdot A^k}\right)^{1/k}
\leq 1,
\]
and let
$\G\circ \Pi$ be the class of all functions
whose restriction on each cube in $\Pi$ lies in $\G$
and which are zero outside of $[-\alpha,\alpha]^d$.
The proof of Lemma \ref{le4} implies
\[
	\Nu_p \left(
	\epsilon, \{T_{\beta} f \, : \, f \in \F \}, x_1^n
	\right)
	\leq
	\Nu_p  \left(
	\frac{\epsilon}{2}, T_\beta \G\circ \Pi, x_1^n
	\right),
\]
so it suffices to derive an upper bound on the covering number on the right-hand side
of the above inequality.
Set
$\Pi^*=\{I \in \Pi \, : \, I \cap x_1^n \neq \emptyset \}$. Obviously we have
\[
	\Nu_p  \left(
	\frac{\epsilon}{2}, T_\beta  \G\circ \Pi,x_1^n
	\right)
	=
		\Nu_p  \left(
	\frac{\epsilon}{2}, T_\beta \G\circ \Pi^*, x_1^n
	\right),
\]
where we set the functions in $T_\beta \G\circ \Pi^*$ to zero on all cubes in $I \in \Pi \setminus \Pi^*$.
Since $x_1, \dots, x_n \in \M_\delta$, there exists $z_1, \dots, z_n \in \M$ such that
$\|x_i-z_i\| \leq 2 \cdot \delta$
holds for $i=1, \dots, n$.
By Lemma \ref{le14} we know 
\[
|\{I \in \Pi \, : \, I \cap z_1^n \neq \emptyset \}|
\leq
|\{I \in \Pi \, : \, I \cap \M \neq \emptyset \}|
\leq
s \cdot (2 \cdot C_{\psi,2} \cdot \sqrt{d^*} +4)^{d}
\cdot
   \left( \frac{1}{h} \right)^{d^*}.
\]
If $z_i \in I \in \Pi$, then 
$\|x_i-z_i\| \leq 2 \cdot \delta$
implies that $x_i$ is contained in one of the $3^d-1$ cubes of side length $h = 2 \cdot \delta$ which are adjacent to $I$.
Hence
\[
| \Pi^*| 
\leq 3^d \cdot 
\{I \in \Pi \, : \, I \cap z_1^n \neq \emptyset \}|
\leq
3^d \cdot s \cdot (2 \cdot C_{\psi,2} \cdot \sqrt{d^*} +4)^{d}
\cdot
   \left( \frac{1}{h} \right)^{d^*}.
\]
Consequently, 
$\G\circ \Pi^*$ is a linear vector space of dimension less than or equal to
\begin{eqnarray*}
&&
k^d \cdot 
3^d \cdot s \cdot (2 \cdot C_{\psi,2} \cdot \sqrt{d^*} +4)^{d}
\cdot
   \left( \frac{1}{h} \right)^{d^*}
   \\
   &&
   \leq
   k^d \cdot 
3^d \cdot s \cdot (2 \cdot C_{\psi,2} \cdot \sqrt{d^*} +4)^{d}
\cdot
d^{d^*}  \cdot  (r+k)^{3 \cdot L \cdot d^*} \cdot
2^{L \cdot d^*}
\\
&&
\hspace*{3cm}
\cdot \left( \frac{C }{\epsilon} \right)^{d^*/k} \cdot B^{(L-1) \cdot d^*} \cdot A^{d^*} .
\end{eqnarray*}
As in the proof of Lemma \ref{le4} we conclude
from Theorem 9.4 and Theorem 9.5 in Gy\"orfi et al. (2002),
   \begin{align*}
    	\mathcal{N}(\frac{\epsilon}{2}, T_\beta \mathcal{G} \circ \Pi^*, x_1^n)
    	\leq 
	 \cdot  \left(\frac{12e \cdot \beta}{\epsilon}\right)^
             {
             \const[le15c1]
             \cdot
             k^{3 \cdot L \cdot d^* + d}
           \cdot A^{d^*} \cdot B^{(L-1) \cdot d^*}
\cdot
 \left( \frac{C}{\epsilon} \right)^{\frac{d^*}{k}}  
}
 \end{align*}
     \hfill $\Box$

\noindent
{\bf Proof of Theorem \ref{th4}}. We mimick the proof of Theorem \ref{th3}.

In the {\it first step of the proof} we show that if (\ref{apth3eq1})
holds, then we have
(\ref{apth3eq2})
for $s=0,1, \dots, t_n-1$.
 This follows as in the first step of the proof of Theorem \ref{th3}.

In the {\it second step of the proof} we show (\ref{apth1eq9})
provided (\ref{apth3eq1}) holds. This follows directly
from Lemma \ref{le13} and the first step of the proof.

In the {\it third step of the proof} we decompose the $L_2$ error
of the estimate in a sum of several terms.

Let $E_n$ be the event that $|Y_i|^2 \leq \const \cdot \log n$ holds
for $i=1, \dots, n$ and set
\[
m_{\beta_n}(x)=\EXP\{ T_{\beta_n} Y | X=x \}.
\]
Then we have
\begin{eqnarray*}
&&
\int | m_n(x)-m(x)|^2 \PROB_X (dx)
\\
&&
=
\left(
\EXP \left\{ |m_n(X)-Y|^2 | \D_n \right\}
-
\EXP \{ |m(X)-Y|^2\}
\right)
\cdot 1_{E_n}
\\
&&
\quad
+
\int | m_n(x)-m(x)|^2 \PROB_X (dx)
\cdot 1_{E_n^c}
\\
&&
=
\Big[
\EXP \left\{ |m_n(X)-Y|^2 | \D_n \right\}
-
\EXP \{ |m(X)-Y|^2\}
\\
&&
\hspace*{2cm}
- \left(
\EXP \left\{ |m_n(X)-T_{\beta_n} Y|^2 | \D_n \right\}
-
\EXP \{ |m_{\beta_n}(X)- T_{\beta_n} Y|^2\}
\right)
\Big] \cdot 1_{E_n}
\\
&&
\quad +
\Big[
\EXP \left\{ |m_n(X)-T_{\beta_n} Y|^2| \D_n \right\}
-
\EXP \{ |m_{\beta_n}(X)- T_{\beta_n} Y|^2\}
\\
&&
\hspace*{2cm}
-
2 \cdot \frac{1}{n} \sum_{i=1}^n
\left(
|m_n(X_i)-T_{\beta_n} Y_i|^2
-
|m_{\beta_n}(X_i)- T_{\beta_n} Y_i|^2
\right)
\Big] \cdot 1_{E_n}
\\
&&
\quad
+\Big[
2 \cdot \frac{1}{n} \sum_{i=1}^n
|m_n(X_i)-T_{\beta_n} Y_i|^2
-
2 \cdot \frac{1}{n} \sum_{i=1}^n
|m_{\beta_n}(X_i)- T_{\beta_n} Y_i|^2
\\
&&
\hspace*{2cm}
- \left(
2 \cdot \frac{1}{n} \sum_{i=1}^n
|m_n(X_i)-Y_i|^2
-
2 \cdot \frac{1}{n} \sum_{i=1}^n
|m(X_i)- Y_i|^2
\right)
\Big] \cdot 1_{E_n}
\\
&&
\quad
+
\Big[
2 \cdot \frac{1}{n} \sum_{i=1}^n
|m_n(X_i)-Y_i|^2
-
2 \cdot \frac{1}{n} \sum_{i=1}^n
|m(X_i)- Y_i|^2
\Big] \cdot 1_{E_n}
\\
&&
\quad
+
\int | m_n(x)-m(x)|^2 \PROB_X (dx)
\cdot 1_{E_n^c}
\\
&&
=: \sum_{j=1}^5 T_{j,n}.
\end{eqnarray*}

In the {\it fourth step of the proof} we show
\[
\EXP \{T_{j,n}\} \leq \const \cdot \frac{\log n}{n} \quad
\mbox{for } j \in \{1,3\}.
\]
This follows from the proof of Lemma 1 in Bauer and Kohler (2019).

\noindent
In the {\it fifth step of the proof} we show
\[
\EXP \{T_{5,n}\} \leq \const \cdot \frac{(\log n)^2}{n^2}.
\]
This mimics the fifth step of the proof of Theorem \ref{th3}.

\noindent
In the {\it sixth step of the proof} we show
\[
\EXP \{T_{2,n}\} \leq
\const \cdot (\log n)^{4 \cdot d^* \cdot L + d+3} \cdot n^{-\frac{2p}{2p+d^*}}.
\]
Let $\W_n$ be the set of all weight vectors
$\bw=(w_{i,j}^{(l)})_{i,j,l}$ which satisfy
\[
| w_{1,k}^{(L)}| \leq \const \cdot \log n =:C_n/r_n \quad (k=1, \dots, K_n),
\]
\[
|w_{i,j}^{(l)}| \leq \const \cdot \log n =:B_n \quad (l=1, \dots, L-1),
\]
\[
|w_{i,j}^{(0)}| \leq \const \cdot (\log n) \cdot n^{\frac{1}{2p+d^*}} =:A_n
\]
and
\[
|\{ j \in \{1, \dots, r_n\} \, : \, w_{i,j}^{(l)} \neq 0 \}| \leq r
\quad
\mbox{for all}
\quad
l \in \{1, \dots, L-1\}, i \in \{1, \dots, r_n\}.
\]
By the second step of the proof and our assumptions on the initialization
of the weights we know $\bw^{(t_n)} \in \W_n$ on $E_n$, hence we have
\[
m_n= f
\quad \mbox{for some} \quad
f \in \F_n = \left\{ T_{\beta_n} f_\bw \quad : \quad \bw \in \W_n \right\}
\]
on $E_n$.
This implies for any $u > 0$
\begin{eqnarray*}
&&
\PROB \{ T_{2,n} > u \}
\\
&&
\leq
\PROB \Bigg\{
\exists f \in \F_n :
\EXP \left(
\left|
\frac{f(X)}{\beta_n} - \frac{T_{\beta_n}Y}{\beta_n}
\right|^2
\right)
-
\EXP \left(
\left|
\frac{m_{\beta_n}(X)}{\beta_n} - \frac{T_{\beta_n}Y}{\beta_n}
\right|^2
\right)
\\
&&\hspace*{3cm}-
\frac{1}{n} \sum_{i=1}^n
\left(
\left|
\frac{f(X_i)}{\beta_n} - \frac{T_{\beta_n}Y_i}{\beta_n}
\right|^2
-
\left|
\frac{m_{\beta_n}(X_i)}{\beta_n} - \frac{T_{\beta_n}Y_i}{\beta_n}
\right|^2
\right)
\Bigg\}
\\
&&\hspace*{2cm}
> \frac{1}{2} \cdot
\left(
\frac{u}{\beta_n^2}
+
\EXP \left(
\left|
\frac{f(X)}{\beta_n} - \frac{T_{\beta_n}Y}{\beta_n}
\right|^2
\right)
-
\EXP \left(
\left|
\frac{m_{\beta_n}(X)}{\beta_n} - \frac{T_{\beta_n}Y}{\beta_n}
\right|^2
\right)
\right) \Bigg\}.
\end{eqnarray*}
Next we want to apply Lemma \ref{le15}
with
$x_1^n \in supp(\PROB_X)$,
$k= \lceil (\const[c5]+2) \cdot \log n \rceil$ and $\epsilon \geq 1/n$.

For this we have to show
that $x_1^n \in supp(\PROB_X)$ are contained in $\M_{\delta_n}$, where
\[
  \delta_n= 
   \frac{1}{2} \cdot
  \left(
  \frac{k!}{2 d^k} \cdot \frac{\epsilon}{(r+k)^{3 \cdot k \cdot L} \cdot \left( 2^{k-1} \right)^{L} \cdot C_n \cdot B_n^{(L-1) \cdot k} \cdot A_n^k }
  \right)^{1/k}.
   \]
  Since
  \[
  \delta_n \geq \const \cdot (\log n)^{-3L-L} \cdot n^{-\frac{1}{2p+d^*}},
  \]
  this follows from the assumption on the distribution of $X$ in Theorem \ref{th4}.
Application of  Lemma \ref{le15} yields 
\begin{eqnarray*}
&&
\Nu_1 \left(
\delta , \left\{
\frac{1}{\beta_n} \cdot f : f \in \F_n
\right\}
, x_1^n
\right)
\leq
\Nu_1 \left(
\delta \cdot \beta_n , \F_n
, x_1^n
\right)
\\
&&
\leq
3 \cdot  \left(\frac{12e \cdot \beta_n}{\delta \cdot \beta_n}\right)^
             {
             \const[le15c1b]
             \cdot
             k^{3 \cdot L \cdot d^*+d}
           \cdot ( \const \cdot (\log n) \cdot n^{\frac{1}{2p+d^*}})^{d^*} \cdot (\const \cdot \log n)^{(L-1) \cdot d^*}
\cdot
 \left( \frac{r_n \cdot \const \cdot \log n}{\delta \cdot \beta_n} \right)^{\frac{d^*}{k}}}
\end{eqnarray*}
for any $k \in \N$. Using
\[
k= \lceil (\const[c5]+2) \cdot \log n \rceil
\]
we conclude from (\ref{th1eq2}) for $\delta>1/n^2$
\begin{eqnarray*}
  &&
\Nu_1 \left(
\delta , \left\{
\frac{1}{\beta_n} \cdot f : f \in \F_n
\right\}
, x_1^n
\right)
\\
&&
\leq
3 \left(12e \cdot n\right)^
             {
               \const \cdot
               (\log n)^{3 \cdot d^* \cdot L+d} \cdot (\log n)^{d^*} 
               \cdot n^{\frac{d^*}{2p+d^*}} \cdot (\log n)^{(L-1) \cdot d^*}
\cdot
 \left( n^{\const[c5]+2}  \right)^{\frac{d^*}{(\const[c5]+2) \cdot \log n }}  
}
             \\
             &&
             \leq
             3 \left(12e \cdot n\right)^
             {
               \const \cdot
               (\log n)^{4 \cdot d^* \cdot L +d} \cdot n^{\frac{d^*}{2p+d^*}}}. 
\end{eqnarray*}

This together with Theorem 11.4 in Gy\"orfi et al. (2002) leads for $u
\geq 1/n$ to
\[
\PROB\{T_{2,n}>u\}
\leq
14 \cdot
\const[c32b] \cdot n^{ \const[c33b]   \cdot
               (\log n)^{4 \cdot d^* \cdot L + d } \cdot n^{\frac{d^*}{2p+d^*}}}
\cdot
\exp \left(
- \frac{n}{5136 \cdot \beta_n^2} \cdot u
\right).
\]
For $\epsilon_n \geq 1/n$ we can conclude
\begin{eqnarray*}
\EXP \{ T_{2,n} \}
& \leq &
\epsilon_n + \int_{\epsilon_n}^\infty \PROB\{ T_{2,n}>u \} \, du
\\
& \leq &
\epsilon_n
+
14 \cdot
\const[c32b] \cdot n^{
\const[c33b]   \cdot
               (\log n)^{4 \cdot d^* \cdot L +d  } \cdot n^{\frac{d^*}{2p+d^*}}}
\cdot
\exp \left(
- \frac{n}{5136 \cdot \beta_n^2} \cdot \epsilon_n
\right)
\cdot
\frac{5136 \cdot \beta_n^2}{n}.
\end{eqnarray*}
Setting
\[
\epsilon_n = \frac{5136 \cdot \beta_n^2}{n}
\cdot
\const[c33b]   \cdot
               (\log n)^{4 \cdot d^* \cdot L +d  } \cdot n^{\frac{d^*}{2p+d^*}}
\cdot \log n
=
\frac{5136 \cdot \beta_n^2}{n}
\cdot
\log
\left(
n^{
\const[c33b]   \cdot
               (\log n)^{4 \cdot d^* \cdot L +d  } \cdot n^{\frac{d^*}{2p+d^*}}
}
\right)
\]
yields the assertion of the sixth step of the proof.

In the {\it seventh step of the proof} we show
\[
\EXP \{T_{4,n}\} \leq
\const \cdot  n^{-\frac{2p}{2p+d^*}}.
\]
Since $|T_{\beta_n} z - y| \leq |z-y|$ for $|y| \leq \beta_n$ we have
\begin{eqnarray*}
T_{4,n}
&\leq&
\Big[
2 \cdot \frac{1}{n} \sum_{i=1}^n
|m_n(X_i)-Y_i|^2
-
2 \cdot \frac{1}{n} \sum_{i=1}^n
|m(X_i)- Y_i|^2
\Big] \cdot 1_{E_n}
\\
&\leq&
\Big[
2 \cdot \frac{1}{n} \sum_{i=1}^n
|f_{\bw^{(t_n)}} (X_i)-Y_i|^2
-
2 \cdot \frac{1}{n} \sum_{i=1}^n
|m(X_i)- Y_i|^2
\Big] \cdot 1_{E_n}
\\
& \leq &
\Big[
  2 \cdot F_n( \bw^{(t_n)})
  -
2 \cdot \frac{1}{n} \sum_{i=1}^n
|m(X_i)- Y_i|^2
\Big] \cdot 1_{E_n}.
\end{eqnarray*}
We will use Lemma \ref{le11} to derive an upper bound on $ F_n( \bw^{(t_n)})$
on $E_n$.

By the first step of the proof we know that (\ref{le11eq2}) holds.

Next we choose  $w_1^*, \dots, w_{r_n}^*$ as in Lemma \ref{le16} and define $\bw^*$
by
\[
(\bw^*)_{i,j}^{(l)}=
\begin{cases}
  w_j^* & \mbox{ if } l=L,i=1 \\
  (\bw^{(0)})_{i,j}^{(l)} &  \mbox{ if } l \in \{1, \dots, L-1\}
  \end{cases}
\]
if $E_n$ holds and set $\bw^*=\bw^{(0)}$ else.
Then we have
\[
\| \bw^* - \bw^{(0)} \|^2
\leq
\sum_{k=1}^{r_n} | w_{k}^*|^2 \leq \frac{\const \cdot n^{4L \cdot r^{L-1} \cdot (d+1)+4}}{r_n}.
\]

By Lemma \ref{le6} we know that (\ref{le11eq3}) holds with
$C_n=\const \cdot \sqrt{r_n} \cdot (\log n)^{2L+1}$, where we apply Lemma \ref{le6} with
$B_n=\gamma_n^*=\const \cdot \log n$, which is possible because
of the result of the second step of the proof and our initial choice
of $\bw^{(0)}$. Here we have used that for any $\bw$ with
$\| \bw^* - \bw \| \leq \| \bw^* - \bw^{(0)}\|$ we have
\[
\sum_{t=1}^{r_n} | w_{1,t}^{(L)}|^2
\leq 2 \cdot \| \bw^* - \bw^{(0)} \|^2 + 2
\cdot \sum_{t=1}^{r_n} | (\bw^*)_{1,t}^{(L)}|^2
\leq \const \cdot (\log n)^2.
\]

This implies that also (\ref{le11eq6}) holds, and
since $f_{\bw^*}$ approximates $m$ in supremum norm well on $E_n$
and $m$ is bounded, we can assume w.l.o.g. that $n$ is so large
that (\ref{le11eq5}) holds.

So the assumptions of Lemma \ref{le11} are satisfied, and we can
conclude
\begin{eqnarray*}
  &&
\Big[
  2 \cdot F_n( \bw^{(t_n)})
  -
2 \cdot \frac{1}{n} \sum_{i=1}^n
|m(X_i)- Y_i|^2
\Big] \cdot 1_{E_n}\\
&&
\leq
\Big[
  2 \cdot
    F_n( \bw^*)
    +
    2 \cdot \left(
5 \cdot \beta_n \cdot C_n + \frac{1}{2 \cdot 1}
    \right)
    \cdot
    \| \bw^* - \bw^{(0)}\|^2
    +
    2 \cdot F_n( \bw^{(0)}) \cdot
    \frac{\frac{1}{n}}{1}
    \\
    &&
    \quad
  -
2 \cdot \frac{1}{n} \sum_{i=1}^n
|m(X_i)- Y_i|^2
\Big] \cdot 1_{E_n}\\
&&    
\end{eqnarray*}
This yields
\begin{eqnarray*}
T_{4,n}
&\leq&
2 \cdot \frac{1}{n} \sum_{i=1}^n
|f_{\bw^*}(X_i)-Y_i|^2
-
2 \cdot \frac{1}{n} \sum_{i=1}^n
|m(X_i)- Y_i|^2
+
\const \cdot \frac{\log n}{n}
\\
&&
+
2 \cdot \frac{1}{n} \sum_{i=1}^n
|m(X_i)- Y_i|^2
\cdot 1_{E_n^c}.
\end{eqnarray*}
Because of
\begin{eqnarray*}
&&  \EXP \left\{
\frac{1}{n} \sum_{i=1}^n
|f_{\bw^*}(X_i)-Y_i|^2
-
\frac{1}{n} \sum_{i=1}^n
|m(X_i)- Y_i|^2
\right\}
\\
&&=
\EXP \int |f_{\bw^*}(x) - m(x)|^2 \PROB_X (dx)
\\
&&\leq
\const \cdot n^{-\frac{2p}{2p+d^*}} + \const  \cdot
n^{\const} \cdot \exp \left( -  (\log n)^2 \right)
\\
&&
\leq
\const \cdot n^{-\frac{2p}{2p+d^*}} 
  \end{eqnarray*}
and
\begin{eqnarray*}
  &&
  \EXP \left\{
\frac{1}{n} \sum_{i=1}^n
|m(X_i)- Y_i|^2
\cdot 1_{E_n^c}
\right\}
\\
&&
\leq
\frac{1}{n} \sum_{i=1}^n
\sqrt{
 \EXP \left\{ |m(X_i)- Y_i|^4 \right\}
}
\cdot
\sqrt{\PROB(E_n^c)}
\leq
\frac{\const}{n}
  \end{eqnarray*}
this implies the assertion.
\hfill $\Box$




\end{document}